\documentclass{amsart}%

\usepackage{amsfonts}
\usepackage{amsmath}
\usepackage{amsthm,amstext,amsfonts,bm,amssymb,amsopn}
\usepackage{amssymb}
\usepackage{xfrac}
\usepackage{graphicx}%
\usepackage{color}%
\usepackage{refcount}
\allowdisplaybreaks
\usepackage{xypic}
\usepackage{tikz-cd}
\usepackage{environ,enumitem}
\usepackage[margin=.5cm]{caption}

\setenumerate{label={\normalfont\arabic*)}}

\makeatletter

\usepackage{tikz}

\makeatletter
\newcommand{\xleftrightarrow}[2][]{\ext@arrow 3359\leftrightarrowfill@{#1}{#2}}
\makeatother

\newcommand{\xdasharrow}[2][->]{
\tikz[baseline=-\the\dimexpr\fontdimen22\textfont2\relax]{
\node[anchor=south,font=\scriptsize, inner ysep=1.5pt,outer xsep=5pt](x){#2};
\draw[shorten <=3.4pt,shorten >=3.4pt,dashed,#1](x.south west)--(x.south east);
}
}

\newcommand{\dashmap}{\xdasharrow{\ \ \ \ }{}}\providecommand{\U}[1]{\protect\rule{.1in}{.1in}}
\providecommand{\comment}[1]{}

\makeatletter
\newtheorem*{rep@theorem}{\rep@title}
\newcommand{\newreptheorem}[2]{%
\newenvironment{rep#1}[1]{%
 \def\rep@title{#2 \ref{##1}}%
 \begin{rep@theorem}}%
 {\end{rep@theorem}}}
\makeatother

\newtheorem*{namedtheorem}{\theoremname}
\newcommand{\theoremname}{testing}

\newtheorem{theorem}{Theorem}[section]
\newtheorem*{fact*}{Fact}

\newtheorem{claim}[theorem]{Claim}

\newtheorem{question}[theorem]{Question}
\newtheorem{corollary}[theorem]{Corollary}

\newtheorem{lemma}[theorem]{Lemma}

\newtheorem{proposition}[theorem]{Proposition}
\newtheorem{remark}[theorem]{Remark}
\newtheorem{fact}[theorem]{Fact}

\theoremstyle{definition}
\newtheorem{example}[theorem]{Example}

\newtheorem{definition}[theorem]{Definition}

\newreptheorem{theorem}{Theorem}
\newreptheorem{lemma}{Lemma}
\newreptheorem{proposition}{Proposition}

\NewEnviron{thmquote}{\vspace{1ex}\par
\refstepcounter{equation}%
\hfill\parbox{\dimexpr \textwidth-2cm}%
{\centering\small\textit{\BODY}}%
\hfill\llap{(\thequote)}\vspace{1ex}\par}

\newcommand{\integral}{\int}
\newcommand{\isometries}{\mathrm{Isom}}
\newcommand{\volume}{\mathrm{vol}}
\newcommand{\converges}{\to}
\newcommand{\CF}{\mathcal F}
\newcommand{\BR}{\mathbb R}
\newcommand{\BN}{\mathbb N}\newcommand{\CP}{\mathcal P}
\newcommand{\BH}{\mathbb H}

\newcommand{\BC}{\mathbb C}

\newcommand{\PSL}{\mathrm {PSL}}

\newcommand{\inj}{\mathrm{inj}}

\newcommand{\actson}{\circlearrowleft}
\newcommand{\coset}{{\mathrm{Cos}_G}}

\newcommand{\subgroup}{{\mathrm{Sub}_G}}

\title{Unimodular measures on the space of all Riemannian manifolds}
\author{Mikl\'os Ab\'ert and Ian Biringer}

\address {Mikl\'os Ab\'ert \\ Renyi Institute of Mathematics \\
13-15 Realtanoda utca, 
1053 Budapest, Hungary\\}
\email{abert.miklos@renyi.mta.hu}

\address {Ian Biringer \\ Boston College \\
Mathematics Department \\
 140 Commonwealth Ave.
Chestnut Hill, MA  02467-3806\\}
\email{ianbiringer@gmail.com}

\begin{document}

\maketitle

\begin{abstract}
 We study unimodular measures on the space $\mathcal M^d$ of all pointed Riemannian $d$-manifolds. Examples can be constructed from finite volume manifolds, from  measured foliations with Riemannian leaves, and from invariant random subgroups of Lie groups.   Unimodularity  is preserved under weak*  limits, and under certain geometric constraints (e.g.\ bounded geometry) unimodular measures can be used to compactify sets of finite volume manifolds.   One can then understand  the geometry of manifolds $M$ with large, finite volume by  passing to unimodular limits.

We develop a structure theory for  unimodular measures on $\mathcal M^d$, characterizing them via invariance under a certain geodesic flow, and showing that they correspond to transverse measures on a foliated `desingularization' of $\mathcal M^d$. We also give a geometric proof of a compactness theorem for unimodular measures on the space of pointed manifolds  with pinched negative curvature, and characterize unimodular measures supported on  hyperbolic $3$-manifolds with finitely generated fundamental group.

\end{abstract}

\section{Introduction}
\label{intro}
The focus  of this paper is on the class of `unimodular' measures on the space
\[
\mathcal M^d = \big\{\text {pointed Riemannian } d\text {-manifolds } (M,p)\big\} / \text{pointed isometry},
\]  Throughout the paper, \emph {all Riemannian manifolds we  consider are connected and complete.}  We consider $\mathcal M^d$  with the \emph{smooth topology}, where two pointed manifolds $(M,p)$ and $(N,q)$ are  smoothly close if there are  compact subsets of $M$ and $N$  containing large radius neighborhoods of the base points that are  diffeomorphic via a map that is $C^\infty$-close to an isometry, see \S \ref{smoothsec}.

Let $\mathcal M_2 ^ d$ be the space of isometry classes of doubly pointed Riemannian $d $-manifolds $(M, p, q) $, considered in the appropriate smooth topology, see \S \ref{urms}.
\begin {definition}\label {unimodular}
A $\sigma$-finite Borel measure $\mu$ on $\mathcal M ^ d  $ is \emph {unimodular} if and only if for every nonnegative Borel function $f : \mathcal M_2^ d \longrightarrow \BR$ we have:
\begin{equation}
\int_{(M,p) \in \mathcal M^ d} \int_{q \in M} f(M,p,q) \, dq \, d\mu = \int_{(M,p) \in \mathcal M^ d} \int_{q \in M} f(M,q,p) \, dq \, d\mu. 
\label{MTP}
\end{equation}
\end {definition}
Here, \eqref{MTP} is usually called the \emph{mass transport principle} or MTP.  One sometimes considers $f $ to be a `paying function', where $f(M,p,q)$ is the amount that the point $(M,p)$ pays to $(M,q)$, and the equation says that the expected income of a $\mu$-random element $(M,p) \in \mathcal M ^ d $  is the same as the expected amount paid.   Note that two sides of the mass transport principle  can be considered as  integrals $\int f \, d\mu_l$ and $\int f \, d\mu_r$ for two appropriate `left' and `right' Borel measures $ \mu_l,\mu_r$ on $\mathcal M^d_2$, so $\mu$  is unimodular if and only if $ \mu_l=\mu_r$. See the beginning of \S \ref{urms}.

\begin {example}[Finite volume manifolds] Suppose $M$ is a finite volume\footnote{If $M$  has infinite volume, the push forward measure $\mu_M$ still satisfies the mass transport principle, but may not be $\sigma$-finite.  For instance, if $X$ has transitive isometry group then the map  $X \longrightarrow\mathcal M^d$  is a constant map, and $\mu_X$ is only $\sigma$-finite  if $X$ has finite volume.  On the other hand, if the isometry group of $X $ is trivial,  $\mu_X$  will always be $\sigma$-finite.}  Riemannian $d$-manifold, and let $\mu_M$ be the measure on $\mathcal M^d$ obtained by pushing forward the  Riemannian measure $ \volume_M$ under  the map \begin {equation} M \longrightarrow \mathcal M^d, \ \ \ p \longmapsto (M,p).\label {map}
\end {equation} Then both sides of the mass transport principle are equal to the integral of $f(M,p,q)$ over $(p,q)\in M \times M$,  equipped with the product Riemannian measure, so the measure $\mu_M $ is unimodular.  

More generally,  one can construct a  finite unimodular measure on $\mathcal M^d$  from any Riemannian $M$ that regularly covers a finite volume manifold.  The point is that because of the symmetry  given by the action of the deck group, the image of $M $ in $\mathcal M^d$ actually looks like the finite volume quotient.  See  Example \ref{regcover}.\label {finvolex}
\end {example}

\begin {example}[Measured foliations] \label {foliationex} Let $X $ be a \emph{foliated space}, a  separable metrizable space $X$ that is a union of `leaves' that fit together locally as the horizontal factors in a product $\BR^d \times Z$ for some transversal space $Z$.   Suppose $X$ is \emph{Riemannian}, i.e.\ the leaves all have Riemannian metrics, and these metrics vary smoothly in the transverse direction. (See \S \ref{foliatedsec} for details.)   There is then a Borel\footnote{This is Proposition \ref{leafBorel}, where in fact we show that the leaf map is upper semi-continuous in a certain sense, extending a result of Lessa \cite{Lessabrownian}.}  \emph{leaf map}  $X \longrightarrow\mathcal M^d, \ x \longmapsto (L_x,x),$ where $L_x$  is the leaf through $x$.  

 Suppose that $\mu$ is a finite \emph{completely invariant}  measure on $X $, that is, a measure obtained by integrating the Riemannian measures on the leaves of $X $ against some invariant transverse measure, see \cite{Candelfoliations2}.   The push forward of $\mu$ under the leaf map  is a unimodular measure on $\mathcal M ^ d $, see Theorem \ref {desingularizing}.
\end {example}

\begin {example}[Many transitive manifolds] Let $X $ be a Riemannian manifold with transitive isometry group, and  note that any base point for $X$ gives the same element $X \in M ^ d $.  In Proposition \ref{terminologyprop}, we show that \emph{an atomic measure on $X \in \mathcal M ^ d $ is unimodular if and only if $\mathrm{Isom}(X)$ is a unimodular Lie group. }
Examples where $\mathrm{Isom}(X)$  is unimodular include nonpositively curved symmetric spaces, e.g.\ $\BR^d, \BH^d, SL_n\BR / SO(n)$, and  compact transitive manifolds like $\mathbb S^d$.
An example where $\mathrm{Isom}(X)$  is non-unimodular is the $3$-dimensional Lie group $\mathrm{Sol}(p,q)$, where $p\neq q$, equipped with any left invariant metric\footnote{In \cite[Lemma 2.6]{Brofferiobrownian}, it is shown that the isometry group is a finite extension of $\mathrm{Sol}(p,q)$, which is not unimodular when $p\neq q$.}.  

Proposition \ref{terminologyprop} is  one reason why these measures are called unimodular,  although the mass transport principle also has a formal similarity to  unimodularity  of topological groups, being an equality of two `left' and `right'  measures.\label {transitiveex}
\end {example}


\subsection{Motivation} \emph{Though this paper first appeared online in 2016, we have rewritten the section in 2020 to indicate how this paper fits into the field currently. As such, many of the papers referenced below appeared after this one.}

\smallskip

 There are two main reasons to study unimodular measures.  First,  the space $\mathcal M^d$  provides a convenient universal setting in which to view finite volume manifolds, measured foliations, and infinite volume manifolds that have a sufficient amount of symmetry (e.g.\  the transitive examples above).  More importantly, though, interpreting all these examples as measures on a single universal space allows one to define a notion of \emph{convergence} from one to another through weak* convergence of the associated measures on $\mathcal M^d$. Here, recall that $\mu_i\to\mu$ in the \emph{weak* topology} (or for some authors, the \emph{weak topology}) if $\int f \, d\mu_i \to \int f \, d\mu$ for every bounded, continuous function $f:\mathcal M^d \longrightarrow\BR$. 

 An important special  case of weak*  convergence to keep in mind is when  the measures $\mu_i=\mu_{M_i}/\volume (M_i)$,  for some sequence of finite volume manifolds $(M_i)$,  as in Example \ref {finvolex}. If $\mu_i\to\mu$ in the weak* topology, we say that the sequence $(M_i)$ \emph {Benjamini--Schramm (BS) converges} to $\mu$, see e.g.\ \cite{abert2012growth}. This is in honor of an analogous notion introduced by those two authors in graph theory \cite{Benjaminirecurrence}. See also \S \ref{history} below for more of this history. In some sense, the limit measure $\mu$ encodes,  for large $i$,  what the geometry of $M_i$  looks like  near randomly chosen base points,  up to small  metric distortion. As an example,  consider Figure~\ref{rattlefig}.  The transition between the spheres and the neck is lost  in the weak* limit, since  the probability that a randomly chosen base point will  lie near  there is negligible,  and the  small metric distortion allows $\BR^2$  to approximate the large radius spheres.

\begin {figure}[t]
\centering
\includegraphics {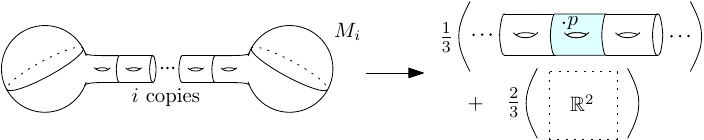} 
\caption {Create surfaces $M_i$  by gluing $i$  copies of some unit volume surface $T$ end-to-end  via a fixed gluing map, capping off with two  round spheres, each with volume $i$, and smoothing the result.  Here, $\mu_{M_i}/\volume (M_i)$ weak*  converges to  a convex combination of an atomic measure on  the single point  $\BR^2$ in $\mathcal M^d$,  and a probability measure constructed by gluing bi-infinitely many copies of $T$  together  and choosing a random base point from, say, the center copy.}
\label{rattlefig}
 \end {figure}

As mentioned in \S \ref{urms}, unimodularity is preserved under weak* limits, so in particular, all BS-limits of finite volume manifolds are unimodular\footnote{The converse is open, and is essentially equivalent to the analogous question for \emph{unimodular random graphs}, which generalizes the open question of whether \emph{all groups are sofic}, see \cite{Aldousprocesses}.}. 
So, if one can develop a robust theory of unimodular measures, one can then try to use this theory to analyze  sequences of finite volume manifolds via their weak* limits, as above. 

As an example, let $X$ be a irreducible symmetric space of noncompact type. An \emph{$X$-manifold} is a quotient $M=\Gamma\backslash X$, where $\Gamma$ acts freely, properly discontinuously and isometrically. In recent joint work with Bergeron and Gelander \cite{abert2018convergence}, using the framework developed in this paper we showed:

\begin{theorem}[Abert, Bergeron, Biringer, Gelander \cite{abert2018convergence}]\label{abbg}
	Suppose that $X$ is not a  metric scale of $\BH^3$. If $(M_i)$ is any BS-convergent sequence of finite volume $X$-manifolds, then for all $k$ the sequence $b_k(M_i)/vol(M_i)$ converges.
\end{theorem}

Essentially, this means that  in the context above, a given limit measure $\mu$ has some sort of `Betti number' that is the limit of the volume normalized Betti numbers of any approximating sequence of $M_i$'s. It would be interesting to develop an intrinsic definition of such a Betti number for an arbitrary unimodular measure $\mu$ on $\mathcal M^d$. Relatedly, an analogous definition has been recently given by Michael Schr{\"o}dl for \emph{unimodular random rooted simplicial complexes}, see \cite{schrodl2018ell}. However, the manifold case presents additional difficulties.

One situation in which there is an existing definition, though, is when the unimodular measure $\mu$ is just an atomic measure $\mu_X$ on the single point $X \in \mathcal M^d$. In this case, the appropriate invariants are the \emph{$L^2$-Betti numbers} $\beta^{(2)}_k(X)$, see e.g.\ \cite{abert2012growth} for definitions. In our 2012 paper with Abert et al \cite{abert2012growth}, we all had previously shown that if $rank_\BR(X)\geq 2$, then \emph{any} sequence of distinct finite volume $X$-manifolds BS-converges to $\mu_X$. Combining this with Theorem \ref{abbg} above gives:

\begin{theorem}[Corollary 1.4 of \cite{abert2018convergence}]
	 Suppose that $rank_\BR(X)\geq 2$ and $(M_i)$  is any sequence of distinct finite volume $X$-manifolds. Then for all $k\in \BN$, we have
$$b_k(M_n) / vol(M_n) \to \beta_k^{(2)}(X).$$
\end{theorem}

This extends earlier work with Nikolov-Raimbault-Samet in \cite{abert2012growth}, and is a uniform version of L\"uck's approximation theorem \cite{Luckapproximating} that applies to all quotients of a fixed symmetric space, not just covers of a single quotient. We stress that it is necessary in \cite{abert2018convergence} to work geometrically, and so most of that paper is written using the framework of convergence of measures on $\mathcal M^d$, as developed in this paper.

\smallskip

\smallskip

Subsequent work of Ab\'ert--Bergeron--Masson \cite{abert2018eigenfunctions} exploits the language of Benjamini-Schramm convergence of manifolds introduced in this paper to analyze eigenfunctions of the Laplacian for compact Riemannian manifolds. The asymptotic behavior of eigenfunctions has been studied extensively in the literature. There are two major directions of interest: one can study the eigenvalue aspect, where one has a fixed manifold and the energy of the eigenfunction tends to infinity, and the level aspect, where one looks at a covering tower of manifolds (usually coming from a subgroup chain for an arithmetic lattice) and the energy converges to a fixed value. It was understood in the community that these aspects are related and most theorems on one side tend to find their counterparts on the other side. Our language of Benjamini-Schramm convergence now unifies the level and eigenvalue aspects. Indeed, for a covering tower, the limit of eigenfunctions will be an invariant random eigenwave on the limiting space (usually the symmetric space of the corresponding Lie group). For a fixed manifold, rescaling the Riemannian metric with the energy will produce a sequence of manifolds and a fixed eigenvalue, hence the limiting eigenwave will live on the standard Euclidean space. The language of Benjamini-Schramm convergence, in particular, allows one to give the first mathematically precise formulation for the famous Berry conjecture in physics, and connects the conjecture to Quantum Unique Ergodicity. See \cite{abert2018eigenfunctions} for details.

\subsection{History and related papers}
\label{history}
Much of  our work is inspired by a recent program in graph theory, in particular work of Aldous-Lyons \cite{Aldousprocesses} and Benjamini-Schramm \cite{Benjaminirecurrence}.   For instance, the term `unimodularity' was  previously used in \cite{Aldousprocesses}, for measures $\mu$ on the space
\[
\mathcal G = \big\{\text {rooted, connected,  locally finite graphs } (G,p)\big\} / \text{automorphism}
\] 
 such that  for any Borel  function $f$ on the space of doubly rooted graphs,
\begin {equation}\label {graphMTP}
\int_{(G,p) \in \mathcal G} \sum_{q \in V(G)} f(G,p,q) \, dq \, d\mu = \int_{(G,p) \in \mathcal G} \sum_{q \in V(G)} f(G,q,p) \, dq \, d\mu. 
\end {equation}
 In fact, this  version of the mass transport principle appeared  even earlier in \cite{Benjaminirecurrence},  generalizing a concept important in percolation theory \cite{Benjaminigroup,Haggstrominfinite}.

 As in  Example \ref {finvolex}, every finite graph $G$ gives a unimodular measure  $\mu_G$ on $\mathcal G$, by  pushing forward the  counting measure on its vertices under the map
$$G \longrightarrow \mathcal G ,\ \ p \longmapsto (G,p).$$
Similarly, a transitive graph gives a  unimodular measure on $\mathcal G$ if and only if its automorphism group is unimodular, see \cite {Benjaminigroup} and \cite[Section 8.2]{Lyonsprobability}. One  can study unimodular measures on  $\mathcal G $ that are  weak* limits of the $\mu_G$, c.f.\cite{Abertbenjamini,Benjaminirecurrence,Namazidistributional}, and extending results  known for finite graphs to  arbitrary unimodular measures on $\mathcal G$  has recently become a small industry, see e.g.\ \cite{Aldousprocesses,Benjaminigroup,Benjaminiunimodular,Haggstrominfinite}.

\vspace{2mm}

 Ideas similar to ours have  also  appeared previously in the continuous setting,  even apart from ABBGNRS \cite{abert2012growth}.  Most directly, Bowen \cite{Bowencheeger}  used unimodular measures on the space of pointed metric measure spaces  to bound the Cheeger constants of hyperbolic $4$-manifolds with free fundamental group.  In his thesis, Lessa \cite{Lessabrownian}, see also \cite{Lessareeb,Lessabrownianpaper},  studied measures on $\mathcal M^d$  that are {stationary} under Brownian motion, of which unimodular measures are examples\footnote {On foliated spaces, Lessa's `stationary measures' correspond to harmonic measures,  while our unimodular measures  correspond to  completely invariant measures. See \S \ref{foliatedsec}, \cite{Candelfoliations2} and \cite{Lessabrownian}. Also, see Benjamini--Curien \cite{Benjaminiergodic}  for  a corresponding theory of stationary random graphs.},  and  a few of the technical  parts of this paper are  similar to  parts of his. Namazi-Panka-Souto~\cite{Namazidistributional},  analyzed weak*-limits of the measures $\mu_{M_i}/\volume(M_i)$ for  sequences $(M_i)$ of manifolds   that are all quasi-conformal deformations of a fixed closed manifold and that all have bounded geometry.  Also, Vadim Kaimanovich  has for some time  promoted measured foliations  from a viewpoint similar to ours, and we refer the reader to his papers \cite{Kaimanovichinvariance,Kaimanovichbrownian,Kaimanovichconformal}  for culture.

\label {motivation}

\subsection{Statements of results}
\label{uniintro}
\vspace{2mm}

Most of the paper concerns the structure theory of unimodular measures.   The case where $\mu$  is a unimodular  probability measure is of particular interest: a $\mu$-random  element of $\mathcal M^d$  is  then called a \emph{unimodular random manifold} (URM). In this section, we will start by  explaining  the close relationship  between unimodular  measures and completely invariant measures on foliated spaces, as mentioned in Example \ref {foliationex}. Then, we will outline the dictionary between \emph {invariant random subgroups} and unimodular random locally symmetric spaces.  As an interesting  trip to the zoo, a characterization of unimodular random hyperbolic $2$ and $3$-manifolds with finitely generated fundamental group is given.  We then  discuss  conditions under which  sets of  unimodular  measures on $\mathcal M^d$ are weak* compact, and finish with a discussion of the  rather long appendix, where it is shown that $\mathcal M^d$  and various related spaces have reasonable topology. 

\subsubsection{Unimodularity and foliated spaces}

 As mentioned above, a separable, metrizable space $X$ is a \emph{foliated space} if it is a union of leaves that fit together locally as the horizontal factors in a product $\BR^d \times Z$ for some transversal space $Z$. We say $X$ is \emph{Riemannian} if the leaves all have Riemannian metrics, and if these metrics vary smoothly in the transverse direction. See \S \ref{foliatedsec} for details.

On such an $X$, let $L_p$   be the leaf through $p$. A $\sigma $-finite Borel measure $\mu$ on $X$ is \emph{unimodular} if for every nonnegative Borel $f : X \times X \longrightarrow \BR$ we have:
\begin{equation}
\int_{p\in X} \int_{q \in L_p} f(p,q) \, d\volume_{L_p} \, d\mu = \int_{p\in X} \int_{q \in L_p} f(q,p) \, d\volume_{L_p} \, d\mu. \label{folmtp}
\end{equation}
 Also, a measure $\mu$ on $X$ is called \emph {completely invariant} if it is obtained by integrating the Riemannian measures on the leaves of $X$  against some invariant transverse measure,  see \S \ref{foliatedsec} and \cite{Candelfoliations2}.  We then have:

\begin{theorem}\label{foliationsintro}
Suppose that $X $ is a Riemannian foliated space and $\mu $ is a $\sigma$-finite Borel measure on $X$.  Then the following are equivalent:
\begin {enumerate}
\item $\mu$ is completely invariant,
\item $\mu $ is unimodular,
\item $\mu $ lifts uniformly to a measure $\tilde \mu$ on the leaf-wise unit tangent bundle $T^1X$ that is invariant under leaf-wise geodesic flow.
\end {enumerate}
\end{theorem}

To understand the `uniform lift' in 3), take  a measure $\mu$ on $X$, and integrate $\mu$ against the (round) Riemannian measures on all leaf-wise tangent spheres $T^1_p L$  to get a measure $\tilde \mu$ on  $T^1X$.  A version of this condition will reappear below as an alternative characterization  of unimodularity  for measures on $\mathcal M ^d.$ Theorem \ref {foliationsintro} is proved in \S \ref{foliatedsec},  where it is restated as Theorem~\ref {foliations}. Two additional characterizations of unimodularity  are included in the new statement,  one of which parallels a well-known result in graph theory.

\vspace{2mm}

 In some sense, the space $\mathcal M^d$ is  itself almost foliated, where the `leaves' are the subsets obtained by fixing a manifold $M $ and varying the basepoint $p\in M $.   One would like to  say that unimodular measures are  just completely invariant measures, with respect to this foliation. However, due to the  equivalence relation defining $\mathcal M^d$, these `leaves' are actually of the form $M / \mathrm{Isom}(M)$,  so the foliation is highly singular,  and complete invariance does not make sense.

 However, there is a way to make this point of view precise.  Recall that if $X$  is any Riemannian foliated space, its \emph{leaf map}  takes $x\in X$ to  the pointed Riemannian manifold $(L_x,x)$, where $L_x$  is the leaf through $x$.  We then have:

\begin {theorem}[Desingularizing $\mathcal M^d$]\label {desingularizing}
If $\mu$ is a completely invariant probability measure on a Riemannian foliated space $X$, then $\mu$ pushes forward  under the leaf map to a unimodular probability measure $\bar \mu$ on $\mathcal M ^ d $. 

Conversely, there is a Polish Riemannian foliated space $\mathcal P ^ d$  such that any $\sigma$-finite unimodular measure on $\mathcal M ^ d $ is the push forward  under the leaf map of some completely invariant measure on $\mathcal P ^ d $.    Moreover, for any fixed  manifold $M $, the preimage of $\{(M,p) \ | \ p \in M\} \subset \mathcal M^d$  under the leaf map is a union of leaves of $\mathcal P^d$, each of which is isometric to $M $.
\end {theorem}

 This theorem  indicates an advantage our continuous framework has over graph theory: although  the mass transport principle \eqref{graphMTP}  does indicate a compatibility between unimodular measures on $\mathcal G$ and the counting measures  on the vertex sets of fixed graphs $G$,  there is no precise statement saying  that  unimodular measures are  made by  locally integrating up these counting measures  in analogy with  Theorem \ref{desingularizing}.  In some sense, the problem is that graphs do not have enough local structure for this perspective to translate.

 \'Alvarez L\'opez and Barral Lij\'o~\cite{Alvarezbounded} independently prove a desingularization theorem similar to Theorem~\ref{desingularizing},  which they  use to  show that any manifold with bounded geometry can be realized   isometrically as a leaf in a compact Riemannian foliated space.  As their goals are topological, rather than measure theoretic, their foliated space is not set up so that one can lift   measures on $\mathcal M^d$ to the foliated space using Poisson processes, though, a property that is crucial in our applications.

The idea behind the construction of $\mathcal P^d$ is simple. Since the problem  is that Riemannian manifolds may have  nontrivial isometries, we set $\mathcal P^d$ to  be the set of  isometry classes of triples $(M,p,D)$ where $p\in M$  is a base point and $D\subset M$ is a closed subset {such  that there is no isometry $f : M \longrightarrow M$ with $f(D)=D$}.  The leaves are obtained by fixing $M$ and $D$ and varying $p$, and the leaf map is just  the projection $(M,p,D) \mapsto (M,p)$. However, it takes some work to see that these leaves  fit together locally into a product structure $\BR^d \times Z$: a  brief sketch of this argument is given in the beginning of \S \ref{proofpdsec}. Assuming this, though, measures  on $\mathcal M^d$  induce measures on $\mathcal P^d$ after  integrating against a Poisson  process on  each fiber  of the leaf map. See \S \ref{pdsec}  for details.

Completely invariant measures on foliated spaces have been well studied, e.g.\ \cite {Candelharmonic,Feresleafwise,Garnettstatistical,Garnettfoliations,Ghystopologie}. So for instance,  one can now take a sequence of finite volume manifolds $(M_i)$, pass to the associated unimodular probability measures $\mu_{M_i}/\volume(M_i)$, extract a weak* limit measure $\mu$, study this $\mu$ using tools from foliations,  and deduce results about the manifolds $M_i$.

 For those working in foliations, the  mass transport principle \eqref{MTP} may seem  less interesting now that we know  unimodularity can also be characterized in terms of complete invariance.   However, we would like to stress that often, the MTP  is the more convenient definition to use.  We illustrate this in Theorem~\ref{hypcompact}, where we use the MTP  to give a proof of weak*  compactness of  the set of unimodular  measures supported on manifolds with pinched negative curvature.  For another example, Biringer-Raimbault \cite {Biringertopology} have studied the space of ends of a unimodular random manifold, showing for example that it has either $0,1$ or $2$ elements, or is a Cantor set.  This parallels a result of Ghys \cite {Ghystopologie} on the topology of generic leaves of a measured foliation.  Neither of these  results quite implies the other, although Ghys's result is  really more general,  as it  applies to harmonic measures, and not just completely invariant ones.  However, the MTP  encapsulates a recurrence that makes the proof in \cite {Biringertopology} extremely short.

 One other reason  to prefer unimodularity in our setting is that to talk about complete invariance, one must leave $\mathcal M^d$, passing to an associated foliated space using the desingularization theorem.  On the other hand,  the geodesic flow invariance  of Theorem~\ref {foliationsintro} can be  phrased (more or less) directly within $\mathcal M^d$.

\begin {theorem}\label {gflow}
Suppose that $\mu $ is a Borel measure on $\mathcal M^d$.  Then $\mu $ is unimodular if and only if its uniform lift $\tilde \mu$ on $T^1 \mathcal M^d$ is  geodesic flow invariant.
\end {theorem}

See \S \ref{pdsec}  for the proof. Here, $T^1 \mathcal M^d$ is the space of   isometry classes of rooted unit tangent bundles $(T^1M,p,v)$, where $v\in T_p^1 M$. Each fiber $T^ 1_p M$  of  $$T^1\mathcal M^d \longrightarrow \mathcal M^d, \ \ (M,p,v) \longmapsto (M,p)$$ comes with a  natural Riemannian metric  induced by the inner product on $T_p M$, and we write $\omega_{M,p}$ for the associated Riemannian measure on $T ^ 1_p M $.  Then $\tilde \mu$  is the measure on $T ^ 1\mathcal M^d$ defined by the equation $d\tilde\mu =\omega_{M, p} \, d\mu.$  The geodesic flows on  individual $T^1M $ combine to give a  continuous flow $$g_t : T^1\mathcal M^d \longrightarrow T^1\mathcal M^d$$  and this is the  geodesic flow referenced in the statement of the theorem.

This theorem is an analogue of a result in graph theory. Let $\mathcal G_d \subset \mathcal G$ be the space of isometry classes of pointed $d$-regular graphs. The associated space $\mathcal E_d$ of $d$-regular graphs with a distinguished oriented edge projects onto $\mathcal G_d $,  where the map replaces a distinguished edge with its original vertex.   For each $(G,v) \in \mathcal G_d$, the  uniform probability measure on the set of $d$  edges originating at $v$   quotients to a probability measure on the fiber over $(G,v)$ in $\mathcal E_d$.   Integrating these  fiber measures against $\mu$  gives  a measure $\tilde\mu $ on $\mathcal E $, and Aldous-Lyons  \cite {Aldousprocesses} proved that $\mu $ is unimodular if and only if $\tilde\mu $ is invariant under the map $\mathcal E \to \mathcal E $ that switches the orientation of the distinguished edge.

\subsubsection{Unimodular  random manifolds  and IRSs}

 We now focus on unimodular probability measures $\mu$ on  $\mathcal M^d$,  in which case a $\mu$-random  element of $\mathcal M^d$  is called a \emph{unimodular random manifold} (URM).  There is a close relationship between URMs and \emph {invariant random subgroups} (IRSs),  which  have been studied in \cite {Biringerunimodularity,bowen2012invariant,bowen2012invariantb,Eisenmanngeneric,
Gelanderinvariant,hartman2012furstenberg,hartman2013stabilizer,Thomasinvariant}. 

Let $G $ be a  locally compact, second countable group, and let $\subgroup$  be the space of closed subgroups of $G$,  endowed with its \emph {Chabauty topology}, see \ref{Chabautysection}. 

\begin {definition}  An \emph {invariant random subgroup} (IRS) of $G$ is a random element of a Borel  probability measure  $\mu$ on $\subgroup $ that is invariant under the conjugation action of $G $ on $\subgroup$.\end {definition}

 When $G $ is  finitely generated, say by  a symmetric set $S$, there is a dictionary between IRSs of $G$  and \emph {unimodular random $S$-labeled graphs,} or URSGs,  which we will briefly explain. An \emph {$S$-labeled graph} is a countable directed graph with edges labeled by elements of $S$, such that the edges coming out from any given vertex $v$ have labels in 1-1 correspondence with elements of $S$, and the same is true for the labels of edges coming into $v$. Every  subgroup $H<G$  determines an $S$-labeled \emph {Schreier graph} $\mathrm{Sch}_S(H\backslash G)$, whose vertices are right cosets $Hg$ and where each $s \in S$ contributes  a labeled edge from every coset $Hg$ to $Hgs$.  Note that $\mathrm{Sch}_S(H\backslash G)$ comes with a natural base point, the  identity coset $H $.

In a variation of the discussion in \S \ref {motivation}, let $\mathcal  G_S $ be the space of isomorphism classes of rooted  $S $-labeled graphs. A URSG is a random element of $\mathcal  G_S $  with respect to a probability measure  that satisfies the appropriate $S$-labeled analogue of the mass transport principle \eqref{graphMTP}. A random  subgroup $H < G$  determines a random rooted  Schreier graph, and conjugation invariance of the distribution of $H$ is equivalent  to unimodularity of $\mathrm{Sch}_S(H\backslash G)$. So, IRSs of $G$  exactly correspond to URSGs. See \cite{abert2012kesten} and \cite[\S 4]{Biringerunimodularity}  for details.

 In the continuous setting, IRSs were first  studied in ABBGNRS \cite{abert2012growth}.   In analogy with the above, when a group $G $ acts on $X$  by isometries, there should be a dictionary between IRSs of $G$ and certain unimodular random $X$-manifolds. Here, an \emph {$X$-manifold} is just a quotient $\Gamma \backslash X$, where $\Gamma$  acts freely and properly discontinuously  on $X $ by isometries.   Two parts of this dictionary are discussed in \S \ref {unimodularIRS},  in the cases where $G $ acts transitively, or discretely  on $X $.   The following is a particularly nice case of our  analysis  of IRSs of  transitive $G$.
  
 \begin {proposition}[URMs vs IRSs]\label {urmirsintro}
   Suppose $X $ is a simply connected Riemannian manifold     whose isometry group  is unimodular and acts  transitively.	Then there is a weak*-homeomorphism between  the spaces of  distributions of discrete, torsion free IRSs of $\mathrm{Isom}(X)$  and of unimodular random $X$-manifolds.
 	\end {proposition}
 
So, $X $ could be a non-positively curved symmetric space, for instance $\BR ^d,$ $\BH ^d$ or $SL_n\BR/SO(n)$.  Note that when we say an IRS or URM has a property like `torsion free' or `$X$', this  property is to be  assumed to be satisfied almost always.   For instance,  the above says that there is a homeomorphism between the space of  conjugation invariant probability measures $ \mu$ on $\subgroup $ such that $\mu$-a.e.\ $H\in \subgroup $ is discrete and torsion free,  and the space of unimodular  probability measures $\mu$ on $\mathcal M^d$  such that $\mu$-almost every $(M,p)$ is an $X$-manifold.  
  
\subsubsection{Compactness theorems}

 To understand sequences of  finite volume manifolds, then, one  would naturally like  to understand conditions under which  sets of unimodular probability measures are compact, so that a unimodular weak* limit  of the measures $\mu_{M_i}/\volume (M_i)$ can be extracted after passing to a subsequence.

 By work of Cheeger and Gromov, c.f. \cite{Gromovmetric} and \cite[Chapter 10]{Petersenriemannian}, the  subset of $\mathcal M^d$  consisting of pointed manifolds $(M,p)$ with bounded geometry  is compact. Here, \emph {bounded geometry}  means that the  sectional curvatures of $M$, and all of their derivatives, are uniformly bounded above and below, and  the injectivity radius at  the base point $p$  is bounded away from zero. See \S \ref {hypcompactsec}.  
By the Riesz representation theorem and Alaoglu's theorem,  this implies that  the set of unimodular probability measures  supported on manifolds with bounded geometry is weak* compact, since unimodularity  is a weak*  closed condition.

Both the curvature bounds and the bound on injectivity radius are necessary for compactness of pointed manifolds.  However, we show that  in the presence of pinched negative curvature, an  injectivity radius bound is unnecessary for compactness once we pass to measures:

\begin {theorem}\label{hypcompact}
 The set of all  unimodular probability measures on $\mathcal M^d$  that are  concentrated on pointed manifolds with pinched negative curvature and uniform upper and lower bounds on all derivatives of curvature is weak* compact. 
\end {theorem}

See \S \ref {hypcompactsec}  for a more precise  statement.  The condition on the derivatives of curvature is only necessary because we consider $\mathcal M^d$  with the smooth topology; a topology of weaker regularity would require weaker assumptions.   Essentially, the reason the injectivity radius assumption is not necessary is because in pinched negative curvature,  the  $\epsilon $-thin part of a manifold takes up at most some uniform proportion $C(\epsilon) $ of the total volume, where $C\to 0$  as $\epsilon\to 0$.  In fact, Theorem~\ref{hypcompact}   boils down to a  precise version of this  kind of statement, see the proof of Proposition~\ref{thickbase},  that still applies  to manifolds with infinite volume.

 We explain in \S \ref{hypcompactsec}  that there is no analogue of  Theorem \ref {hypcompact} in nonpositive curvature, but  using work from ABBGNRS \cite{abert2012growth}, one can show that  we still have a weak*  compactness theorem for locally symmetric spaces:

\begin{theorem}\label {Xcompact} Let $X$ be a  symmetric space of nonpositive curvature with no Euclidean factors,  and let $\mathcal M^X \subset \mathcal M^d$ be the subset of pointed $X$-manifolds. Then the space of unimodular probability measures  on $\mathcal M^X$ is weak*-compact. 
\end{theorem}

 The proof of  Theorem \ref {Xcompact} is algebraic: it uses the dictionary between unimodular measures and IRSs  discussed in the previous section, and  arguments related to Borel's density theorem, c.f.\ \cite{Furstenbergnote}.  We give this proof in \S \ref{localsymcomp},  and also briefly discuss the  question of whether there is a universal theorem that generalizes both  Theorems \ref{hypcompact} and \ref{Xcompact}.

\subsubsection{Hyperbolic $3$-manifolds  with finitely generated fundamental group, and The No-Core Principle}
\label {fingenintro}

 Finite volume hyperbolic manifolds have finitely generated $\pi_1$,  \cite{Benedettilectures}.   While the converse is not true in general,  the question is  at least interesting for $d$-manifolds $M$ with enough symmetry:  is it true that when $M \neq \BH^d$ regularly covers a  finite volume $d$-manifold and $\pi_1 M$  is finitely generated, then $M$ has finite volume? 

When $d=2$,  the answer is yes. Any surface $S$ with finitely generated  fundamental group is {geometrically finite}, see \cite[Theorem 4.6.1]{Katoklectures}. If $S\neq \BH^2$  regularly covers a finite volume surface, its  limit set is the entire circle $\partial_\infty \BH^2$,  say by \cite[Theorem 12.2.14]{Ratcliffefoundations}, and  then geometric finiteness  implies that $S$  has finite volume, see \cite[Theorem 4.5.1]{Katoklectures}. The question is open for $d\geq 4$.

When $d=3$, Thurston's fibered hyperbolization theorem \cite{Thurstonhyperbolic1} states that the
mapping torus $M $ of a pseudo-Anosov homeomorphism of a surface $S$ admits a hyperbolic metric. The fundamental group of $M $ splits as the semidirect product
\begin{equation}
\label{semidirect}1 \longrightarrow\pi_{1} S\longrightarrow\pi_{1} M
\longrightarrow\mathbb{Z }\longrightarrow1,
\end{equation}
and the regular cover $\hat M$  corresponding to $\pi_1 S$   is a hyperbolic $3$-manifold  with finitely generated fundamental group.  However, it  is a well-known consequence of the Tameness Theorem of Agol~\cite {Agoltameness} and Calegari-Gabai \cite {Calegarishrinkwrapping} and Canary's covering theorem~\cite{Canarycovering} that  these $\hat M$ are the only examples when $d=3$.

\vspace{2mm}

A \emph {unimodular  random hyperbolic manifold} (URHM) is, as should be expected,  a random element of $\mathcal M^d$  with respect to a  unimodular  probability measure concentrated on  pointed hyperbolic $d$-manifolds.   Simple examples include a finite volume hyperbolic manifold  with a randomly chosen base point, and  the hyperbolic space $\mathbb H^d$. (See Examples \ref{finvolex} and \ref{transitiveex}.)

Any regular cover of a hyperbolic manifold can be considered as a URHM, via Example \ref{regcover}. It turns out that URHMs have enough symmetry that the rigidity results for regular covers discussed above have analogues for URHMs  with finitely generated fundamental group.    For instance,  it follows from \cite[Proposition 11.3]{abert2012growth} that the limit set of a URHM  $M$, with $M \neq \BH^d$,  is always the entire boundary sphere $\partial_\infty \BH^d$. When $d=2$, this means that any URHM with finitely generated $\pi_1$  has finite volume,  via the  same argument as above.

When $d=3$, we constructed examples in \cite[\S 12.5]{abert2012growth} of IRSs (hence URHMs, by  Proposition \ref{urmirsintro}) with  finitely generated $\pi_1$  that are not  regular covers of  finite volume manifolds.  However, these examples  all have the same coarse geometric structure as the $\hat M$  examples above: they are all \emph {doubly degenerate hyperbolic $3$-manifolds homeomorphic to $S \times \BR$},  for some  finite type surface $S$. See \S \ref{fingensec}  for definitions. Here, we show that these are the only examples:

\begin {theorem}\label {fingen3dim}
 Every 	unimodular  random hyperbolic $3$-manifold with finitely generated fundamental group  either is isometric to  $\BH^3$,  has finite volume,  or is a doubly degenerate hyperbolic structure on $S\times \BR $ for some finite type $S$.
\end {theorem}

\begin{figure*}
\centering
\includegraphics{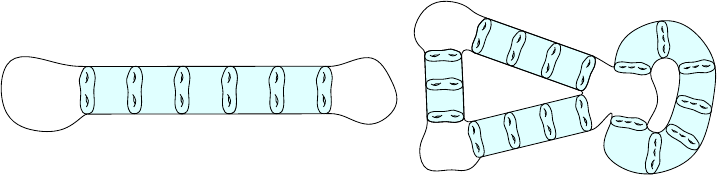}
\caption{  A schematic picturing a number of \emph {long product regions} in hyperbolic $3$-manifolds. Each is homeomorphic to  a product $S \times [0,1]$,  for some finite type surface $S$,  and has bounded area level surfaces.  If the  product regions lengthen while  the complexity of the underlying graph stays bounded, the  associated probability measures $\mu_{M_i}/\volume(M_i)$  on $\mathcal M^d$ weak* converge to  a  URHM  that is a doubly degenerate  hyperbolic structure on $S\times\BR $.}
\label {lpr}\end{figure*}

Here is another  informal way to  motivate Theorem \ref{fingen3dim}.  Suppose that $M $ is a hyperbolic $3$-manifold with finite, but large, volume.   Randomly choose a point $p\in M$ and consider a neighborhood $U \ni p$  with some  fixed radius $R$,  which is large, but say not as large as $\volume(M)$.  What can $U $  look like geometrically?  On the one hand, it could be a large embedded ball from $\BH^3$, while at the other extreme, it could have  very complicated topology,  requiring many elements to generate $\pi_1 U$.  If $\pi_1 U$  can be generated by few elements, though,   the geometry  of $U $ is more limited:  essentially, it will look like a  large piece of an  infinite volume hyperbolic $3$-manifold $N$ with some small number of ends.  

Now, in any sufficiently large piece an  infinite volume $N $, the ends of $N$  take up  a much larger proportion of volume than the `core' of $N$ does. So,  the probability that our base point $p$ was randomly chosen to lie inside the core is negligible.  In other words,  most choices of $p$  that end up in a neighborhood $U$  with not so complicated topology will look like they are  stuck deep inside an end of an infinite volume hyperbolic $3$-manifold. By  the geometric classification of ends  of hyperbolic $3$-manifolds, c.f.\ \cite{Canaryends} and also \cite[The Filling Theorem]{Canarycovering},  this means that the point $p$ will  either be contained in a large embedded  ball from $\BH^3$ (when the end is geometrically finite) or a \emph {long product region} (when the end is degenerate). See  Figure \ref{lpr}.

 Informally, this discussion means that near a randomly chosen point  in a  hyperbolic $3$-manifold $M$ with large finite volume, $M$  will  look like  either
\begin {enumerate}
\item a  large embedded ball from $\BH^3$,
\item a region  for which the minimal number of $\pi_1$-generators is `very large',
\item  a  {long product region}.
\end {enumerate}
 To relate this back to Theorem \ref{fingen3dim},  note that large  hyperbolic balls and  long product regions are  exactly what one obtains by taking large neighborhoods within $\BH^3$ and  doubly degenerate  hyperbolic structures on $ S \times \BR$, which are the only infinite volume manifolds allowed in the theorem. And if a  sequence $(M_i)$ of finite volume manifolds with $\volume(M_i)\to \infty$  gives a sequence of  probability measures $\mu_{M_i}/\volume (M_i)$  that weak*  converges to a unimodular $\mu$  on $\mathcal M^d$,  the    informal local analysis of the geometry of $M_i$  described above translates exactly into Theorem \ref{fingen3dim}.  In fact,  in light of the weak* compactness of the set of unimodular  probability measures  supported on hyperbolic manifolds,  which  follows from Theorem \ref{hypcompact} or  Theorem \ref{Xcompact},  one can view Theorem \ref{fingen3dim} as a precise  version of the  informal statement that  for \emph{any}  large-volume $M$,  the local geometry near a randomly chosen base point is as described above.

\vspace{2mm}

The key idea in  the proof of Theorem \ref{fingen3dim} is  the following:

\begin{theorem}[The No-Core Principle]\label {nocore}	 Suppose that $\mu$ is a unimodular probability measure on $\mathcal  M ^ d$ and that $f:\mathcal M^d\longrightarrow \{0,1\}$ is a Borel function. Then for $\mu$-almost every element  $(M,p)\in \mathcal  M ^ d$, we have 
	 $$0<\volume_M \{q\in M \ | \ f(M,q)=1\}<\infty \implies \volume(M)<\infty.$$ 
	\end{theorem}
	 Geometrically,  one should imagine that $f(M,p)=1$ when the base point $p$ lies in a `core' of $M$.  The theorem then says that when $(M,p)$ lies in the support of a unimodular probability measure, one can only Borel-select a core with finite, nonzero volume for $M$ when $M$ has finite volume.  While the statement above is very useful---it also is used in Biringer-Raimbault \cite{Biringertopology}---it is basically an immediate consequence of the mass transport principle, see \S \ref{nocoresec}.

 Essentially, the  proof of Theorem \ref{fingen3dim}  is that any hyperbolic $3$-manifold with finitely generated fundamental group has a  finite volume core,  obtained by chopping off  neighborhoods of its infinite volume ends.  However, it requires some work to choose the core  in a canonical enough way so that the function $f$ in the No-Core Principle  is Borel.  See \S \ref{fingensec}  for details.

\subsubsection{Appendix: the topology of  $\mathcal M^d$  and the Chabauty topology}
  
 The paper ends with a lengthy appendix. We give in \S \ref{smoothsec} and \ref{stabilitysec} slight extensions of  existing compactness  and stability results  concerning smooth  convergence, but most of the appendix is spent showing $\mathcal M^d$ and various related spaces are Polish. This is necessary to justify the use of measure theoretic tools like Rohlin's disintegration theorem, or Varadarajan's  compact model theorem (see  the proof of Proposition \ref {ergodicdecomposition}).

Candel, \'Alvarez L\'opez and Barral Lij\'o \cite{Alvarezuniversal} have independently and concurrently studied the space $\mathcal M^d$, proving that it is Polish and establishing a number of  interesting topological properties that are very related to this paper, e.g.\ to Proposition~\ref{leafBorel}. Their proof was made publicly available before ours, so the result is really theirs.  The two approaches are quite similar, but our proof produces an explicit metric that we use elsewhere in the paper, and is simpler in some ways, so we still present it here.

We find the proofs that these spaces are Polish quite interesting. For instance, recall that two points $(M,p)$ and $(N,q)$ in $\mathcal M^d$  are smoothly close if  there is a diffeomorphism $f$ from a large neighborhood $B \ni p$  to a large neighborhood of $q$, such that the metric $\langle , \rangle_N$ on $N $ pulls back to a metric on  $B$ that is $C^ \infty$ close to $\langle , \rangle_M$, see \S \ref{smoothsec}.  To metrize this definition directly, one would have to  metrize the $C ^\infty$ topology on  the appropriate space of tensors  on $M$ separately for each $(M,p) \in \mathcal M^d$, and then hope that the choice is  canonical enough that the triangle inequality holds for  the induced metric on $ \mathcal M^d$.    This is hard to do, so instead we  define the distance between $(M,p)$ and $(N,q)$ by  measuring the  bilipschitz distortion of the `iterated total derivatives' $$D^k f: T^kM \longrightarrow T ^k N$$
 on the  $k$-fold iterated tangent bundles  $T^kM = T(T(\cdots T(M))\cdots ),$  which we consider with the associated `iterated Sasaki metrics'. See \S \ref{secmetrizability}.

\subsection{Plan of the paper}

 Section \ref{urms}  introduces unimodular measures in detail,  and discusses the No-Core  Principle  and the dictionary between URMS and IRSs.  In Section \ref{foliatedsec},  we review completely invariant measures on foliated spaces, and   prove Theorem \ref{foliations},  which  shows that complete invariance is equivalent to unimodularity, among other things. Section \ref{foliatedstructuresec}  discusses the foliated structure of $\mathcal M ^ d$,  the desingularization theorem,  and the ergodic decomposition of  unimodular measures.  The weak*  compactness theorems  mentioned in the introduction are proved in  Section \ref{hypcompactsec},  while the characterization of unimodular random hyperbolic $3$-manifolds  with finitely generated $\pi_1$  is the main focus of Section \ref{fingensec}.  The paper ends with an  appendix  concerning the topology of $\mathcal M^d$  and related spaces.

\subsection{Acknowledgments}

We would like to thank Nir Avni, Igor Belegradek, Lewis Bowen, Renato Feres, Etienne Ghys, J.\ A.\ \'Alvarez L\'opez, Juan Souto, Ralf Spatzier and Shmuel Weinberger for a number of useful discussions. We also thank a referee for very helpful comments, and for finding a significant error in the leadup to Theorem~\ref{manifoldmetrizable}. The second author was partially supported by NSF grant DMS-1308678.

\section{Unimodular measures on $ \mathcal M^d$}
\label{urms}

A \emph{rooted Riemannian $d$-manifold} is a pair $(M,p)$ where $M$ is a
Riemannian $d$-manifold and $p\in M$ is a basepoint. We assume that all Riemannian manifolds in this paper are complete and connected.
A \emph{doubly
rooted manifold} is a triple $(M,p,q)$ where $M$ is a Riemannian $d$-manifold
and $p,q\in M$.

\begin{definition}
Let $\mathcal{M}^{d}$ and $\mathcal{M}%
_{2}^{d}$ be the spaces of isometry classes of rooted and doubly rooted Riemannian
$d$-manifolds, endowed with the  smooth topology.
\end{definition}

Recall that a sequence of rooted Riemannian manifolds $(M_{i},p_{i})$ \emph {smoothly converges} to $(M,p)$
if for every $R>0$, there is a $C^{\infty
}$-embedding $$f_{i}:B_M(p,R)\longrightarrow M_{i}$$ such that $f_i(p)=p_i$ and $f_{i}^{\ast}(g_{M_{i}})\to g_{M}$ in the $C^{\infty}$
topology, where $g_{M_{i}}$ and $g_M$ are the  associated Riemannian metrics.  The convergence $(M_i, p_i, q_i) \rightarrow (M,p,q)$ of a sequence of doubly rooted manifolds is the same, except that we require that $f_i (q) = q_i$ when defined. In \S \ref{secmetrizability}, we show that smooth convergence comes from a Polish topology on $\mathcal M^d$. An analogous statement holds for $\mathcal{M}%
_{2}^{d}$.

\begin {example} \label {onecase}Setting $d=1$, $\mathcal M ^1$ is homeomorphic to $(0,\infty]$, since there is a unique rooted $1$-manifold of diameter $x$ for each $0<x\leq \infty$. The space $\mathcal{M}%
_{2}^{1}$ is then naturally homeomorphic to the set $$S=\big\{(x,y) \ \big|\  0 < x < \infty, \ 0 \leq y \leq x \ \  \text{or} \ \ x=\infty, \, 0\leq y < \infty \big\},$$
where $x$ is the diameter of $M$ and $y$ is the distance between the base points $p,q \in M$. Both the left and right projections of $\mathcal M ^1_2$ onto $\mathcal M ^1$ are then the first coordinate projection $S\longrightarrow (0,\infty]$.
\end {example}

Let $\mu$ be a $\sigma$-finite Borel measure on $\mathcal{M}^{d}$.  From $\mu$,  we define two associated Borel measures $\mu_{l}$ and $\mu_{r}$  on $\mathcal{M}_{2}^{d}$,  by setting
\begin {align*}
	\mu_{l}(S) &=\int_{(M,p)\in\mathcal{M}^{d}}
\volume_{M} \left\{  q\in M\mid (M,p,q)\in S\right\} 
d\mu \\
\mu_{r}(S) &=\int_{(M,q)\in\mathcal{M}^{d}} \volume_{M} \left\{  p\in M\mid (M,p,q)\in S\right\}  
d\mu,
\end {align*}
 whenever $S$ is a Borel subset of $\mathcal{M}_{2}^{d}$.  Sometimes, we abbreviate the  above as
$$d\mu_l(M,p,q) = \volume_{M}(q) \, d\mu(M,p), \ \ d\mu_r(M,p,q) = \volume_M(p) \, d\mu(M,q).$$

\begin{definition}\label {unimodularity}
 We say $\mu$  is \emph{unimodular} if $\mu_{l}=\mu_{r}$. When $\mu$  is a probability measure, a $\mu$-random element of $\mathcal M^d$  is a \emph{unimodular random manifold} (URM).
\end{definition}

Unimodular measures were first studied in the context of rooted graphs rather than rooted Riemannian manifolds (see \cite {Benjaminigroup}, \cite {Haggstrominfinite}, \cite {Aldousprocesses}).  In these works, the equality of $\mu_l$ and $\mu_r $ is phrased via the \emph{mass transport principle}, which is the definition we gave in the introduction (Definition \ref{10}). Lewis Bowen \cite{Bowencheeger}  has previously considered unimodularity in the general context of metric-measure spaces; however here we restrict ourselves to the Riemannian setting.

When $d=1$, any Borel probability measure on $\mathcal M ^ d \cong (0,\infty]$ is unimodular, as the measures $\mu_l$ and $\mu_r$ are obtained by integrating $\mu$ against twice the Lebesgue measure on the fiber $[0,x] \subset S$ over $x\in (0,\infty]$.

 For an alternative definition, note that there is an involution \begin {equation}i : \mathcal M_2^d\longrightarrow\mathcal M_2^d, \ \ i(M,p,q) = (M,q,p)\label {involution}\end {equation}
and that $i_*(\mu_l)=\mu_r$. It follows that \emph {a measure $\mu$ on $\mathcal M^d$ is unimodular if and only if either/both of $\mu_l$ and $\mu_r$ are $i$-invariant.}

As mentioned in the introduction, any finite volume Riemannian $d $-manifold $M
$ determines a unimodular measure $\mu$, obtained by pushing forward the
 Riemannian volume $\volume_M$ under the map
\[
M \longrightarrow\mathcal{M }^{d}, \ \ p \mapsto(M, p).
\]
In this case, the measures $\mu_l$ and $\mu_r$ are  both obtained by pushing forward the product measure $\volume_M \times \volume_M$ on $M \times M$ to $\mathcal M ^ 2_d,$ so $\mu $ is unimodular.  
Also, $\mu \mapsto (\mu_l,\mu_r)$ is weak* continuous, so the space of unimodular measures on $\mathcal {M} ^ d $ is closed.  So, more unimodular measures can be constructed as weak* limits.

Here are some other  constructions of unimodular  measures on $\mathcal {M} ^ d $.

\begin {example}[Regular covers]  Suppose that $\pi: N \longrightarrow M$ is a regular Riemannian covering map and $M$  has finite volume.  Then there is a map
$$i_\pi : M \longrightarrow\mathcal M ^ d, \ \ i_\pi(p) = (N,q), \text { where } \pi(q)=p.$$ 
Here, the point is that the isometry class of $(N,q)$ depends only on the projection $p=\pi(q)$, since any two $q$ with the same projection differ by a deck transformation. The push forward of $\volume_M$ under $i_\pi$ is a probability measure $\mu$ on $\mathcal M ^ d $ that is supported on manifolds isometric to $N$. 
For an alternative construction of $\mu$, choose a fundamental domain $F$ for the projection $\pi: N \longrightarrow M$ and push forward the measure $\volume_N / \volume_N(F)$ via the map $q\in F \longmapsto (N,q).$

 To see that $\mu  $ is unimodular, let $\mathcal N_2 = N \times N / \Gamma$, where $\Gamma$ is the group of deck transformations of $\pi$, which acts diagonally on $N\times N$.  We can identify $\mathcal N_2$ with $F \times N$, and give it the measure $\volume_N / \volume_N(F) \times \volume_N$.  Then the map
$$\mathcal N_2 = N \times N / \Gamma\longrightarrow N \times N / \mathrm{Isom}(N) \subset \mathcal M^d_2$$
 is measure preserving, where we consider $\mathcal M^d_2$ with $\mu_l$. On $\mathcal N_2$, the involution $(p,q)\mapsto (q,p)$ is measure preserving, since  for each $\gamma \in \Gamma$ the composition
$$F \times \gamma F \subset F \times N \cong \mathcal N_2  \overset  {(p,q)\mapsto(q,p)}{ \xrightarrow{\hspace*{1.8cm}}} \mathcal N_2  \cong F\times N $$
 is just given by $(x, y)\mapsto (\gamma^{-1}(y),\gamma^{-1}(x))$. So, as this involution on  $\mathcal N_2 $   descends to \eqref{involution} on $\mathcal M^d_2$, the measure $\mu_l$ is invariant under \eqref{involution}, so $\mu$  is unimodular.\label {regcover}
\end {example}

\begin {example}[Restriction to saturated subsets]\label {restriction}
A subset $B\subset\mathcal M^d$ is \emph{saturated} if whenever $(M,p)\in B$ and $q\in M$, then $(M,q)\in B$ as well. Note that saturated Borel subsets of $\mathcal M^d$  form a  $\sigma$-algebra, $\mathcal S$.

If $\mu$ is unimodular and $B$ is a saturated Borel subset of $\mathcal M^d$, then $\mu |_B$  is unimodular as well, since $(\mu_B)_l$ and $(\mu_B)_r$ are just the restrictions of $\mu_l=\mu_r$ to the   set of all $(M,p,q)\in \mathcal M^d_2$ with $(M,p)\in B$.
\end {example}

 Finally, let $X $ be a complete Riemannian manifold with a transitive isometry group $\mathrm{Isom}(X)$.  Up to rooted isometry, the choice of root in $X$ is irrelevant, so we will denote the corresponding point in $\mathcal M ^ d $ by $X $ as well.

\begin {proposition}\label{terminologyprop}
If a Riemannian manifold $X $ has transitive isometry group, the atomic probability measure $\delta_X$ supported on $X \in \mathcal M ^ d $ is unimodular if and only if $\mathrm{Isom}(X)$ is a unimodular Lie group.
\end {proposition}
\begin {proof}
Fix a point $x_0\in X$ and let  $K < G = \mathrm{Isom}(X) $ be the stabilizer of $x_0$, so that we can identify $X = G/K$. Then $(\delta_X)_l$ is supported on $$X_2 := G \, \backslash \, X \times X $$
where $G$ acts diagonally. Since $G$ acts transitively, the natural map
\begin {equation}K \backslash X \overset{[x] \mapsto [(x_0,x)]}{ \xrightarrow{\hspace*{1.7cm}}} X_2.\label {ident}\end {equation}
is a homeomorphism.  With respect to this identification, $(\delta_X)_l$ is just the push forward $\overline \lambda$ of the Riemannian measure $\lambda$ on $X $ to $K \backslash X$.

 Since $K \backslash X=K \backslash G/K$,  the identification \eqref{ident} can also be written as $$K \backslash G / K\overset{KgK\mapsto [(x_0,g(x_0))]}{ \xrightarrow{\hspace*{2.2cm}}} X_2.$$   Conjugating, the involution $[(p,q)] \longrightarrow [(q,p)]$ on $X_2 $  becomes the inversion map
$$\overline i :K \backslash G / K \longrightarrow K \backslash G / K, \ \ \overline i([g])=[g^{-1}].$$
 So, with \eqref{involution} in mind, we want to show that the natural measure $\overline \lambda$ on $K \backslash G / K$ is $\overline i$-invariant if and only if $G $ is unimodular.
 
 Integrate $\lambda $ against the (unique) right $K$-invariant probability measures on the fibers of $G\longrightarrow G/K$; this gives a left Haar measure $\widetilde \lambda$ on $G$.   Then $G $ is unimodular if and only if $\widetilde \lambda$  is invariant under the inversion map $$i : G \longrightarrow G, \ \ i(g)=g^{-1}.$$  By definition, $\widetilde \lambda$  can be  expressed as an integral
 $$\widetilde \lambda = \int_{ KgK \in K \backslash G / K} \eta_{KgK} \, d\overline\lambda ,$$
 where  the fiber measure $\eta_{KgK}$   is the unique probability  measure  on $KgK$ that is $K$-biinvariant.  The action of $i : G\longrightarrow G $ permutes the $\eta_{KgK}$,  which implies that $\widetilde \lambda$ is $i$-invariant if and only if the factor measure $\overline\lambda$ is $\overline i $ invariant.  \end {proof}

As semisimple groups are unimodular, their symmetric spaces $X$ satisfy the assumptions above. One can also see directly that the atomic measure $\delta_X$ is unimodular when $X $ is a model space of constant curvature, i.e.\ when $X=\BR ^n, \BH^n$ or $S^n$.  The measures $(\delta_X)_l$ and $(\delta_X)_r$ are supported on the subset $X_2 \subset \mathcal M_2^d$ consisting of isometry classes of doubly pointed manifolds $(X,p,q)$.  Here, these $(X,p,q)$ are classified up to isometry by $d(p,q)$, which is symmetric in $p,q$. So, the involution $i$ in \eqref{involution} is the identity, and therefore preserves $(\delta_X)_l,(\delta_X)_r$.

\vspace{2mm}

\subsection {The No-Core Principle}
\label {nocoresec}
 This trivial, yet useful, consequence of unimodularity  was mentioned in \S \ref{fingenintro}.

\begin{reptheorem}{nocore}[The No-Core Principle]	 Suppose that $\mu$ is a unimodular probability measure on $\mathcal  M ^ d$ and that $f:\mathcal M^d\longrightarrow \{0,1\}$ is a Borel function. Then for $\mu$-almost every element  $(M,p)\in \mathcal  M ^ d$, we have 
	 $$0<\volume_M \{q\in M \ | \ f(M,q)=1\}<\infty \implies \volume(M)<\infty.$$ 
	\end{reptheorem}

	It is very important here that $\mu$ is a probability measure,  and not just $\sigma$-finite. Otherwise, one could take a fixed Riemannian manifold $N$ with infinite volume and no symmetries, and any finite (nonzero) volume subset $B\subset M$, and define $f(M,p)=1$ if there is an isometry $M \to N$  that takes $p$ into $B$.  As $N$ has no symmetries,  the map $N \longrightarrow\mathcal M^d, \ \ q \longmapsto (N,q)$  is an embedding, so $\volume_N$  pushes forward to a $\sigma$-finite unimodular measure $\mu_N$ on $\mathcal M^d$, and the pair $f$ and $\mu_N$  violates the statement of the theorem.

\begin {proof}
If the theorem fails, then for some $C>0$ the set of all $(M,p)$ such that \begin{equation}
 	0<\volume_M \{q\in M \ | \ f(M,q)=1\}<C \text{ and } \volume(M)=\infty \label{upthere}
 \end{equation}
has positive $\mu$-measure. This set is saturated and Borel, so we may assume after restriction and rescaling that $\mu$ is supported on it, as in Example \ref{restriction}.

 Well, by unimodularity we know that
\begin{align*}\int_{(M,p)\in\mathcal{M}^{d}}
\int_{q\in M} f(M,p) \, d \volume_M \, d\mu
=\int_{(M,p)\in\mathcal{M}^{d}}
\int_{q\in M} f(M,q) \, d \volume_M \, d\mu.
\end{align*}
On the right, the integrand is at most $C$ $\mu$-almost surely, by \eqref{upthere}. So, the right-hand side is finite.  Therefore, the integrand on the left is finite $\mu$-almost surely.  So, for $\mu$-a.e.\ $(M,p)$, we have 
$f(M,p)=0$ by \eqref{upthere}.
This implies that the left side is zero. So, the integrand on the right side is zero $\mu$-a.e., contradicting \eqref{upthere}.
\end {proof}

\subsection{Unimodularity and IRSs}
\label {urmirssec}
In previous work with Bergeron, Gelander, Nikolov, Raimbault and Samet, the authors studied the following group theoretic analogue of URMs, see \cite {Abertgrowth}.  
Let $G$ be a locally compact, second countable topological group and let $\mathrm {Sub}_G$ be the space of closed subgroups of $G $, endowed with the \emph{Chabauty topology}, see \S \ref{Chabautysection}.

\begin {definition}  An \emph {invariant random subgroup} (IRS) of $G $ is a random element of $\mathrm{Sub}_G $ whose law $\mu $ is a Borel measure invariant under conjugation by $G$.  (In an abuse of notation, we will often refer to the law $\mu$ itself as an IRS.)
\end {definition}

Invariant random subgroups supported on discrete groups of unimodular $G$  satisfy a useful group-theoretic unimodularity property.  Fix a Haar measure $\lambda$ for $G$.  If $H$ is a discrete subgroup of $G $, then $\lambda$ pushes forward locally to a Radon measure $\lambda_H$ on the coset space $H \backslash G$.  Let $\coset $ be the set of cosets of closed subgroups of $G $, endowed with its Chabauty topology.

\begin{theorem}[Biringer-Tamuz \cite{Biringerunimodularity}] \label {unimodularIRS} Assume $G$ is unimodular, and $\mu$ is a  Borel probability
  measure on $\subgroup$ such that $\mu$-a.e.\ $H\in\subgroup $ is
  discrete.  Then $\mu $ is an IRS if and only if for every Borel
  function $f \colon \coset \longrightarrow \mathbb R$, we have
  \begin{align*}
 \int_{\subgroup} \int_{H \backslash G}  f(Hg)\, d\lambda_H(Hg) \, d\mu(H) \nonumber  =\int_{\subgroup} \int_{H \backslash G}f(g^{-1}H) \, d\lambda_H(Hg)\, d\mu(H).
  \end{align*}
\end{theorem}

As noted in \cite{Biringerunimodularity}, a more aesthetic version of the equality above is 
\begin {align*} \int_{H \in\mathrm {Sub}_G } \int_{g \in H \backslash G} f (Hg) \, d\lambda_H \, d\mu
= \int_{H \in\mathrm {Sub}_G } \int_{g \in G/H} f (gH) \, d\lambda^H \, d\mu,
\end {align*}
where $\lambda ^ H $ is the measure on $G/H$ obtained by locally pushing forward $\lambda $.  In other words, the `right' measure obtained on $\coset$ by viewing it as the space of \emph{right} cosets and then integrating the natural invariant measures on $H \backslash  G$ against $\mu$ is the same as the analogous `left' measure on $\coset$.  The version given in the theorem is that which we will use here, though.

Suppose now that our unimodular $G $ acts isometrically and transitively with compact stabilizers on a Riemannian $d$-manifold $X $, and write $X = G/K $, where $K$ is a compact subgroup of $G$.   A \emph {$(G,X)$-manifold} is a quotient $H \backslash X$, where $H < G$  acts freely and properly discontinuously.  Let 
\begin {align*}
\mathrm{Sub}^{df}_G &= \{ H \in \mathrm{Sub}_G \  | \ H \text {   acts freely  and properly discontinuously on } X \},\end {align*}
 and let $\mathcal M^{(G,X)} \subset \mathcal M^d$  be the space of all  pointed $(G,X)$-manifolds.

\begin {proposition}[URM vs. IRS, transitive case]\label{urmirsprop}
The continuous map
$$\phi : \mathrm{Sub}^{df}_G \longrightarrow \mathcal M ^{(G,X)}, \ \ \phi (H) = (H \backslash X,[id])$$ induces a (weak* continuous) map
$$\phi_*: \left\{\text {IRSs } \mu \text { of } G \text { with } \mu (\mathrm{Sub}^{df}_G ) = 1 \right\} \longrightarrow\left\{\text {URMs } \nu \text { with } \nu (\mathcal M ^{(G,X)}) = 1 \right\}. $$
If $X $ is simply connected and $G =\mathrm {Isom} (X)$ is the full isometry group of $X $, then $\phi_* $ is a weak* homeomorphism.
\end {proposition}

Recall that a \emph{unimodular random manifold} (URM) is a random element of a unimodular probability measure $\nu$ on $\mathcal M^d$. However, we routinely abuse  terminology  by calling measures IRSs, and  we will similarly call $\nu$ itself a URM.
Note that $\phi$ is surjective, but is not in general injective, since conjugating $H$ by an element of $K$ does not change its image. Continuity of $\phi$ follows from Proposition 3.10 and Lemma 3.7 of \cite{abert2012growth}.

\begin {proof}

Fix a Haar measure $\lambda$ on $G $ normalized so that $\pi_*\lambda= \mathrm{vol_X}$, where $\pi: G \longrightarrow G/K= X$ is the projection.  If $H\in \mathrm{Sub}^{df}_G$ and $\pi_H: H \backslash G \to H \backslash X$ is the natural projection, it follows that $(\pi_{H})_*\lambda_H = \mathrm{vol}_{H \backslash X}.$

\vspace {2mm}

\noindent \it Unimodularity of the image. \rm 
\noindent Let $\mu $ be an IRS with $\mu (\mathrm{Sub}^{df}_G ) = 1 $.  We must show that the measures $(\phi_* \mu)_l $ and $(\phi_*\mu)_r $ in Definition \ref{unimodularity} are equal.  So, let $f:\mathcal M_2 ^ d \longrightarrow \mathbb R$ be a Borel function.  We then define a new function 
\[
\tilde f: \mathrm{Cos}_G^{df} \longrightarrow \mathbb R, \ \ \tilde f (Hg) = f(H \backslash X,[id],[g]),
\]
where $ \mathrm{Cos}_G^{df} $ is the set of cosets of subgroups $H \in \subgroup ^ {df}$.
We then compute: 
\begin {align*}
\integral_{\mathcal M_2 ^ d} \, f \ d (\phi_*\mu)_l & = \int_{(H\backslash X, [id]) \in \mathcal M ^ d} \integral_{[g] \in H \backslash X} f (H \backslash X , [id] , [g]) \ d\volume_{H \backslash X} \, d (\phi_*\mu) \\
& = \int_{H\in \mathrm{Sub}^{}_G} \integral_{Hg \in H \backslash G} \tilde f (Hg) \ d\lambda_H \, d \mu, \\
& =\int_{H \in\mathrm {Sub}_G } \int_{Hg \in H \backslash G} \tilde f ((g^{-1}Hg) g ^ {- 1}) \, d\lambda_H \, d\mu, \\
& =  \int_{(H\backslash X, [id]) \in \mathcal M ^ d} \integral_{[g] \in H \backslash X} f (g ^ {- 1}Hg \backslash X , [id] , [g ^ {- 1}]) \ d\volume_{H \backslash X} \, d (\phi_*\mu) \\
& =  \int_{(H\backslash X, [id]) \in \mathcal M ^ d} \integral_{[g] \in H \backslash X} f (H \backslash X , [g] , [id]) \ d\volume_{H \backslash X} \, d (\phi_*\mu) \\
& =\integral_{\mathcal M_2 ^ d} \, f \ d (\phi_*\mu)_r.
\end {align*}
Here, the first equation is the definition of $d (\phi_*\mu)_l $, keeping in mind that it is enough to integrate over rooted manifolds of the form $(H\backslash X, [id])$. The second and forth equations follow from the normalization of the Haar measure, while the third is Theorem \ref {unimodularIRS}.  The fifth equation reflects the fact that $(H \backslash X , [g] , [id])$ and $(g ^ {- 1}Hg \backslash X , [id] , [g ^ {- 1}])$ are isometric as doubly rooted manifolds.

\vspace {2mm}
\noindent \it The case of the full isometry group. \rm  Assume now that $X $ is simply connected and $G=\mathrm{Isom}(X)$ is the full isometry group of $X $.

We first analyze the fibers of $\phi : \mathrm{Sub}^{df}_G \longrightarrow \mathcal M ^{(G,X)} $.  Conjugate subgroups of $G $ give isometric $X$-quotients, and if two subgroups $H, H'\in \subgroup ^ {df}$ are conjugate by an element of $K $, then the \emph {pointed} manifolds $(H \backslash X,[id]) $ and $ (H' \backslash X,[id])$ are isometric.  Conversely, as $K$ is the full group of isometries of $X$ fixing $[id]$, any based isometry of quotients lifts to a $K$-conjugacy of subgroups, so we have:
\begin {align}\text {fibers of } \phi :  \mathrm{Sub}^{df}_G \longrightarrow \mathcal M ^{(G,X)} \ \ \longleftrightarrow \ \  K\text {-conjugacy classes in } \mathrm{Sub}^{df}_G\label {phifibers}.\end {align}

\vspace {2mm}
\noindent \it Injectivity. \rm Let $\mu$ is an IRS with $\mu (\subgroup^{df})=1$.  By Rohlin's disintegration theorem (see \cite [Theorem 6.2]{Simmonsconditional}), $\mu $ disintegrates as an integral
$$\mu = \int_{(M, p)\in \mathcal M ^{(G,X)}} \eta_{(M, p)} \, d\phi_*(\mu), $$
where $\eta_{(M, p)} $ is a Borel probability measure on the preimage $\phi ^ {- 1} (M, p) $.  The $K $-action by conjugation on $\mathrm{Sub}^{df}_G$ leaves the $\phi $-fibers invariant \eqref{phifibers} and preserves $\mu $, so it must preserve $\phi_*\mu $-a.e.\ fiber measure $\eta_{(M, p)} $.   As each $\phi $-fiber is a $K$-homogeneous space, each $\eta_{(M, p)}$ is just the push forward of the unique Haar probability measure on $K $.  Therefore, $\mu $ can be recovered by integrating the canonical measures $\eta_{(M, p)}$ against $\phi_* \mu$.   So, $\phi_*$ is injective.

\vspace {2mm}
\noindent \it Surjectivity. \rm If $\nu$ is an URM with $\nu (\mathcal M ^{(G,X)}) = 1 $, define a measure on $\subgroup $ by
\begin{equation}\label {inverseirs}\mu = \int_{(M, p)\in \mathcal M ^{(G,X)}} \eta_{(M, p)} \, d\nu, \end{equation}
where as above each $\eta_{(M, p)}$ is the unique $K$-invariant probability measure on $\phi ^ {- 1}(M,p)$.  Then $\phi_*(\mu)=\nu$, and we claim that $\mu $ is an IRS.

 Our strategy will be to use the unimodularity of $\nu$ to establish the equality in Theorem \ref {unimodularIRS} (the mass transport principle for IRSs). Consider the map
$$\phi_2 : \coset^{df} \longrightarrow \mathcal M^{(G,X)}_2, \ \ \phi_2 (Hg) = (H \backslash X,[id],[g]),$$
where $\mathcal M^{(G,X)}_2$ is the space of doubly rooted $X $-manifolds.  The fibers of $\phi_2 $ are exactly the $K\times K $-orbits $(Hg)^{K \times K} $ in $\coset ^ {df} $, where the action is defined as 
\begin {align}K\times K  \curvearrowright \coset, \ \ \ (k,k')\cdot Hg :=(kHk^{-1})kgk'.\label {kkaction}\end {align}

We define $\eta_{(M,p,q)}$ to be the unique $K\times K $-invariant probability measure on the fiber $\phi_2 ^ {- 1} (M,p,q) $, i.e.\ the push forward of the Haar measure on $K\times K$ under the conjugation action.  
\begin {claim} For $(M, p) \in \mathcal M ^{(G,X)}$, we have $ \lambda_H \, d\eta_{(M,p)}(H) = 
 \eta_{(M,p,q)} \, d\mathrm{vol}_{M}(q) . $ \label {measureequality}
\end {claim}
We will prove the claim below, but first we use it to prove that $\mu $ is an IRS, by deriving the equality in Theorem \ref {unimodularIRS} from the unimodularity of $\nu$.  Suppose that $ f: \coset^{df} \longrightarrow \mathbb R$ is a Borel function, and define a new function
$$(\phi_2)_* f : \mathcal M ^{(G,X)}_2  \longrightarrow \mathbb R, \ \ (\phi_2)_* f (M, p,q) := \int_{Hg \in \phi_2 ^ {- 1}(M,p,q)} f (Hg) \, d\eta_{(M,p,q)}.$$

We first compute the left side of the equality in Theorem \ref {unimodularIRS}.
\begin {align*}
 & \ \ \ \  \int_{\subgroup} \int_{H \backslash G}  f(Hg)\, d\lambda_H(Hg) \, d\mu(H) \nonumber  \\
 &=\int_{(M,p)\in \mathcal M ^ d} \int_{H \in \phi ^ {- 1}(M,p)} \int_{H \backslash G}  f(Hg)\ d\lambda_H(Hg) \, d\eta_{(M,p)} \, d\nu \nonumber  \\
 &=\int_{(M,p)\in \mathcal M ^ d} \int_{q\in M} \int_{Hg \in \phi ^ {- 1}(M,p,q)} f(Hg)\ d\eta_{(M,p,q)} \, d\mathrm{vol}_{M} \, d\nu \nonumber \\
& =\int_{\mathcal M_2 ^{(G,X)}} (\phi_2)_* f \, d\nu_l \nonumber
\end {align*}
and using a similar argument, we compute the right side:
\begin {align*}
 & \ \ \ \  \int_{\subgroup} \int_{H \backslash G}  f(g^{-1}H)\, d\lambda_H(Hg) \, d\mu(H) \nonumber  \\
 &=\int_{(M,p)\in \mathcal M ^ d} \int_{q\in M} \int_{Hg \in \phi ^ {- 1}(M,p,q)} f(g^{-1}H)\ d\eta_{(M,p,q)} \, d\mathrm{vol}_{M} \, d\nu \nonumber \\
 &=\int_{(M,p)\in \mathcal M ^ d} \int_{q\in M} \int_{Hg \in \phi ^ {- 1}(M,q,p)} f(Hg)\ d\eta_{(M,q,p)} \, d\mathrm{vol}_{M} \, d\nu \nonumber \\
& =\int_{\mathcal M_2 ^{(G,X)}} (\phi_2)_* f \, d\nu_r \nonumber. 
\end {align*}

So, the unimodularity of $\nu$ implies that $\mu $ is an IRS.

\vspace {2mm}

\noindent \it Weak* homeomorphism. \rm  Finally,  recall that \eqref{inverseirs}  defines an inverse for $\phi_*$. Weak* continuity of the inverse will follow if we show that the map
$$\mathcal M^{(G,X)} \longrightarrow \mathcal P(\subgroup), \ \ (M,p) \longmapsto \eta_{(M,p)}$$
is continuous,  where $\mathcal P(\subgroup)$  is the space of Borel probability measures on $\subgroup$,  considered with the weak* topology. However, $ \eta_{(M,p)}$  is the unique $K$-invariant measure on $\phi^{-1}(M,p)$, and if $(M_i,p_i) \to (M,p)$, we can  pass to a subsequence so that $ \eta_{(M_i,p_i)}$ converges. The limit must be supported on $\phi^{-1}(M,p)$, and is $K$-invariant since its approximates are, so must be $ \eta_{(M,p)}$.
\end {proof}

 Finally, we promised to prove Claim  \ref {measureequality} during the proof above:

\begin {proof} [Proof of Claim \ref {measureequality}]
Let $(M, p, q) \in \mathcal M ^{(G,X)}_2$ and let $Hg \in \coset$ such that $\phi_2(Hg)= (M, p, q) $.  
By \eqref {kkaction}, we have a commutative diagram
$$\xymatrix{
K \times K  \ar[d]_{\pi_l}  \ar[rrrr]^{(k,k') \longmapsto (kHk^{-1})kgk'}& & & & \coset    \ar[d]^{\pi_r}\\
 \sfrac{K \times K }{N_K(H) \times K} \ar@{^{(}->}[rrrr]^{[(k,k')]\longmapsto kHk^{-1}} &   & & &  \subgroup,}$$
where $\pi_r:\coset \longrightarrow \subgroup, \ \pi(Hg )= H$.  As we had previously defined  $\eta_{(M,p)}$ as the $K$-invariant probability measure on $H^K \subset \subgroup$, the diagram shows that
 $$(\pi_r)_*(\eta_{(M, p, q)}) =\eta_{(M,p)}. $$

The Haar probability measure on $K\times K $ disintegrates under $\pi_l$ as an integral of invariant probability measures on the cosets of $N_K(H) \times K$ against the pushforward measure on $\sfrac{K \times K }{N_K(H) \times K}$.  Here, the coset $(k,1)N_K(H) \times K$ has a measure invariant under its stabilizer, which is $N_K(kHk^{-1}) \times K$.
 This disintegration pushes forward to a $\pi_r$-disintegration of $\eta_{(M,p,q)}$:
\begin {equation}\eta_{(M,p,q)} = \int_{F \in \subgroup} \eta_{(M,p,q)}^{F} \ d\eta_{(M, p)},\label {eta1}\end {equation}
where $\eta ^{F}_{(M, p, q)}=0$ unless $F$ is a conjugate of $H$, in which case $\eta ^{F}_{(M, p, q)}$ is an invariant probability measure on the $N_K(F) \times K$-orbit in $ \coset$ obtained by intersecting $\phi_2 ^ {-1} (M, p, q) $ with $F \backslash G$.  

Now fix $(M,p) \in \mathcal M ^{(G,X)}$, let $H \in \phi^{-1}(M,p)$ and fix an isometric identification of $(H \backslash G /K,[id])$ with $(M,p)$.  The fibers of the composition
$$\xymatrix@R-2pc{
    H \backslash  G \ar[r] & M  \ar[r] & \mathcal M  ^{(G,X)}_2\\
     Hg \ar@{|->}[r]     & [g]\in H \backslash G /K \ar@{|->}[r]     & (H \backslash G /K,[id],[g]).
}$$
are exactly the $N_K(H)\times K$-orbits in $H \backslash  G$.  As $\lambda_H $ is invariant under the action of $N_K(H)\times K$, it disintegrates as an integral of invariant probability measures on these orbits against its pushforward under the composition.  Under the first map, $\lambda_H $ pushes forward to $\volume_M$, so we may write instead:
\begin {equation}\lambda_H= \int_{q\in M} \eta^H_{(M,p,q)} \, d\mathrm{vol}_M .\label {eta2}\end {equation}

Combining Equations \eqref{eta1} and \eqref{eta2}, we can now prove the claim:
\begin {align*}
\int_\subgroup \lambda_H \, d\eta_{(M,p)} & = \int_\subgroup \int_M \eta ^ H_{(M, p, q)} \, d\volume_M (q) \, d\eta_{(M, p)} \\
& =  \int_M \int_\subgroup \eta ^ H_{(M, p, q)}  \, d\eta_{(M, p)} \, d\volume_M (q)  \\
& =  \int_M \eta_{(M, p, q)} \, d\volume_M (q). \qedhere
\end {align*}
\end {proof}

\vspace{2mm}

We now construct URMs from IRSs of discrete groups. Suppose that $G $ is a discrete group that acts freely and properly discontinuously on a Riemannian $d$-manifold $X $ and that the quotient $G \backslash  X $ has finite volume.  There is a map
$$\mu \text { an IRS of } G \ \ \longmapsto \ \ \overline\mu \text { a probability measure on } \mathcal M^{(G,X)},$$  where a $\overline\mu $-random element of $\mathcal M ^ d $ has the form $(H \backslash  X, [x])$, where we first take $x\in X$ to be an \emph {arbitrary} lift of a random point in $ G \backslash  X $ and then choose $H\in\mathrm{Sub}_G$ $\mu $-randomly.    The conjugation invariance of $\mu $ makes the measure $\overline\mu $ well-defined despite the arbitrary choice of lift.  Alternatively, consider 
\[
B= (\mathrm{Sub}_G \times X)/\Gamma, \text { where } (H, x) \overset {\gamma} {\longmapsto} (\gamma  H\gamma^ {- 1}, \gamma \cdot x).
\]
Then $B $ is a $\mathrm{Sub}_G$-bundle over $G \backslash  X $, and each of its fibers has an identification with $\mathrm{Sub}_G$ that is canonical up to conjugation.  So, as $\mu $ is conjugation invariant there is a well-defined probability measure $\mu_B $ on $B $ obtained as the integral of $\mu$ on each fiber against the (normalized) Riemannian volume of $G \backslash  X $.  The map
\[
\mathrm{Sub}_G \times X \longrightarrow \mathcal M ^ d, \ \ (H,x) \longmapsto (H \backslash X, [x])
\]
factors through the $\Gamma$-action to a map $B \longrightarrow \mathcal M ^ d$, and
$\overline\mu $ is the push forward of $\mu_B $ under this map.

\begin {proposition}[IRS $\implies$ URM, discrete case]\label {urmirsdisc}If $\mu $ is an IRS of $\mathrm{Sub}_G$, then $\overline \mu$ is an URM.
\end {proposition}
\begin {proof}
Pick a Borel fundamental domain $D\subset X $, i.e.\ a Borel subset such that
\begin {enumerate}
\item $\volume_X (D \cap gD) = 0 $ for every $g\in G $,
\item $\volume_X (X \setminus \cup_{g\in G} gD) = 0 $.
\end {enumerate}
It follows that $\volume_X = \sum_{g\in G} g_* (\volume_D)$ and moreover that if $H\in \mathrm{Sub}_G$, then
\begin {align}\volume_{H \backslash X} = \sum_{Hg\in H \backslash  G} (\pi\circ g)_* (\volume_D), \label{6}\end {align}
where $\pi : X \longrightarrow H \backslash X$ is the quotient map.
We let $\hat\mu $ be the push forward to $\mathcal M ^d $  of $\mu \times \volume_D $ under the function
\begin {align*}\mathrm{Sub}_G \times D \longrightarrow \mathcal M ^ d, \ \  \ \ (H, x) \longmapsto (H\backslash   X, [x]). \end {align*}
This $\hat\mu$ is the scale by $\volume_X (D)$ of our $\overline\mu $ above.  For simplicity of notation, we show that $\hat\mu$ is unimodular instead.

We must show that $\hat\mu_l =\hat\mu_r $, so let $f:\mathcal M_2 ^ d \longrightarrow \mathbb R$ be a Borel function.  We lift $f $ to a function $\tilde f: \mathrm{Cos}_G  \longrightarrow \mathbb R$ by letting 
\[
\tilde f (Hg) = \int_{(x,y) \in D}  f \left ( \,H \backslash X, \, [x],\,  [gy]\, \right) \ d\volume_D^2,
\]
Note that $[gy] \in H \backslash  X $ only depends on the coset $Hg$.  We now compute: 
\begin {align}
& \ \ \ \ \int_{\mathcal M_2 ^ d} f \, d\hat\mu_l \nonumber\\ 
&= \int_{(H \backslash  X,[x]) \in \mathcal M ^ d} \int_{[y] \in H \backslash X} f \left ( \,H \backslash  X, \, [x],\, [y]\, \right)  \, d\mathrm{vol}_{X/H} \, d\hat\mu  \nonumber\\
 & = \int_{(H \backslash X,[x]) \in \mathcal M ^ d} \sum_{Hg\in H \backslash G} \, \int_{x\in D}f \left ( \,H \backslash X, \, [x],\,  [gy]\, \right) \ d\volume_D  \, d\nu_H \, d\hat\mu  \label{7}\\
&= \int_{H \in\mathrm {Sub}_G } \sum_{Hg\in H \backslash G} \, \int_{(x,y) \in D^ 2} f \left ( \,H \backslash X, \, [x],\,  [gy]\, \right)  \ d\volume_D^ 2 \, d\nu_H \, d\mu \nonumber\\
& = \int_{H \in\mathrm {Sub}_G } \sum_{Hg\in H \backslash G} \tilde f (Hg) \ d\nu_H \, d\mu \nonumber \\
& = \int_{H \in\mathrm {Sub}_G } \sum_{Hg\in H \backslash G} \tilde f ((g^{-1}Hg) g^{-1}) \ d\nu_H \, d\mu  \label{8}\\
&= \int_{H \in\mathrm {Sub}_G } \sum_{Hg\in H \backslash G} \, \int_{(x,y) \in D^ 2} f \left ( \,g^{-1}Hg \backslash X, \, [x],\,  [g^{-1}y]\, \right)  \ d\volume_D^ 2 \, d\nu_H \, d\mu \nonumber\\
&= \int_{H \in\mathrm {Sub}_G } \sum_{Hg\in H \backslash G} \, \int_{(x,y) \in D^ 2} f \left ( \,H \backslash X, \, [gx],\,  [y]\, \right)  \ d\volume_D^ 2 \, d\nu_H \, d\mu \label{9}\\
&=\int_{(H \backslash  X,[y]) \in \mathcal M ^ d} \int_{[x] \in H \backslash X} f \left ( \,H \backslash  X, \, [x],\, [y]\, \right)  \, d\mathrm{vol}_{X/H} \, d\hat\mu \label{10}\\
&=\int_{\mathcal M_2 ^ d} f \, d\hat\mu_r.\nonumber
\end {align}
Above, \eqref{7} and \eqref{10} follow from \eqref{6}, while \eqref{8} is Proposition \ref{unimodularIRS}. Line \eqref{9} uses the fact that $ \left( \,g^{-1}Hg \backslash X, \, [x],\,  [g^{-1}y]\, \right)$ and $\left( \,H \backslash  X, \, [gx],\, [y]\, \right)$ are isometric as doubly rooted manifolds.\end {proof}

\vspace {2mm}

\section {Measures on Riemannian foliated spaces}

\label{foliatedsec}
A \emph {foliated space} with tangential dimension $d$ is a separable metrizable space $X$ that has an atlas of charts of the form 
\[
\phi_\alpha: U_\alpha  \longrightarrow L_\alpha \times Z_\alpha , 
\]
 where each $L_\alpha \subset \mathbb{R}^d$ is open and each $Z_\alpha $ is a separable, metrizable space.   Transition maps must preserve and be smooth in the horizontal direction, with partial derivatives that are continuous in the transverse direction. The horizontal fibers piece together to form the \emph {leaves} of $X$.  See \cite{Candelfoliations} and \cite {Mooreglobal} for details. A foliated space $X$ is \emph {Riemannian} if each of its leaves has a smooth, complete Riemannian metric, and if these metrics \emph {vary smoothly in the transverse direction}, in the sense that the charts $\phi_\alpha$ can be chosen so that if $t_n \to t \in Z_\alpha$, the induced Riemannian metrics $g_{t_n}$ on $L_\alpha$ converge smoothly to $g_t$.

We are interested in measures on a Riemannian foliated space $X $ that are formed by integrating $\mathrm{vol} $ against a `transverse measure'.  To this end, suppose that $\mathcal U =\{( U_\alpha,\phi_\alpha)\} $ is a countable atlas of charts as above and let $Z = \cup_\alpha Z_\alpha $ be the associated `transverse space'.   An \emph {invariant transverse measure} on $X $ is a $\sigma $-finite measure on $Z$ that is invariant under the \emph {holonomy groupoid} of $\mathcal U $.  Here, the holonomy groupoid is that generated by homeomorphisms between an open subset of some $Z_\alpha $ and an open subset of some $Z_\beta $ that are defined by following the leaves of the foliation (see \cite {Candelfoliations}).  The reader can verify that if $\mathcal U$ and $\mathcal U' $ are countable atlases associated to a foliated space $X $, there is a 1-1 correspondence between the invariant transverse measures of $\mathcal U $ and those of $\mathcal U' $.

If $\lambda $ is an invariant transverse measure on a Riemannian foliated space $X$, one can locally integrate  the Riemannian measure $\mathrm{vol}$ against $\lambda$ to give a measure $\mu$ on $X$, specified by writing $d\mu=\mathrm{vol} \, d\lambda$.  For a precise definition, let $\phi_\alpha: U_\alpha  \longrightarrow L_\alpha \times Z_\alpha$  be a foliated chart and define a measure $\mu_\alpha $ on $U_\alpha $ by  the formula
\[
\mu_\alpha  ( E ) = \int_{x \in Z_\alpha} \mathrm{vol} \big(E \cap \phi^{-1}(L_\alpha \times x) \big) \ d\lambda.
\]
Then if $\{f_\alpha\}$ is a partition of unity subordinate to our atlas, we define 
\[\mu= \sum_\alpha f_\alpha \cdot \mu_\alpha .
\]
Using holonomy invariance, one can check that the measure $\mu $ does not depend on the chosen partition of unity. 

Measures $\mu$ on a Riemannian foliated space $X $ satisfying $d\mu =\volume \, d\lambda$ for some  $\lambda $ are usually called \emph {completely invariant}.  Actually, complete invariance just ensures that when $\mu$ is disintegrated locally along the leaves of the foliation, Lebesgue measure is recovered in the tangent direction; that is, a transverse measure $\lambda$ is automatically holonomy invariant whenever the measure $\volume  \, d\lambda$ on the ambient space is well-defined (see \cite{Connessurvey}).

\begin{theorem}
Suppose that $X $ is a Riemannian foliated space and $\mu $ is a $\sigma$-finite Borel measure on $X$.  Then the following are equivalent:
\begin {enumerate}
\item $\mu$ is completely invariant.
\item $\mu $ is unimodular,  as defined in Equation \eqref{folmtp} of \S \ref{uniintro}.  
\item $\mu $ lifts uniformly to a measure $\tilde \mu$ on the unit tangent bundle $T^1X$ that is invariant under geodesic flow, see \eqref{liftu} below.
\end {enumerate}
If the leaves of $X$ have bounded geometry\footnote{ This condition is needed only in $5) \implies 1)$, in order to invoke a theorem of Garnett~\cite{Garnettfoliations}. It means that  there is some  uniform $K $ such that every point $x\in X$ lies in a smooth coordinate patch for its leaf that has derivatives up to order $3$  bounded by $K$, see \cite{Garnettfoliations}.}, then 1)  -- 3) are equivalent to
\begin {enumerate}
\item[\normalfont 4)] $\int_X \mathrm{div}( Y ) \, d\mu = 0$ for every vector field $Y$ on $X $ with integrable leaf-wise divergence.
\item[\normalfont 5)] $\int_X f \cdot \Delta g \, d\mu = \int_X \Delta f\cdot g \, d\mu $ for all continuous functions $f, g: X \longrightarrow \BR $ that are $C ^ 2 $ on each leaf of $X $. 
\end {enumerate}\label{foliations}
\end {theorem}

 The meaning of 3)  was explained in the introduction, but briefly, the leaf-wise unit tangent bundle $T^1 X$ maps onto $X$, and the fibers $T^1_p X$ are round spheres. If $\omega_p$ is the Riemannian measure on the fiber $T^1_x X$, then  we can define\begin{equation} d\tilde \mu = \omega_p \, d\mu,\label{liftu}
\end{equation}
so $d\tilde \mu$ is  a measure on $T^1 X$.
The geodesic flows on  the unit tangent bundles  of the leaves of $X$ then  piece together to a well-defined \emph{geodesic flow} on $T^1 X$,  and 3)  says that  this flow leaves the measure $\tilde \mu$  invariant.

As discussed in the introduction, this  result  may be particularly interesting to those familiar with unimodularity in graph theory.  Condition 3)  is similar to the `involution invariance' characterization of unimodularity of Proposition 2.2 in \cite{Aldousprocesses}. Also, in analogy with 5), the `graphings' of \cite {Gaboriaumeasurable} can be characterized via the self adjointness of their Laplacian.  See \cite{Hatamilimits} for one direction; the other direction follows from the arguments in \cite[Proposition 18.49]{Lovaszlarge}.

The equivalence $1) \Leftrightarrow 4)$ is well-known as a consequence of work of Lucy Garnett \cite {Garnettfoliations}, and a version with slightly different hypotheses on the foliated space appears in a recent paper of Catuogno--Ledesma--Ruffino \cite{Catuognofoliated}. However, we include the very brief proof below.

\vspace{2mm}

\begin {proof}[Proof of $1) \implies 2)$]
Suppose that  $d\mu=\mathrm{vol} \, d\lambda$ for some $\lambda$.  Let \[\phi_i : U_i \longrightarrow L_i \times Z_i, \ \ i= 1, 2, \] be two foliated charts for $X$ and assume that there is a homeomorphism $\gamma: Z_1 \longrightarrow Z_2 $ in the holonomy pseudogroup.  We first check that $\mu_r= \mu_l$ on the set $X_{\phi_1,\phi_2,\gamma} \subset X \times X $ defined by
\[
X_{\phi_1,\phi_2,\gamma} = \{ (x_1, x_2) \in U_1 \times U_2 \, | \,  \phi_i(x_i)= (l_i,z_i) \text { where } \gamma (z_1) = z_2 \}.
\]
For a subset $S \subset X_{\phi_1 ,\phi_2, \gamma}$, we then calculate
\begin {align*}
\mu_l (S) & = \int_{z_1 \in Z_1} \int_{x \in \phi_1^{-1}(L_1 \times \{z_1\})} \int_{y \in \phi_2^{-1}(L_2 \times \gamma(z_1))}  1_S(x,y)  \ d\mathrm{vol}\, d\mathrm{vol}\, d\lambda \\
&=\int_{z_2 \in Z_2}  \int_{y \in \phi_2^{-1}(L_2 \times \{z_2\})} \int_{x\in \phi_1^{-1}(L_1 \times \{\gamma^{-1}(z_2)\})} \  1_S(x,y) \ d\mathrm{vol}\, d\mathrm{vol}\, d\lambda \ \ (*) \\ 
&=\mu_r(S),
\end {align*}
Above, $(*)$ follows from a change of variables, the invariance of $\lambda $ under the holonomy pseudogroup and Fubini's Theorem.
Now, both of the measures $\mu_r $ and $\mu_l $ are supported on the equivalence relation $\mathcal R\subset X \times X $ of the foliation.  However, we claim that $\mathcal R $ can be covered by a countable number of the Borel subsets $X_{\phi_1,\phi_2,\gamma}$, which will prove the claim.  First, the separability of $X $ guarantees that $X \times X $ can be covered by a countable number of open sets $U_1 \times U_2$ with $\phi_i : U_i \longrightarrow L_i \times Z_i$ foliated charts.  If a pair of points with coordinates $(l_i,z_i) \in L_i \times Z_i, \ i = 1, 2 $ determines an element of $( U_1 \times U_2 ) \cap  \mathcal R ,$ then $z_2 = \gamma (z_1)$ for some holonomy map $\gamma $.  The set of germs of holonomy maps taking a given $z_1 \in Z_1$ into $Z_2$ is countable, so as $Z_1$ is separable, a countable number of domains and ranges of holonomy maps suffice to cover $ ( U_1 \times U_2 ) \cap  \mathcal R $.
\end{proof}

\vspace {2mm}

\begin {proof}[Proof of $2) \implies 3)$]
Suppose $\mu $ is a unimodular measure on $ X $ and let $\tilde \mu $ be the induced measure on the foliated space $T^1X$.  Each leaf of $T^1X$ is the unit tangent bundle of a leaf of $X $ and the tangential Riemannian metric is the Sasaki metric.  The Riemannian volume on each leaf of $T^1X$ is then the fiberwise product of $\mathrm{vol}$ with the Lebesgue measures $\omega_x$ on the tangent spheres $T^1_x X$.

First, note that $\tilde \mu $ is unimodular.   For $ T ^ 1 X \times T^1 X$ fibers over $X\times X$ and $d\tilde\mu_l=d\omega_y \, d\omega_x \, d\mu_l$ while $d\tilde\mu_r =d\omega_y \, d\omega_x \, d\mu_r,$ so the fact that $\mu_l=\mu_r$ implies that $\tilde \mu_l = \tilde \mu_r$.  Geodesic flow $\phi_t$ lifts to a map $\tilde \phi_t = (\phi_t,id)$ on $T^1 X \times T ^ 1 X$; as Liouville measure is geodesic flow invariant, the measure $\tilde \mu_r$ is clearly $\tilde \phi_t$-invariant.  As $\tilde \mu$ is unimodular, this implies that $\tilde \mu_l$ is $\tilde\phi_t $-invariant.  But under the first coordinate projection $T^1 X \times T ^ 1 X \longrightarrow T ^ 1 X$, $\tilde \mu_l$ pushes forward to $\tilde \mu$ and $\tilde \phi_t $ descends to $\phi_t$, so it follows that $\tilde \mu$ is $\phi_t $ invariant.
\end {proof}

\vspace {4mm}

Before proving that $3) \implies 1)$, we need the following lemma.

\begin {lemma}\label {flowinvariant}
Suppose that $U$ is an open subset of a Riemannian manifold $M$ and denote the geodesic flow on $T^1M$ by $g_t$.  Let $\mu $ be a Borel measure on $U $ and let $d\tilde\mu =\omega_p \, d\mu$ be the lifted measure on $T ^ 1 U $, as in \eqref{liftu}.   Suppose that for all Borel subsets $S \subset T^1U $ and all $t \in \mathbb R $ with $g_t(S) \subset T^1U$ we have $\tilde\mu (g_t (S)) = \tilde\mu(S)$.  Then $\mu $ is a scale of the Riemannian measure on $U $.
\end {lemma}

The first part of this proof was shown to us by Nir Avni.

\vspace{2mm}

\begin {proof}
We first prove that $\mu $ is absolutely continuous with respect to the Riemannian measure $\lambda $ on $U$, so let $S \subset U$ be a set of $ \lambda $-measure zero, and let $\tilde S $ be the set of all pairs $(p,v)\in T^1U$ where $p\in S$.  After subdividing $S$ into countably many pieces, we may assume that
 there is some $\epsilon >0 $ such that for all $p\in S$, we have $B(p,\epsilon) \subset U $ and $\epsilon < \mathrm{inj}_M(p)$, where $\mathrm{inj}_M(p)$ is the injectivity radius of $M $ at $p $.    

Choose a probability measure $\nu$ supported in $(0,\epsilon)$ that is absolutely continuous with respect to Lebesgue measure.  For each $p \in S$, define a map
\[
\tilde\phi_p : T^1_p  U \times (0,\epsilon) \longrightarrow T^1U, \ \ \tilde\phi_p(v,t) = g_{t}(v),
\]
Note that as $B(p,\epsilon) \subset U $ the image of the map does in fact lie in $U $.  We then define a measure $\lambda_p $ on $U$ via the formula
$$\lambda_p = (\pi \circ \tilde\phi_p)_*(\omega_p \times \nu),$$ where $\pi:T ^ 1 U \longrightarrow U$ is the projection map.  As $\epsilon < \mathrm{inj}_M(p)$, the map $\pi \circ \tilde\phi_p$ is a diffeomorphism onto its image, so the pushforward $\lambda_p $ is absolutely continuous with respect to the Riemannian measure $ \lambda$ on $U$.  Then we have
\begin {align}
\mu(S) = \tilde\mu (\tilde S) &= \int_{t \in (0,\epsilon)} \tilde\mu( g_{-t}(\tilde S) ) \, d\nu \label{first}\\
&= \int_{t \in (0,\epsilon)} \int_{p \in U} \int_{v \in T^1_pU} 1_{ g_{-t}(\tilde S) } \ d\omega_p \, d\mu \,  d\nu \nonumber \\
&= \int_{p \in U} \left (\int_{t \in (0,\epsilon)}  \int_{v \in T^1_pU} 1_{ g_{-t}(\tilde S) } \  d\omega_p \, d\nu\right ) \, d\mu  \nonumber
\end {align}
\begin {align}
\ \ \ \ \ \ \ \ \ \ &= \int_{p \in U} \omega_p \times \nu \left (\{(v,t) \ | \ g_t(v) \in \tilde S \} \right)\ d\mu \nonumber \\
&= \int_{p \in U} \lambda_p(S) \ d\mu \label {second}\\
&= \int_{p \in U} 0 \ d\mu \ \ = \ \ 0. \label {third}
\end {align}
Here, (\ref {first}) comes from the $g_t$-invariance of $\tilde \mu$ and the fact that $\nu$ is a probability measure.  Equation (\ref {second}) follows since $\tilde S $ consists of \emph {all} unit tangent vectors lying above points of $S $ and the projection $\pi$ is injective on the image of $\tilde\phi_p $.  Finally, equation (\ref {third}) is just the fact that $\lambda_p $ is absolutely continuous to the Riemannian measure $\lambda $ on $U $, with respect to which $S$ has measure $0$.  This shows that $\mu$ is absolutely continuous with respect to $\lambda $, which also implies that $\tilde\mu $ is absolutely continuous with respect to the Liouville measure $\tilde \lambda $.

To show that $\mu $ is a scalar multiple of $\lambda, $ consider the commutative triangle
$$\xymatrix{ T ^ 1U \ar [rr]^ {\frac {d\tilde\mu} {d\tilde\lambda}}   \ar [dd] _ -{\pi      }       &    &         \mathbb R \\ \\
U \ar[rruu]_-{\frac {d\mu} {d\lambda}} &  &\\ },$$
where $\frac {d\tilde\mu} {d\tilde\lambda}$ and $\frac {d\mu} {d\lambda}$ are the Radon-Nikodym derivatives.  Since any two points $p,q$ in $U $ can be joined by a geodesic in $M$, there are unit tangent vectors $(p,v)$ and $(q,w)$ with $g_t(p,v) = (q,w)$ for some $t $.  But since both $\tilde\mu$ and $\tilde\lambda$ are geodesic flow invariant, $$\frac {d\mu} {d\lambda}(p)=\frac {d\tilde\mu} {d\tilde\lambda}(p, v) = \frac {d\tilde\mu} {d\tilde\lambda}(q,w) = \frac {d\mu} {d\lambda}(q).$$  It follows that $\mu $ is a scalar multiple of $\lambda $.
\end {proof}

\vspace{2mm}

\begin {proof}[Proof of $3) \implies 1)$]
Let $\phi : U \longrightarrow L \times Z$ be a foliated chart for $X$.  The restriction $\mu |_U $ then disintegrates as $d\mu | _U = \eta_z \, d\nu,$ where 
\begin {itemize}
\item $\nu$ is the pushforward of $\mu$ under the projection $L \times Z \longrightarrow Z$, and
\item  each $\eta_z$ is a Borel probability measure on $L \times \{z\}$.  
\end {itemize}
The map $z \mapsto \eta_z$ is Borel, in the sense that for any Borel $B \subset L \times Z$ we have that $z \mapsto \eta_z(B)$ is Borel.  

Consider now the foliated chart $\tilde\phi : T^1U \longrightarrow T^1 L \times Z$ for $T ^ 1 X $.  The lifted measure $d\tilde\mu =  \omega_p \, d\mu$ then disintegrates as $\omega_{(z,l)} d\eta_z \, d\nu$.
As $\tilde\mu $ is invariant under the geodesic flow $g_t : T^1 X \longrightarrow T ^ 1 X$, it follows that for $\nu$-almost all $z \in Z$, the measure $\omega_{(z,l)} d\eta_z$ is invariant under the geodesic flow of $L \times \{z\}$, regarded as an open subset of its leaf in $X $.  Thus, by Lemma \ref {flowinvariant}, the probability measure $\eta_z = \frac 1{ \mathrm{vol}_z(L \times\{z\})} \mathrm{vol}_z$ for $\nu$-almost all $z \in Z$.  Since this is true within every $U$, there is a holonomy invariant transverse measure $\nu'$, defined locally by $\nu' = \mathrm{vol}_z(L \times\{z\}) \nu$, with $\mu = \mathrm{vol} \, d\nu'.$  This proves the claim.
\end {proof}

\vspace{2mm}

\begin {proof} [Proof of $1 \implies 4$]
Assume that $d\mu = \mathrm {vol} \, d\lambda $ and that $Y $ is a continuous vector-field on $X$ with integrable divergence on each leaf.  Decomposing $Y $ using a partition of unity, we may assume that $Y $ is supported within some compact subset of the domain of a foliated chart $\phi : U \longrightarrow L \times Z$. Then 
\begin {align*} \int_X \mathrm{div}(Y) \, d\mu &= \int_{z \in Z} \int_{ L \times \{z\}} \mathrm{div}(Y) \ d\mathrm {vol_z} \, d\lambda \\
&= \int_{z \in Z} 0 \, d\lambda =0,
\end {align*}
by the divergence theorem applied to each leaf $L \times \{z\}$.
\end {proof}

\vspace{2mm}

\comment{ and    Using a partition of unity, we may assume that $f$ and $g $ are both supported within a compact subset of the domain of some foliated chart $\phi : U \longrightarrow L \times Z$.  Then
\begin {align*} \int_X f\cdot \Delta g \, d\mu &= \int_{z \in Z} \int_{ L \times \{z\}} f \cdot \Delta g \ d\mathrm {vol_z} \, d\lambda \\
&= \int_{z \in Z} \int_{ L \times \{z\}} \Delta f\cdot g \ d\mathrm {vol_z} \, d\lambda \\
&=\int_X \Delta f\cdot g \, d\mu,
\end {align*}
by the self-adjointness of the Laplacian on each leaf $L \times \{z\}$.}

\vspace {2mm}

\begin {proof}[Proof of $4) \implies 5) $]We compute:
\begin {align*}
\int_X f \cdot \Delta g \, d\mu  & = \int_X f \cdot \mathrm{div} (\nabla g) \, d\mu  \\
& = \int_X \mathrm{div} (f \nabla g) - \langle \nabla f, \nabla g\rangle  \, d\mu \\
& = \int_X -\langle \nabla f, \nabla g\rangle \, d\mu,
\end {align*}
by condition 4).  As this is symmetric in $f $ and $g $, condition 5) follows.
\end {proof}

\vspace {2mm}
\begin {proof} [Proof of $5)\implies 1) $]
It follows immediately from $5) $ that $\int_X \Delta f \, d\mu = 0$ for every continuous $f : X \longrightarrow \BR$ that is $C ^ 2 $ on each leaf of $X $.  In the terminology of Garnett \cite {Garnettfoliations}, $\mu $ is \emph {harmonic}.   Using the bounded geometry condition, Garnett proves that in every foliated chart $\phi : U \longrightarrow L \times Z$, a harmonic measure $\mu $ disintegrates as $d\mu =h (l,z) \, d\mathrm {vol_z} \, d\lambda(z) $, where $h $ is a positive leaf-wise harmonic function and $\lambda $ is a measure on the transverse space $Z$.  We must show that $h(\cdot,z) $ is constant for $\lambda $-almost every $z $.  

If $f,g \in C(X)$ are continuous functions supported in some compact subset of $U$ that are $ C ^ 2 $ on each plaque $L \times \{z\}  $, we have by $5) $ that
\begin {align*}
\int_{z \in Z} \int_{ L \times \{z\}}  f \cdot \Delta g \cdot h\, d\mathrm {vol_z} \, d\lambda
& = \int_{z \in Z} \int_{ L \times \{z\}} \Delta f \cdot  g \cdot h\, d\mathrm {vol_z} \, d\lambda \\
& = \int_{z \in Z} \int_{ L \times \{z\}}  f \cdot  \Delta(g \cdot h)\, d\mathrm {vol_z} \, d\lambda 
\end {align*}
As $f $ is arbitrary, this implies that $\Delta g \cdot h = \Delta (g\cdot h)$ on $\lambda $-almost every plaque $L \times \{z\}$.  As $g$ is arbitrary, $h(\cdot,z)$ must be constant for $\lambda $-a.e.\ $z $.
\end {proof}

\section {The foliated structure of $\mathcal M ^ d $}
\label{foliatedstructuresec}

Let $\mathcal M ^ d $ be the space of isometry classes of pointed Riemannian manifolds $(M,p)$, equipped with the smooth topology.  The space $\mathcal M ^ d $ is separable and completely metrizable -- we refer the reader to the appendix \S\ref {smoothsec} for a detailed introduction to the smooth topology and a proof of this result.

\subsection {Regularity of the leaf map}

When $X $ is a $d$-dimensional Riemannian foliated space, there is a `leaf map'

\[
\mathcal L : X \longrightarrow \mathcal M ^ d, \ \ \mathcal L (x) = (L_x,x),
\]
 defined by mapping each point $x$ to the isometry class of the pointed manifold $(L_x,x)$, where $L_x$ is the leaf of $X $ containing $x $.  We claim:

\begin {proposition}[The leaf map is Borel]\label {leafBorel}
If $U \subset \mathcal M ^ d $ is open, then
$\mathcal L^{-1}(U) = \cup_{i\in \BN} O_i \cap C_i,$ where each $O_i$ is open and each $C_i $ is closed in $X $. 
\end {proposition}

In \cite[Lemma 2.8]{Lessabrownian}, Lessa showed that the leaf map is measurable when the Borel $\sigma $-algebra of $X$ is completed with respect to any Borel probability measure on $X$.  The proof is a general argument that any construction in a Lebesgue space that does not use the axiom of choice is measurable, and uses the existence of an inaccessible cardinal.  He remarks that a more direct investigation of the regularity of $\mathcal L$ can probably be performed, which is what we do here.  We should also mention that \'Alvarez L\'opez and Candel \cite{Alvarezturbulent} study the leaf map from a foliated space into the Gromov-Hausdorff space of pointed metric spaces, and have observed, for instance, that it is continuous on the union of leaves without  holonomy. See also \cite{Alvarezuniversal}, where together with Barral Lij\'o, they study the leaf map into $\mathcal M^d$.

\vspace{2mm}

The key to proving Proposition \ref {leafBorel} is the following slight extension of a result of Lessa \cite[Theorems 4.1 \& 4.3]{Lessabrownian}, which we prove in the appendix.  

\begin {reptheorem}{foliationcompactness}
Suppose $X $ is a $d$-dimensional Riemannian foliated space in which $x_i \to x$ is a convergent sequence of points. Then $\mathcal L(x_i)$ is pre-compact in $\mathcal M ^ d $, and every accumulation point is a pointed Riemannian cover of $\mathcal L (x)$.
\end {reptheorem}

There is a partial order $\succ$ on $\mathcal M ^ d $, where $(N, q)\succ (M, p)  $ whenever $(N,q)$ is a pointed Riemannian cover of $(M,p)$.  With respect to $\succ$, Theorem \ref {foliationcompactness} asserts an `upper semi-continuity' of the leaf map.  The degree of regularity of $\mathcal L $ indicated in Proposition \ref {leafBorel} is exactly that of upper semicontinuous maps of between ordered spaces, so to get the same conclusion in our setting we must show a compatibility between $\succ$ and the smooth topology on $\mathcal M ^ d $:

\begin {lemma}\label {BorelLemma}
Every point $(M, p)\in\mathcal M ^ d $ has a basis of neighborhoods $\mathcal U$ such that the following properties hold for each $U\in \mathcal U$:
\begin {enumerate}
\item there is no $(N,q) \succ (M, p) $ such that $(N, q)\in \partial U$,
\item if $(N',q') \succ (N,q) \succ (M, p) $ and $(N',q') \in U$, then $(N,q) \in U$.
\end {enumerate}
\end {lemma}
\begin {proof}
In \S\ref {secmetrizability}, we define the  \emph {open $k^{th} $-order $(R,\lambda)  $-neighborhood} of $(M, p) $,  $$\mathring{\mathcal N} ^ k_{R,1,\lambda} (M, p), $$ 
to be the set of all $(N,q) $ such that there is an embedding $f: B_M (p, R)\longrightarrow N $ with $f (p) = q $ such that $D  ^ k f: T ^ k U\longrightarrow T ^ k N  $ is locally $\lambda' $-bilipschitz with respect to the iterated Sasaki metrics on the $1$-neighborhood of the zero section in $T^kU$, where $1< \lambda' <\lambda $. 
Any sequence of these neighborhoods is a basis around $(M, p) $ as long as $\lambda\to 1$ and $R,k \to \infty $, and we will show that when $\lambda,R,k$ are chosen appropriately then these neighborhoods satisfy the conditions of the lemma. 

The subset $\mathcal C \subset \mathcal M ^ d$ of pointed covers of $(M, p) $ is compact: if $R >0 $ is given the uniform geometry bounds on $B(p,R)\subset M $ lift to any cover, see Definition \ref {boundedgeometry}, so Theorem \ref {firstcompactness} gives pre-compactness of $\mathcal C\subset\mathcal M ^ d $, and $\mathcal C$ is closed in $\mathcal M ^ d $ since Arzela-Ascoli allows one to take a limit of covering maps.  Now fix some $R>0$.  If $(N_i,q_i) \in\mathcal C $ is a convergent sequence, the isometry type of $B(q_i,R) \subset N_i$ is eventually constant.  So by compactness, the $R$-ball around the base point takes on only finitely many isometry types within $\mathcal C $.

 Arzela-Ascoli's theorem implies that when forming the closure of $\mathring{\mathcal N} ^ k_{R,\lambda} (M, p)$, we just allow $\lambda'=\lambda $.  So, the boundaries $\partial \mathring{\mathcal  N} ^ k_{R,\lambda} (M, p)$ are disjoint for distinct values of $\lambda $.  As there only finitely many isometry types of, say, $2R$-balls around the base point in pointed covers of $(M, p)$, there can be only finitely many $\lambda <2$ such that there is a cover of $(M, p) $ in $\partial \mathcal N ^ k_{R,\lambda} (M, p)$.  So, the first condition in the lemma is satisfied as long as we choose $\lambda <2 $ to avoid these points.

To illustrate which neighborhoods $\mathring{\mathcal N} ^ k_{R,\lambda} (M, p)$ satisfies the second condition of the lemma, we need the following:

\begin {claim}\label {lemmaclaim}
Fix $(M,p)\in \mathcal M ^ d $.  Then for all $R$ in an open, full measure subset of $\BR_{>0}$, there is some $\lambda >1 $ such that whenever
$$\pi' :(N', q')\longrightarrow (M, p) $$
is a pointed Riemannian covering and $$f: B(p,R) \longrightarrow N'$$ is a locally $\lambda $-bilipschitz embedding with $f (p) = q' $, then $\pi'$ is injective on $f(B (p, R)) $.  
\end {claim}
\begin {proof}
If not, there is a sequence indexed by $i $ such that $\lambda_i\rightarrow 1 $, but 
$$\pi'_i:(N'_i, q'_i)\longrightarrow (M, p) $$ is not injective on $f_i(B (p, R)) $.  In the limit, we obtain a Riemannian cover
$$\pi':(N', q'){\longrightarrow} (M, p), $$
such that there is a pointed  {isometry} $$f: B(p,R) \longrightarrow B(q',R) \subset N',$$ but where $\pi' $ is not injective on $\overline {B(q',R)}. $  Here, the non-injectivity persists in the limit since the distance between points of $(N'_i, q'_i)$ with the same projection to $(M,p)$ is bounded below by the injectivity radius of $(M,p)$, which is positive on the compact subset of $(M, p) $ in which we're interested.   


The map $\pi' \circ f:B(p,R) \longrightarrow M$ is an isometry fixing $p$, so it extends to an embedding $F: \overline{B(p,R)} \longrightarrow M$.  As long as $R$ is a (generalized) regular value for the (nonsmooth) function $d(p,\cdot)$ on $M $, a full measure open condition \cite {Riffordmorse}, the inclusion $B(p,R) \hookrightarrow \overline {B (p, R)} $ is a homotopy equivalence, see \cite [Isotopy Lemma 1.4]{Cheegercritical}. So in this case, the map $F$ takes $\pi_1(\overline{B(p,R)}) \cong \pi_1({B(p,R)})$ into the $\pi'$-image of $\pi_1(N',q')$. Hence $F$ lifts to an isometry $\overline{B(p,R)} \longrightarrow B(q',R) \subset N'$,  by the  lifting criterion.  Since $F$ is an embedding, this contradicts that $\pi'$ is non-injective on $\overline {B(q',R)} $.
\end {proof}

As long as $\lambda,R$ are chosen according to Claim \ref {lemmaclaim}, $\mathring{\mathcal N} ^ k_{R,\lambda} (M, p)$ satisfies the second condition of the lemma.  For if 
$$(N',q') \succ (N,q) \succ (M, p) , \ \ \ (N',q') \in \mathring{\mathcal N} ^ k_{R,\lambda} (M, p),$$
then there is a map $f: B(p,R) \longrightarrow N'$ as above, so $\pi': (N',q') \longrightarrow (M, p)$
is injective on $f(B (p, R)) $.  In particular, the covering map 
$$\pi :(N',q')\longrightarrow (N,q)$$ is also injective there, so the composition 
$$\pi \circ f : B(p,R) \longrightarrow N'$$ is an embedding.  As $\pi \circ f $ inherits the same Sasaki-bilipschitz bounds that $f$ has, this shows that $(N,q)\in \mathring{\mathcal N} ^ k_{R,\lambda} (M, p)$ as well. 

Therefore, for any $k$, almost every $R>0$, and $\lambda$ sufficiently close to $1 $,  the neighborhood $ \mathring{\mathcal N} ^ k_{R,\lambda} (M, p)$ satisfies both conditions of the lemma.  As these neighborhoods form a basis for the topology of $\mathcal M ^ d $  at $(M,p)$, we are done.
\end {proof}

Using the lemma, we now complete the proof of Proposition \ref {leafBorel}.  Recall that $X$ is a $d$-dimensional Riemannian foliated space and
$$\mathcal L : X \longrightarrow \mathcal M ^ d, \ \ x \mapsto (L_x,x)$$
is the leaf map.  We want to show that for each open $U \subset \mathcal M ^ d $, the preimage $\mathcal L^{-1}(U) = \cup_{i\in \BN} O_i \cap C_i,$ where each $O_i$ is open and each $C_i $ is closed in $X $.  

It suffices to check this when $U $ is chosen as in Lemma \ref {BorelLemma}.
 If $ \mathcal L ^ {-1 }(U)$ does not have the form $\cup_{i\in \BN} O_i \cap C_i,$ there is a point $x \in \mathcal L ^ {-1 }(U)$ and a sequence $$x_i \in  \overline {\mathcal L ^ {-1 }(U)} \setminus \mathcal L ^ {-1 }(U), \ \ x_i \rightarrow x \in \mathcal L ^ {-1 }(U). $$
Passing to a subsequence, we may assume by Theorem \ref {foliationcompactness} that 
$$\mathcal L (x_i) \rightarrow (N,q) \succ \mathcal L (x) . $$
Note that as $\mathcal L (x_i)\notin U$ for all $ i $, we have $(N,q) \notin U$ as well.

Each $x_i $ is the limit of some sequence $(y_{i,j})$ in $\mathcal L ^ {-1 }(U)$, and Theorem \ref {foliationcompactness} implies that after passing to a subsequence, we have that for each $i $, $$ \mathcal  L (y_{i,j})\ {\rightarrow}\ (Z_i,z_i) \succ \mathcal L(x_i) \text { as } j \rightarrow\infty.$$   Now fixing $R>0$, since the manifolds $\mathcal L (x_i) $ converge in $\mathcal M ^d $, the $R$-balls around their base points have uniformly bounded geometry, as in Definition \ref {boundedgeometry}.  These geometry bounds lift to pointed covers, so are inherited by the $(Z_i,z_i)$. So by Theorem \ref{firstcompactness}, after passing to a subsequence we may assume that $$ (Z_i, z_i) \rightarrow (N',q')\in \mathcal M ^ d .$$
Moreover, as $ (Z_i,z_i)\succ \mathcal L(x_i)$ for each $i $, we have
$$(N',q') = \lim  (Z_i,z_i) \succ \lim \mathcal L(x_i)=  (N,q),$$
simply by taking a limit of the covering maps.  Remembering now that $(Z_i,z_i)$ was defined as the limit of $\mathcal L (y_{i,j})$ as $j\to \infty $, if we choose for each $i $ some large $j=j(i)$ and abbreviate $y_i=y_{i,j(i)}$, then 
$$\mathcal L (y_i) \rightarrow (N',q')$$
as well.
However, by construction we have $\mathcal L (y_i) \in U$, so $(N',q') \in \overline U$.  

The first part of Lemma \ref {BorelLemma} implies that $(N',q') \in U$, and then the second part shows $(N,q) \in U$.  This is a contradiction, as we said above that $(N,q) \notin U$.

\subsection{Resolving singularities in $\mathcal M ^ d $}
\label {pdsec}

 n this section, we will assume that $d\geq 2$. We saw in Example \ref {onecase} that $\mathcal M^1\cong (0,\infty]$ is completely understood; the reader is encouraged to think through the proofs of our results when $d=1$ on his/her own.

$\mathcal M ^d$ is not a naturally foliated space: although the images of the maps
\[M \longrightarrow \mathcal M ^d, \ \ p \mapsto (M,p) \]
partition $\mathcal M ^d$ as would the leaves of a foliation, these maps are not always injective and their images may not be manifolds. However, the following theorem, discussed in \S \ref{intro}, shows that there is a way to desingularize $\mathcal M^d$  so that the theory of unimodular measures becomes that of completely invariant measures. 

\begin {reptheorem}{desingularizing}[Desingularizing $\mathcal M^d$]If $\mu$ is a completely invariant probability measure on a Riemannian foliated space $X$, then $\mu$ pushes forward  under the leaf map to a unimodular probability measure $\bar \mu$ on $\mathcal M ^ d $. 

Conversely, there is a Polish Riemannian foliated space $\mathcal P ^ d$  such that any $\sigma$-finite unimodular measure on $\mathcal M ^ d $ is the push forward  under the leaf map of some completely invariant measure on $\mathcal P ^ d $.    Moreover, for any fixed  manifold $M $, the preimage of $\{(M,p) \ | \ p \in M\} \subset \mathcal M^d$  under the leaf map is a union of leaves of $\mathcal P^d$, each of which is isometric to $M $.\end {reptheorem}

As  a corollary of this  and Theorem \ref {foliations},  we have the following theorem, which we also discussed in the introduction.

\begin{reptheorem}{gflow}
	A $\sigma$-finite Borel measure on $\mathcal M^d$ is unimodular  if and only if the lifted measure $\tilde \mu$ on $T ^ 1\mathcal M^d$ is  geodesic flow  invariant.
\end{reptheorem}
\begin{proof}
By Theorem \ref{desingularizing}, $\mu$ is the push forward of a completely invariant measure $\nu$  on a Riemannian foliated space $\phi: X \longrightarrow \mathcal M^d$. (Taking $X=\mathcal P^d$.) By Theorem \ref{foliations}, the induced measure $\tilde \nu$ on $T^1 X$ is invariant under geodesic flow. Now, the leafwise  derivative $D\phi: T^1 X \longrightarrow\mathcal T^1\mathcal  M^d$ is  geodesic flow equivariant, so the  push forward measure $D\phi_* \tilde \nu = \tilde \mu$ is geodesic flow  invariant.
\end{proof}

The first assertion of Theorem \ref {desingularizing} is easy to prove. If $$\mathcal R =\{(x,y) \in X \times X \ | \ x, y \text { lie on the same leaf} \} , $$ then the measures $\mu_l, \mu_r$ on $X \times X$ are supported on $\mathcal R$ and push forward to $\bar\mu_l, 
\bar\mu_r$ under the natural map $\mathcal R \longrightarrow \mathcal M ^ d_2.$  By Theorem \ref{foliations}, $\mu_l=\mu_r $, so $\bar\mu_l = \bar\mu_r$.

The idea for the `conversely' statement is to use Poisson processes to obstruct the symmetries of these manifolds, converting $\mathcal M ^ d $ into a foliated space $\mathcal P ^ d $. To do this, we will recall some background on Poisson processes, define $\mathcal P^d$ and show how to translate between measures on $\mathcal M ^ d $ and on $\mathcal P^d$, and then verify that $\mathcal P ^ d $ is a Riemannian foliated space.

\vspace{1mm}

If $M $ is a Riemannian $d $-manifold, the \emph {Poisson process of $M $} is the unique probability measure $\rho_M$ on the space of locally finite subsets $D\subset M $ such that
\begin {enumerate}
\item if $A_1, \ldots, A_n$ are disjoint Borel subsets of $M $, the random variables that record the sizes of the intersections $D\cap A_i$ are independent,
\item if $A \subset M$ is Borel, the size of $D \cap A$ is a random variable having a Poisson distribution with expectation $\volume_M(A)$.
\end {enumerate}
For a finite volume subset $A\subset M$ and $n\in \mathbb{N}$, we have (cf. \cite[Example 7.1(a)]{Daleyintroduction})
\begin {equation}\label {janossy}\mathrm {Prob}\left (\substack{\text {for } (x_1,\ldots,x_n)\in A^n, \text{ we have } D\cap A =\{x_1,\ldots,x_n\}, \\ \text { given that } D\cap A \text { has } n \text { elements.}}\right)=d\volume_M^n(x_1,\ldots,x_n).\end {equation}
In other words, if $D$ is chosen randomly, the elements of $D\cap A$ are distributed within $A$ independently according to $\volume_M $.  

We refer the reader to \cite{Daleyintroduction} for more information on Poisson processes.  In this text, they are not introduced on Riemannian manifolds, but for measures on $\mathbb{R} ^ d$ that are absolutely continuous with respect to Lebesgue measure.   However, as the Poisson process  behaves naturally under  restriction and  disjoint union, it is `local', and  can be defined naturally for manifolds. In fact, the Poisson process really only depends on the Riemannian  measure on $M$,  and not on  the topology of $M$. Since $M$  is isomorphic as a measure space to  the (possibly infinite)  interval $(0,\volume(M))$, see \cite{Rohlinfundamental}, one really only needs to understand the  usual Poisson  process on $\BR_+$, as that of an interval is just its restriction.

\vspace{1mm}

When $M $ is a Riemannian $d $-manifold, let $\CF M$ be its \emph {orthonormal frame bundle}, the bundle in which the fiber over $p\in M$ is the set of orthonormal bases for $T_pM $.  If we regard $\CF M$ with the Sasaki-Mok metric \cite {Mokdifferential}, then
\begin {enumerate}
\item when $f:M\to M$ is an isometry, so is its derivative $Df : \CF M \to \CF M$,
\item the Riemannian measure $\volume_{\CF M}$ is obtained by integrating the Haar probability measure on each fiber $\CF_p M \cong O(d)$ against $\volume_M $.
\end {enumerate}

\begin {lemma}\label {free}
If $M$ is a Riemannian $d$-manifold, $d\geq 2$, then $\isometries(M)$ acts essentially freely, with respect to the Poisson measure $\rho_{\CF M}$, on the set of  {nonempty} locally finite subsets of $\CF M$.
\end {lemma}

The subset $\emptyset \subset \CF M$ is fixed by $\isometries (M) $ and has $\rho_{\CF M}$-probability\footnote{Via the measure  isomorphism $\mathcal FM\longrightarrow (0,\volume(\mathcal FM))$,  this is just the probability that there are no points in the interval $(0,\volume(\mathcal FM))$,  under the usual Poisson process on $\BR_+$.} $e^{-\volume \CF M}$. So if $M$ has finite volume, we must exclude $\emptyset$ in the statement of the lemma.

 Also, if $M\cong S^1$, then after choosing an orientation, every subset $\{e_1,e_2\} \subset \CF M$, where $e_1,e_2$ have opposite orientations, is stabilized by an involution of $M $.  Since $S^1$ is compact, two-element subsets of $\CF S^1$ appear with positive $\rho_{\CF S^1}$-probability, so the statement of the lemma fails for $1 $-manifolds.

\begin {proof}

The Lie group $\isometries(M)$ acts freely on $\CF M$, so any nonempty subset $D\subset \CF M$ that is stabilized by a nontrivial element $g\in\isometries(M)$ has at least two points. As the action is proper, its orbits are properly embedded submanifolds, so unless one is a
union of components of $\CF M$, all orbits have $\volume_{\CF M}$-measure zero. In this case, the $\volume_{\CF M}$-probability of selecting two points from the same orbit is zero, so  $\rho_{\CF M}$-a.e. $D\subset \CF M$ has trivial stabilizer, by Equation \eqref{janossy}.

So, we may assume from now there is an orbit of $\isometries (M) \circlearrowright \CF M$ that is a union of components of $\CF M$. (The frame bundle $\CF M$ has either 1 or 2 components, depending on whether $M $ is orientable.) Then $\isometries (M)$ acts transitively on $2$-planes in $TM$, so $M$ has constant sectional curvature. As $\isometries (M)$ also acts transitively on $TM$, $M$ is either $S^d,\BR \mathrm P^d, \BR^d$ or $\BH^d$.


 If $D \subset \CF M$ is stabilized by some nontrivial $g\in \isometries (M)$, it has at least two points $e_1,e_2$, and we can consider the images $g(e_1),g(e_2)$. Either $e_1,e_2$ are exchanged by $g$, or one is sent to the other, which is sent to something new, or both elements are sent to new elements of $D$. As the elements of a random $D$ are distributed according to $\volume_{\CF M}$, by  \eqref{janossy}, it suffices to prove that for $( e_1,\ldots, e_4)\in \CF M^4$, the following are $\volume_{\CF M}^n$-measure zero conditions:
\begin {enumerate}
\item $ g(e_1) =e_2,\ g(e_2)=e_1$, for some $g\in \isometries (M)$,
\item $d(e_1,e_2)=d(e_2,e_3)$, 
\item $d(e_1,e_2)=d(e_3,e_4)$.
\end {enumerate}

An isometry that exchanges two frames must be an involution, since its square fixes a frame. So, for 1) we want to show that the probability of selecting frames $e_1,e_2\in \CF M$ that are exchanged by an involution is zero. The point is that in each of the cases $M=S^d,\BR \mathrm P^d, \BR^d$ or $\BH^d$, an involution exchanging $p,q\in M$ leaves invariant some geodesic $\gamma : [0,1]\longrightarrow M$ joining $p,q$, and then exchanges $-\gamma'(0) \in T_p M$ with $\gamma'(1)\in T_qM$.  So, after fixing a frame $e_1 \in \CF M_p$, the frames in $\CF M_q$ that are images of $e_1$ under involutions form a subset of $\CF M_q$ of dimension at most that of $ O(d-1)$, which has zero Haar measure inside of $\CF M_q \cong O(d)$. Integrating over $q$, we have that for a fixed $e_1$, the probability that a frame $e_2\in \CF M$ is the image of $e_1$ under an involution is zero. Integrating over $e_1$ finishes the proof of part 1).

For 2), note that for a fixed $e_2\in \CF M$, the function 
$$d(e_2,\cdot) : \CF M \longrightarrow \BR$$ pushes forward $\volume_{\CF M}$ to a measure on $\BR$ that is absolutely continuous with respect to Lebesgue measure -- its RN-derivative at $x\in \BR$ is the $(\mathrm{dim}( \CF M) - 1)$-dimensional volume of the metric sphere around $e_2$ of radius $x $.  So, if $e_1$ and $e_3$ are chosen against $\volume_{\CF M}$, the distances $d(e_1,e_2)$ and $d(e_2,e_3)$ will be distributed according to a measure absolutely continuous to Lebesgue measure on $\BR^2$, so will almost never agree. The proof that 3) is a measure zero condition is similar.
\end {proof}

We now show how to convert $\mathcal M ^ d$ into a foliated space by introducing Poisson processes on the frame bundles of each Riemannian $d$-manifold.  Let
$$\mathcal{P} ^ d_{all}= \{(M,p,D) \ | \ \substack{M \text{ a complete Riemannian } d\text {-manifold,} \\ p\in M, \text { and } D \subset \CF M \text { a closed subset}} \}/\sim,$$
where $(M,p,D)\sim(M',p',D') $ if there is an isometry $\phi: M \longrightarrow M$ with $\phi(p)=p'$ whose derivative $d \phi$ takes $D$ to $D'$. There is a Polish \emph {smooth-Chabauty topology} on $\mathcal {P}^d_{all}$ obtained from the smooth topology on $\mathcal M^d$ and the Chabauty topology on the subsets $D$, see \S \ref{smoothChabauty}.  
Now consider the subset
$$\mathcal {P} ^ d =\{(M,p,D) \in \mathcal {P} ^ d_{all} \ | \ \nexists \text { an isometry } \phi:M\longrightarrow M \text { with } d \phi(D)=D\}.$$ 
The subset $\mathcal {P} ^ d$ is $G_\delta$, since
$\mathcal {P} ^ d_{all} \setminus \mathcal {P} ^ d = \cup_{n \in \BN} F_n,$
where $F_n$ is the set of all $ (M, p,D)$ such that there is an isometry $\phi: M \to M$ with $$d \phi (D) = D \text { and } 1/n \leq d_M(p,\phi(p)) \leq n;$$ here, $F_n$ is closed by the Arzela Ascoli theorem. Hence, by  Alexandrov's theorem,   $\mathcal {P} ^ d $ is  a Polish space.  Note that $\mathcal {P} ^ d $ is dense in $\mathcal {P} ^ d_{all}$, since any Riemannian manifold can be perturbed to have no nontrivial isometries.

\begin {theorem}\label {pd}
$\mathcal {P}^d$ has the structure of a Polish Riemannian foliated space, where  $(M,p,D)$ and $(M',p',D')$ lie in the same leaf  when there is an isometry $$\phi:M\longrightarrow M' , \ \ d\phi(D)=D'. $$
\end {theorem}

Assuming Theorem \ref {pd} for a moment, let's indicate how to transform a unimodular measure  $\mu $ on $\mathcal M ^ d $ into a completely invariant measure on $\mathcal P ^ d$, which will finish the proof of Theorem \ref {desingularizing}. 
Each fiber of the projection
$$\pi: \mathcal {P} ^ d_{all} \longrightarrow \mathcal M ^d, \ \ (M,p,D)\longmapsto (M,p)$$
is identified with  a set of of closed subsets of $\CF M$,  and this identification is unique up to isometry.  The Poisson process on $\CF M $ (with respect to the natural volume, e.g.\ that induced by the Sasaki-Mok metric) induces a measure $\rho_{(M,p)}$ on $\pi ^ {-1} (M, p) $,  supported on the Borel set of locally finite  subsets. We call  this  measure the \emph {framed Poisson process} on that fiber, and by Lemma~\ref {free} we have $$\rho_{(M,p)}(\mathcal {P} ^ d )=1-e ^ {-\volume \CF M}>0$$ for each $(M, p) \in \mathcal M ^ d $. Moreover, the map
$$\mathcal M ^ d \longrightarrow \mathcal M (\mathcal P ^ d), \ \ (M, p) \longmapsto \rho_{(M, p)}$$
is continuous, where $\mathcal M (\mathcal P ^ d)$ is the space of Borel measures on $\mathcal P ^ d $, considered with the weak$^*$ topology. (This follows from weak* continuity of the Poisson process associated to a space with a measure $\mu$ as the measure varies in the weak* topology, which is a consequence of \eqref{janossy}.)

So, given a measure $\mu $ on $\mathcal M ^ d $, we define a measure $\hat\mu $ on $\mathcal P ^ d$ by 
\begin {equation}\label {muint}\hat \mu = \int_{(M, p)\in\mathcal M ^ d} \frac{\rho_{(M, p)}}{1-e ^ {-\volume \CF M}} \  d\mu.\end {equation}
The push forward of $\hat \mu$ under the projection to $\mathcal M ^ d $ is clearly $\mu $, so we must only check that $\hat \mu $ is completely invariant. 

Suppose that $R \subset\mathcal P ^ d\times\mathcal P ^ d$ is the leaf equivalence relation of $\mathcal P ^ d $, i.e.\ the set of all pairs $ ((M,p,D),(M',p',D'))$ such that there is an isometry $\phi:M\longrightarrow M' $ with $d \phi(D)=D'$. Each such pair determines a tuple $(M',\phi(p),p',D')$, a doubly rooted manifold together with a  closed subset of its frame bundle, that is unique up to isometry. So, there is a map
$$R \longrightarrow\mathcal M_2 ^ d, \ \ ((M,p,D),(M',p',D')) \longmapsto (M',\phi(p),p'),$$
where the fiber over $(M,p,q) \in \mathcal M_2 ^ d $ is canonical identified (up to isometry of $M$) with the set of  closed subsets of $\mathcal F M$ on which $\isometries(M)$ acts freely.  From the construction of $\hat \mu$, the measures $\hat \mu_l$ and $\hat \mu_r$ on $R $ are obtained by integrating (rescaled) Poisson processes on each fiber against $\mu_l$ and $\mu_r$. So, if $\mu_l=\mu_r$, then $\hat \mu_l=\hat \mu_r$, implying $\hat\mu$ is completely invariant by Theorem~\ref{foliations}.

\vspace {2mm}

\subsection {The proof of Theorem \ref {pd}}\label {proofpdsec}
The goal is to cover $\mathcal P ^ d$ by open sets $\mathcal U $, together with homeomorphisms $$h:\BR ^ d \times Z \longrightarrow \mathcal U  ,$$ where $Z=Z(\mathcal U)$ is separable and metrizable,  with transition maps
$$t : \BR ^ d \times Z  \longrightarrow \BR ^ d \times Z', \ \ t(x,z)=(t_1(x,z),t_2(z)),$$
 where $t_1(x,z)$ is smooth in $x$, and where $t_1,t_2$ and all  the partial derivatives $\frac {\partial t_1(x,z)}{\partial x_i}$ are continuous on $\BR^d \times Z$. We also want the leaves of the foliation to be obtained by fixing $M $ and $D\subset M$, and letting the base point $p\in M$ vary. 

The construction of the charts will require some work --- in outline, the idea is as follows. Given $X\in\mathcal P ^ d $, we show that on any small neighborhood $\mathcal U \ni X $, there is an equivalence relation $\sim $ whose equivalence classes are obtained by  taking some $(M,p,D)$ and making slight variations of the base point $p$.  Eventually, these equivalence classes will be the plaques $h(\BR ^ d\times\{z\} )$, and  the quotient space $\mathcal U / \sim$ will be the transverse space $Z$.  That is, we will have a homeomorphism $h$ to complete  the  following commutative diagram: $$\xymatrix{ \BR ^ d \times Z  \ar[r]^{\ \ \  h} \ar[d] & \mathcal U \ar[d] \\ Z \ar[r] & U /\sim} $$ 

Of course, this cannot be done without a careful choice of $\mathcal U $.  It is not hard to choose $\mathcal U $ so that each $\sim$-equivalence class is a small disc of base points  in some manifold $M$  with a distinguished closed subset $D$.  Shrinking $\mathcal U$,  we show that one can construct a continuously varying family of base frames,  one for each of these $M$. Then, we use the Riemannian exponential maps associated to these frames to parameterize the $\sim$-classes, which allows us to identify them with $\BR^d$  in a way that is transversely continuous.

 Most of the proof involves constructing the base frames.   Essentially, this is just a framed version  of  the following, which we can discuss now without introducing more notation.  \emph{Given our neighborhood $\mathcal U $ of $X \in \mathcal P^d$, there is a section $s: \mathcal U/\sim \longrightarrow \mathcal U$ for the  projection map with $s([X])=X$.}  Briefly, the idea  for this is as follows. Fix a metric $d$ on $\mathcal P^d$, and for each $\sim $-class $\mathcal E$, set $$\mathcal E_0 = \{ p \in \mathcal E \ | \ d(p,X) \leq 2 d(\mathcal E,X)\}.$$ The point $s(\mathcal E) \in \mathcal E$ is  then defined to be the `circumcenter' of $\mathcal E_0 $, when we regard $\mathcal E_0$  as a small subset inside of a Riemannian $M$ as above. (In what follows, this  argument will be done for base frames,  and the notation will be different.)

\vspace{2mm}

Before starting the proof in earnest, we record the following two lemmas, which should convince the reader that such a foliated structure is likely.

\begin {lemma}[Leaf inclusions]\label {leafinclusions}
Suppose that $X=(M, p,D)\in \mathcal P ^ d $, and define 
$$L : M \longrightarrow\mathcal P^d, \ \ L(q)=(M,q,D).$$
Then $L$ is a continuous injection, so the restriction of $L$ to any precompact subset of $M $ is a homeomorphism onto its image.
\end{lemma}

\begin {proof}
For continuity of $L $, note that if $q,q'\in \BR ^ d$ and $|q-q'|=\epsilon$, there is a diffeomorphism $f$ of $\BR ^ d$ taking $q$ to $q'$, such that
\begin {enumerate}
\item $f$ is supported in a $2\epsilon $-ball around $q$,
\item the $k^{th}$ partial derivatives of $f$ are all bounded by some $C(k,\epsilon)$, where $C(k,\epsilon) \rightarrow 0 $ as $\epsilon\rightarrow 0$.
\end {enumerate}
For instance, one can just fix any bump function taking the origin to $(1,0,\ldots,0)$ that is supported in a $2$-ball around the origin, and conjugate it using appropriate Euclidean similarities. So if $q,q' \in M$ are close, we can choose an $\BR ^d$-chart around them and transfer such an $f$ to $M$ to give an almost isometric map verifying that $(M,q,D)$ and $(M,q',D)$ are close in $\mathcal P ^ d$.

Injectivity of $L $ comes from the definition of $\mathcal P^d$, which requires that there are no nontrivial isometries of the pair $(M,D)$. The statement about compact subsets of $M $ follows from point set topology, as $\mathcal P^d$ is Hausdorff.
\end {proof}
We optimistically call the image of $L$ the \emph {leaf} through $X\in \mathcal P ^ d $, and write
$$L_X :=L(M)\subset\mathcal P ^ d.$$ Via $L$, \emph {we will  from now on regard $L_X$ as a smooth Riemannian $d$-manifold that is topologically embedded (non-properly) in the space $\mathcal P^d$}.   (The distance function on a leaf $L_X$ will usually be written as $d_{L_X}$.) Under this identification, the point $p$  becomes $X $ and  the subset $D\subset\mathcal FM$ becomes a distinguished subset of the frame bundle of $L_X$. Note that the natural (manifold)   topology on $L_X$ is not  the subspace   topology induced from  the inclusion $L_X\hookrightarrow \mathcal P ^ d $.
 
\begin {lemma}[Chabauty convergence of leaves]\label {leafinclusions2}
 Suppose $X_i \rightarrow X$ in $\mathcal P ^ d$. Then
 \begin {enumerate}
 \item for every point $Y\in L_X$, there  is a sequence $Y_i \in L_{X_i}$ with  $Y_i \rightarrow Y $ in $\mathcal P^d$ and $d_{L_{X_i}}(X_i,Y_i) \rightarrow d_{L_X}(X,Y)$.
 \item  if $Y_i\in L_{X_i}$ with $d_{L_{X_i}}(X_i,Y_i)<\infty$, then after passing to a subsequence we have $Y_i \to Y\in L_X$, and $d_{L_{X_i}}(X_i,Y_i) \to d_{L_X}(X,Y)$.
 \end {enumerate}\end {lemma}

  Note that  there may be sequences $Y_i \in L_{X_i}$  that converge in $\mathcal P ^ d $, but where $d_{L_{X_i}}(X_i,Y_i) \rightarrow \infty$,  so the statements about distance  above have content.

\begin {proof}
By  definition of the  convergence $X_i\rightarrow X $, see \S \ref {smoothChabauty}, after conjugating by the leaf inclusions, there are almost isometric maps
 \begin {equation}f_i : L_X \dashmap L_{X_i},\label {almostisometric}\end {equation}
with $f_i(X)=X_i$ and where the derivatives $Df _i$  pull back the distinguished subsets of the frame  bundles $L_{X_i}$  to a sequence that  Chabauty converges to the distinguished subset of $L_X$. (Here, $\dashmap$ means that $f_i$  is defined on some ball around  the base point, where the  domains exhaust $L_X$ as $i\to \infty$, see \ref {smoothChabauty}.)

For 1), the sequence $Y_i = f_i(Y)$  is defined for large $i $ and converges to $Y $ in $\mathcal P  ^ d$: one can  use the $ f_i $ in the definition of convergence. For 2), the sequence $f_i^{-1} (Y_i) $  is defined for large $i $, and is pre-compact in $L_X$ by the  condition on distances.  So, after passing to a subsequence, it converges to  some point $Y\in L_X$, and we will have $Y_i \rightarrow Y$ in  $\mathcal P ^ d$  as well,  again using the $ f_i $ in the definition of convergence.
\end {proof}

\subsubsection{Constructing the equivalence relation.}

For any $\epsilon>0$,  define a relation $\sim_\epsilon $ on $\mathcal {P} ^d $ by letting 
$$X \sim Y \ \ \text{ if } \ \ Y \in L_X \text { and } d_{L_X}(X,Y)\leq \epsilon. $$
 Note that $\sim_\epsilon $ is reflexive and symmetric, but we only have
$$X \sim_\epsilon Y \text { and } Y \sim_\delta Z \Longrightarrow X \sim_{\epsilon +\delta} Z,$$  rather than  transitivity for a particular $\sim_\epsilon$.    In particular, the equivalence class of $X$ \emph {with respect to the transitive closure of $\sim_\epsilon$} is exactly $L_X $.

However, each $\sim_\epsilon $ \emph {is} transitive on sufficiently small  subsets of $\mathcal P ^ d $. 

\begin {lemma}\label {technicalities}
If $O\in \mathcal {P}^ d $, then for fixed $\delta,\epsilon>0$, there is a neighborhood $\mathcal U \ni O $ on which the relations $\sim_\delta $ and $\sim_\epsilon$ agree. Hence, if $\mathcal U$  is chosen so that $\sim_\epsilon =\sim_{2\epsilon}$ on $\mathcal U$,  then $\sim_\epsilon $ is an equivalence relation on $\mathcal U $.
\end{lemma}
\begin {proof}
Assuming this is not the case for some $\delta <\epsilon $, let  $X_i $ and $Y_i$ be sequences in $\CP ^ d $ that converge to $ O$, with $X_i \sim_{\epsilon} Y_i,$ but $ X_i \not\sim_{\delta} Y_i.$
Since $X_i \sim_{\epsilon} Y_i$, the  distance  $d_{L_X}(X,Y)\leq \epsilon < \infty $. So, Lemma \ref {leafinclusions2} 2) applies, and we must have $d_{L_{X_i}}(X_i,Y_i)\rightarrow d_{L_O}(O,O)=0$, violating that $d_{L_{X_i}}(X_i,Y_i)\geq \delta $ for all $i$.\end {proof}

From now on, \emph {we assume that all our neighborhoods $\mathcal U $ are small enough so that $\sim_1=\sim_2$, in which case $\sim_1$ is an equivalence relation}.

\begin {lemma}\label {sepmet}
The quotient topology on $\mathcal U / \sim_1$ is separable and metrizable.
\end {lemma}
\begin {proof}
Separability of $\mathcal U / \sim_1$ is immediate, as $\mathcal U $ is separable. Metrizability of $\mathcal U / \sim_1$ can be proved in the same way that we prove it for $\mathcal P ^ d$ in Section \ref{smoothChabauty}. The only difference is that when comparing two triples $(M,p,D)$ and $(M',p
,D')$, we now let our maps $f : B_M(p,R) \longrightarrow M'$ take $p$ to any point within an $\epsilon $-neighborhood of $p'$. Ordinarily, such flexibility would make it hard to establish a triangle inequality, but if $\mathcal U $ is sufficiently small, then such a map that realizes the distance between $(M,p,D)$ and $(M',p
,D')$ will have small distortion, and using a limiting argument as in Lemma \ref{technicalities}, one can show that in fact $f(p)$ must be (arbitrarily) close to $p'$, so that the compositon of two such maps still takes base points within $\epsilon $ of base points.
\end {proof}

\subsubsection{A section of base frames.} We describe here how to construct, for each equivalence class in $\mathcal U /\sim_1 $, a base frame for the corresponding $M $. To make a precise statement, we need a framed version of our space $\mathcal P ^ d $. Define
$$\mathcal{FP} ^ d= \left\{(M,e,D) \ \Big | \ \substack{M \text{ a complete Riemannian } d\text {-manifold,} \\ e\in \mathcal FM, \text { and } D \subset  \mathcal FM \text { a closed subset} \\ \text {such that }\not \exists \text { an isometry } f:M\longrightarrow M, \ Df(D)=D } \right\}/\text {isometry}.$$
 There is a natural  Polish topology on $\mathcal{FP} ^ d $, coming from the framed smooth topology on the pairs $(M, e) $ and the Chabauty topology on the subsets $D $; compare with Section \ref {smoothChabauty}. 
 We denote the natural projection by $$\pi: \mathcal {FP} ^ d\longrightarrow \mathcal {P} ^ d .$$

Lemmas \ref {leafinclusions} and \ref {leafinclusions2} have framed analogues. If $X=(M,e,D)\in\mathcal {FP} ^ d ,$ then $$L : \mathcal FM \longrightarrow \mathcal {FP} ^ d, \   L(e')=(M,e',D) $$ is  a continuous injection, and via $L$  we will view its  image $L_X \subset \mathcal {FP} ^ d $ as (the frame bundle of) a smooth Riemannian manifold.   Moreover, if $$X_i= (M_i, e_i, D_i) \converges X= (M, e,D)  $$ in $\mathcal {FP} ^ d$,  then  by definition, see \S \ref{smoothChabauty}, there are  almost isometric maps $$f_i: M \dashmap M_i$$
such that the derivatives $Df_i$ map $e$ to $e_i$, and pull back $(D_i)$  to a sequence  of  subsets of $\mathcal FM$ that Chabauty converges to $D$.    Identifying frame bundles with  the corresponding leaves in $\mathcal {FP}^d$ as  above,  these derivatives become maps
\begin {equation}\label {framedais} F_i : L_X \dashmap L_{X_i}\end {equation}
 These maps are almost isometric, in the sense that if $L_X$ and $L_{X_i}$ are equipped with the Sasaki-Mok metrics $g_i$ induced by  the Riemannian metrics on $M_i$ and $M$, see \cite{Mokdifferential}, then  we have $F_i  ^*(g_i)\rightarrow g$  in the $C^\infty $-topology. Using the $F_i$,  one can prove a framed analogue of Lemma \ref{leafinclusions2}:

\begin {lemma}[Chabauty convergence of leaves]\label {framedleafinclusions2}
 If $X_i \rightarrow X$ in $\mathcal {FP} ^ d$, then
 \begin {enumerate}
 \item for every point $Y\in L_X$, there  is a sequence $Y_i \in L_{X_i}$ with  $Y_i \rightarrow Y $ in $\mathcal {FP}^d$ and $d_{L_{X_i}}(X_i,Y_i) \rightarrow d_{L_X}(X,Y)$.
\item if $Y_i\in L_{X_i}$ with $d_{L_{X_i}}(X_i,Y_i)<\infty$, then after passing to a subsequence we have $Y_i \to Y\in L_X$, and $d_{L_{X_i}}(X_i,Y_i) \to d_{L_X}(X,Y)$. \end {enumerate}
\end {lemma}

Finally, the relation $\sim_1$   on $\mathcal P ^ d $ pulls back  under $\pi$ to a relation on $\mathcal {FP}^d$,  which we will abusively call $\sim_1$  as well.   On small subsets $\mathcal V \subset \mathcal {FP}^d$, $\sim_1$  is an equivalence relation on $\mathcal V$,  just as in the previous section.

\begin {lemma}\label {existencesection}
Every $F\in \mathcal {FP} ^ d $ has a neighborhood $\mathcal V$ on which there is a continuous map $s:\mathcal V \longrightarrow\mathcal {FP}^d$ with $s (F) = F $ that is constant on $\sim_1 $-equivalence classes and satisfies $s (X)\sim_1 X$ for all $X\in\mathcal V$. 
\end {lemma}

So, $s$ gives a continuous section for the map $\mathcal V \longrightarrow \mathcal V/\sim_1 $ near $F$.

\vspace{2mm}

The  proof of Lemma \ref {existencesection}  will occupy the rest of the section.   The idea is as follows. We first select a distinguished subset of each equivalence class, essentially consisting of those points whose distances to $F$ are at most twice the minimum distance to $F$ in that equivalence class.  We then show that these subsets vary continuously with the equivalence class. The desired section is constructed by always choosing the `circumcenter' of the distinguished subset.

Fix a metric $d_{\mathcal {FP} ^d } $ on $\mathcal {FP} ^ d $ and suppose $\mathcal V $ is a metric ball around $F $.  Define
$$\delta:\mathcal V \longrightarrow \BR, \ \ \delta (X) = \inf\{d_{\mathcal {FP} ^ d}(X',F) \ | \ X'\in { \mathcal V }, \ X'\sim_1 X\}. $$
Each $\sim_1 $-equivalence class in $\mathcal V$ is pre-compact in $\mathcal {FP} ^ d $, since it is the image of a pre-compact subset under a continuous map $\mathcal FM \longrightarrow \mathcal {FP} ^ d$ from the frame bundle of some manifold $M $.  And if $X\in \mathcal V $, the metric ball $\mathcal V$ contains all points of $\mathcal {FP} ^ d $ that are closer to $F $ than $X $.  So, the infimum is always achieved.

\begin {claim}
The map $\delta $ is continuous.
\end {claim}
\begin {proof}
 Suppose that in $\mathcal  V $ we have
$$X_i= (M_i, e_i, D_i) \rightarrow X= (M, e,D)  $$ and that $X_i'\sim_1 X_i$ realize the infimums defining $\delta (X_i) $.  We can assume that $\sup d_{L_{X_i}}(X_i,X_i') \leq 2$, so after passing to a subsequence, Lemma \ref{framedleafinclusions2} 2) implies that $X_i'$  converges in $\mathcal {FP}^d$ to a point $X' \in L_X$. However,
$$d_{\mathcal {FP} ^ d}(F,X')=\lim_i d_{\mathcal {FP} ^ d}(F,X_i') \leq \lim_i d_{\mathcal {FP} ^ d}(F,X) = d_{\mathcal {FP} ^ d}(F,X),$$
so as $\mathcal V $ is a metric ball around $F$  that contains $X $,  we have $X' \in \mathcal V $. So, $X'$  can be used in the definition of $\delta (X)$,  implying that
$$\lim_i \delta (X_i)=\lim_i d_{\mathcal {FP} ^ d}(F,X_i')  = d_{\mathcal {FP} ^ d}(F,X') \geq \delta(X).$$
The reverse inequality is proved similarly: we assume that $X'\sim_1 X $ realizes the infimum defining $\delta (X) $, then we reroot the $X_i $ using Lemma \ref{framedleafinclusions2} 1) to produce elements $X_i'\in\mathcal {V}$ with $X_i\sim_1 X_i' $ and $X_i'\to X' $, and use the continuity of  the distance function $d_{\mathcal {FP}^d}$.	
\end {proof}
Moving toward the definition of the map $s $, we first define a map that selects a small subset of each $\sim_1 $-equivalence class in $\mathcal V$.  If $X \sim_1 F$, set 
$$A_X=\{F\}.$$
Otherwise, if $X \not \sim_1 F$ define 
\begin {align*}A_X &=\overline {\big\{Y\in\mathcal {V} \ \big| \  \ Y \sim_1 X \ \text { and } \ d_{\mathcal {FP} ^ d} (Y, F)<2\delta (X)\big\}} \subset\mathcal {FP} ^ d.
\end {align*}
Note that by continuity of $d_{\mathcal {FP} ^ d}$ and positivity of $\delta(X)$, the set $A_X$ is always nonempty. (However, if $X\sim_1 F$ then $\delta(X)=0$, so this latter definition of $A_X$ \emph{would} give the empty set, which is the reason we set $A_X=\{F\}$ in that case.) Moreover, $A_X$ is compact: for if $X=(M,e,D)$ and 
$$L : \mathcal FM \longrightarrow\mathcal {FP}^d, \ \ L(e')=(M,e',D)$$
is the framed analogue of the leaf inclusion map of Lemma \ref {leafinclusions}, then the conditions $L(e')\sim_1 X$ and $ d_{\mathcal {FP} ^ d} (L(e'), F)< 2\delta (X)$ define a pre-compact subset of $M$, whose closure is the (compact) preimage $L^{-1}(A_X)$. 

\begin {figure}
\centering
\includegraphics{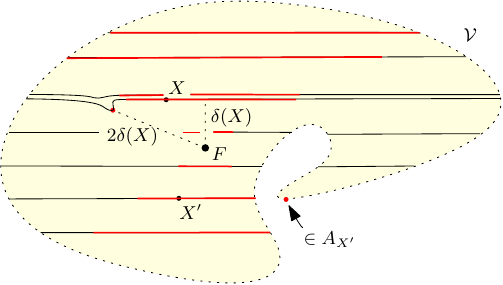}
\caption {Each $A_X \subset [X] \subset \mathcal V$ is drawn in red.}
\label {ux}
\end {figure}

\begin {claim}[Chabauty continuity of $X \mapsto A_X$]\label {ChabautyAX}
 After  possibly shrinking $\mathcal V$,  we have that if $X_i\rightarrow X$ in $\mathcal V$, then $A_{X_i} \rightarrow A_X$ in the Chabauty topology: every accumulation point in $\mathcal{FP}^d$ of a sequence $Z_i \in A_{X_i}$ lies in $A_X$, and every point of $A_X$ is the limit of a sequence $Z_i \in A_{X_i}$.
\end {claim}

 Figure \ref {ux} indicates two issues related to this continuity.

\begin {enumerate}
\item The shape of $\mathcal V$ may be a problem: in the figure, the indicated point of $A_{X'}$ cannot be approached from `below' within the red subsets. This is resolved by  shrinking $ V$. 
\item It is important to define $A_X$ as the closure of a set defined by a strict inequality, as opposed to a set defined by a non-strict inequality. For the leftmost point in the figure that is equivalent to $X$ has distance exactly $2\delta(X)$ from $F$, but cannot be approached from above by red points.
\end {enumerate}

\begin {proof}
Pick $\mathcal V'$ small enough so that for each $X\in\mathcal V' $, the closed $2\delta (X)$-ball around $X \in\mathcal {FP}^d$ is contained in the interior of $\mathcal V$. This is possible since $\delta(X)\rightarrow 0 $ as $X\rightarrow F $.  Note that with $\mathcal V'$ so chosen, the condition that $Y\in \mathcal V $, rather than just in $\mathcal{FP} ^ d $, is superfluous in the definition of $A_X$.

Assume that $X_i\rightarrow X$ in $\mathcal V'$. The fact that every accumulation point in $\mathcal{FP}^d$ of a sequence $Z_i \in A_{X_i}$ lies in $A_X$ follows immediately from Lemma~\ref {framedleafinclusions2} 2) and the continuity of $d_{\mathcal {FP} ^ d}$ and $\delta$.  

If $Y\in A_X $, pick a sequence $Y_i \sim_1 X$ with $d_{\mathcal {FP} ^ d} (Y_i, F)<2\delta (X)$ such that $Y_i\rightarrow Y $, as in the definition of $A_X $.  Passing to subsequences of $(X_i)$ and $(Y_i)$ and using Lemma~\ref {framedleafinclusions2} 2), pick for each $i $ some $Z_i\sim_1 X_i$ with 
\begin {equation}\label{close}d_{\mathcal {FP} ^ d} (Y_i, Z_i)<\frac 1i.\end{equation}
 Note that we can assume $Z_i\in \mathcal V'$, simply because the latter is open.
As $d_{\mathcal {FP} ^ d} $ and $\delta$ are continuous on $\mathcal{FP}^d$, for large $i$ we have
$$d_{\mathcal {FP} ^ d}(Z_i,X_i)<\delta(X_i),$$
so $Z_i \in A_{X_i}$. But $Z_i \rightarrow Y$, so we are done.
\end {proof}

The goal now is define a map  $s : \mathcal V \longrightarrow \mathcal {FP}^d$ by taking $s(X)$ to be the `circumcenter' of $A_X$, in the following sense.

\begin {lemma}[Circumcenters]\label {circumcenter}
Let $M$ be a Riemannian manifold, let $p\in M$ and $R=\mathrm{inj}_M(p)$.  Suppose that the sectional curvature of $M $ is bounded above by $\kappa$ on $B(p,R)$, and let $R'=\min\{R,\frac \pi{16\sqrt{\kappa}}\}$. If $A \subset B(p,R')$, the function
\begin {equation}\label {function}q\in M \longmapsto \sup\{d_{M}(q,a)\ | \ a\in A\} \end {equation}
has a {unique} minimizer $\odot(A)\in M$, called the \emph {circumcenter} of $A$. 
\end {lemma}

Here, $\mathrm {inj}_M(p)$ is the \emph {injectivity radius} of $M $ at $p $, and the notation for $\odot(A)$  reflects that it is the center of a minimal radius ball containing $A$. 

\begin {proof}
As $A \subset B(p,D)$ is pre-compact, the  supremum above  is always realized on  the closure of $A $. Also, pre-compactness implies that the function \eqref{function} is continuous and proper, so it must have at least one minimum. 

Suppose $q_1\neq q_2$ both realize the minimum, which we'll say is $D$, and let $q $ be the midpoint of the unique minimal geodesic connecting $q_1, q_2$.   Since $D $ is the minimum,  there must be some point $a \in \overline A$ with $$d(q,a) \geq D.$$  By definition of $D$, we also have $$d(q_i,a) \leq D \ \ \text{ for } i = 1,2. $$
In other words, we have a triangle $\Delta(q_1,a,q_2)$ in $M$ where the distance from $a$ to the midpoint $q$ of $q_1q_2$ is  at least both $d(a,q_1)$ and $d(a,q_2)$.   Note also that because $A\subset  B (p, R')$, we have 
$q_1,q_2 \in \overline {B(p,2R')}$, so all  side lengths of our triangle  are  at most
$4R' \leq \frac{\pi}{4\sqrt{\kappa}}.$

\begin {figure}
\centering
\includegraphics{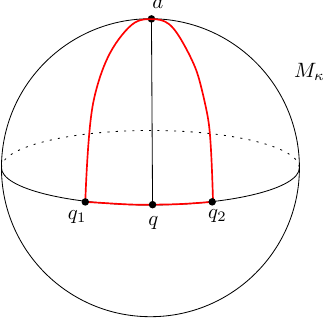}
\caption {A triangle  on a sphere in which a  midpoint-vertex distance is the same as the adjacent side  lengths.  This happens only if the side lengths are  at least half the diameter of the sphere.}
\label {modelfig}
\end {figure}

We claim this is impossible.  A theorem of Alexandrov, see \cite[1A.6]{Bridsonmetric},  implies that distances  between points in the triangle $\Delta(q_1,a,q_2)$ are  at most those in a `comparison triangle' with the same side lengths in the model space $M_\kappa$ with constant curvature $\kappa$.   (Here is where we use that $R'$  is less than the injectivity radius at $p$.) In particular, $d(q,a)$ is at most the corresponding midpoint-vertex distance of the comparison triangle.  And in $M_\kappa$, midpoint-vertex distances are always less than the maxima of the adjacent side lengths, unless $M_\kappa$  is a sphere and  some side of the triangle has length at least  half the  diameter of the sphere, i.e.\  at least $\frac{\pi}{2\sqrt{\kappa}}$, see Figure \ref {modelfig}.  But  our triangle   has side lengths  at most $\frac{\pi}{4\sqrt{\kappa}}$.\end {proof}

\begin {claim}\label {lastclaim}
After  possibly shrinking $\mathcal V $, there is a well-defined,  continuous map
$$s : \mathcal V \longrightarrow \mathcal {FP}^d, \ \ s (X) =\odot(A_X),$$ 
where $\odot$ is the circumcenter map on $L_X$, as in Lemma~\ref{circumcenter}.
\end {claim}

 This will  finish the proof of  Lemma~\ref {existencesection}, for  as defined above, $s$  is constant on $\sim_1$-equivalence classes and $X \sim_1 s(X)$.  So, it remains to prove the claim.

\begin {proof}
 Given $\epsilon >0  $, we can choose $\mathcal V $ small enough such that for every $X\in \mathcal V $, the $\sim_1$-equivalence class $[X] \subset  L_X $ has $d_{L_X}$-diameter at most $\epsilon$.  (The argument is almost  the same as that used  to prove Lemma~\ref {technicalities}.)  As injectivity radius and sectional curvature near the base point vary continuously in  $\mathcal{FM}^d$, we can assume that $\mathcal V $ is small enough so that by Lemma \ref {circumcenter},  each subset $[X] \subset L_X$ has a  well-defined circumcenter.  On such a  neighborhood $\mathcal V $,  the map in the statement of the claim is well-defined.
 
For continuity, suppose that $X_i\rightarrow X$ in $\mathcal V $. Let $$F_i : L_X \dashmap L_{X_i}$$ be the almost isometric maps of \eqref{framedais}. Combining (the proof of) Lemma \ref {framedleafinclusions2} and Claim \ref {ChabautyAX},  we have 
\begin {equation}\label {Hconv}F_i   ^ {-1} (A_{X_i}) \rightarrow A_X\end {equation}  in the Chabauty topology on closed subsets of $L_X$.  As these sets are all contained in a compact subset of $L_X $, Chabauty convergence means that  $F_i   ^ {-1} (A_{X_i}) $ and $ A_X$ are contained in $\epsilon $-neighborhoods of each other, with $\epsilon\rightarrow 0$ as $i\rightarrow \infty $. 

We  need to show that  if $ p_i  =\odot(A_{X_i})$  is the circumcenter  in $L_{X_i}$, then 
$$F_i^{-1}(p_i) \rightarrow \odot(A_X),$$  the corresponding circumcenter in $L_X$.   After passing to a subsequence, the points $F_i^{-1}(p_i)$  converge to a point $p\in L_X$. If $R_i$ is the minimum radius of a closed ball around $ p_i$ that contains $ A_{X_i}$, then as the $F_i$ are  almost  isometric, the sets $F_i^{-1}(B(p_i,R_i))$ Chabauty converge to a ball $B(p,R)\supset A_X$, where $R_i \rightarrow R$. But $A_X$ cannot be contained in a ball with radius less than $R$, since then a slight enlargement of such a ball would contain $F_i   ^ {-1} (A_{X_i}) $ for large $i$, contradicting the fact that $R_i \rightarrow R$. So, $p$ is a circumcenter for $A_X$.
\end {proof}

\subsubsection{Constructing the charts}

 Let us recall  our setup.  We have the two spaces  
\begin{align*}
\mathcal {P} ^ d &=\left\{(M,p,D) \ \Big | \ \substack{M \text{ a complete Riemannian } d\text {-manifold,} \\ p\in  M, \text { and } D \subset  M \text { a closed subset} \\ \text {such that }\not \exists \text { an isometry } f:M\longrightarrow M, \ f(D)=D } \right\}/ \text {isometry}\\
\mathcal{FP} ^ d &= \left\{(M,e,D) \ \Big | \ \substack{M \text{ a complete Riemannian } d\text {-manifold,} \\ e\in \mathcal FM, \text { and } D \subset \CF M \text { a closed subset} \\ \text {such that }\not \exists \text { an isometry } f:M\longrightarrow M, \ Df(D)=D } \right\}/\text {isometry},
\end{align*}
together with the projection map $\pi:\mathcal {FP}^d\longrightarrow\mathcal P^d$.    The relevant smooth-Chabauty topologies  are discussed in \S \ref{smoothChabauty}. Note that $\pi $ is an open map:  for if $(M,p,D)$ and $(M',p',D')$  are close in $\mathcal P ^ d$,  there is a (locally defined)  almost isometric map $f$ between them  that takes $p$ to $p'$, and then  given $e\in \mathcal FM_p$ we have  that $(M,e,D)$ and $(M',Df(e),D')$ are close in $\mathcal {FP}^d$.

Choose a point $O \in \mathcal P ^ d $, and a point $F\in\mathcal {FP}^d$  with $\pi (F) = O$. Let $\mathcal V$  be a neighborhood of $F$  that is small enough so that  Lemma \ref {existencesection}  applies,  and so that $\mathcal U=\pi (\mathcal V) $ satisfies the assumptions of Lemma \ref {technicalities}. Since the  equivalence relation  $\sim_1$ on $\mathcal V$ is a $\pi$-pullback, Lemma \ref {existencesection} gives a continuous map
$$s: \mathcal U / \sim_1 \  \longrightarrow\mathcal {FP}^d$$  such that $\pi(s([X]))\sim_1 X$ for all $X\in \mathcal U $.   We define a chart
 $$h:  \BR ^ d \times \, \mathcal U/\sim_1 \longrightarrow \mathcal {P}^d, $$
as follows. Each $Z\in \mathcal {FP}^d$ gives an \emph {exponential map} $$\exp_Z : \BR ^ d\longrightarrow L_{\pi(Z)} \subset \mathcal P ^ d,$$
  where if $Z=(M,e,D) $ then $\exp_Z$  is the  exponential map of $M$,  with respect to the frame $e\in \mathcal FM$, but composed with the leaf inclusions of Lemma \ref {leafinclusions},  so that it can be  considered as a map into $L_{\pi(Z)} \subset \mathcal {P}^d$.  Then
  $$h(v,[X]):=\exp_{s([X])}(v).$$

   Since injectivity radius  at the base point varies continuously in  $\mathcal P^d$,  after possibly shrinking $\mathcal V$  we may  fix $0<\epsilon<\frac 14$  such that for each $X \in \mathcal U$,
   the  map  \begin {equation}
  	\label {expeq}
\exp_{s([X])} : B(0,2\epsilon) \longrightarrow L_{s([X])}
  \end {equation} 
    is a diffeomorphism  onto its image.   We claim:
    
    \begin {claim}
    $h :  B (0,\epsilon) \times \mathcal U/ \sim_1 \, \longrightarrow\mathcal P ^ d $  is a homeomorphism onto its image.
    \end {claim}
    \begin {proof}
   For convenience, we work with the closed ball $\overline {B(0,\epsilon)} $.  We'll  show that 
       $$h : \overline {B (0,\epsilon) } \times \mathcal U/ \sim_1 \, \longrightarrow\mathcal P ^ d $$
        is a continuous, proper injection. As  $\mathcal P ^ d$  is Chabauty, this  will imply that $h$ is a homeomorphism onto its image.

   Injectivity follows immediately from the definition of $\epsilon$: the reason $2\epsilon $ appears in \eqref{expeq} is  to ensure the exponential maps stay injective on the closed balls $\overline {B(0,\epsilon)} $. 
   
 For continuity,  remember that $h(v,[X])=\exp_{s([X])}(v)$ and note that
    $$\exp : \mathcal {FP}^d \times \BR^d \longrightarrow\mathcal P ^d, \ \ (Z,v)\longmapsto \exp_Z(v)$$
     is continuous, since the Riemannian exponential map varies smoothly when the metric is varied smoothly, a consequence of the smooth variation of solutions to smoothly varying families of ODEs. Hence, $h$ is continuous.

  We now claim that $h$ is proper.  Assume that $(v_i,[X_i])$ is a sequence in $ \overline {B (0,\epsilon) } \times \mathcal U/ \sim_1$, and that
  $h(v_i,[X_i]) \rightarrow Y \in \mathcal P ^ d $.  As $$d_{L_{X_i}}\big(X_i,h(v_i,[X_i])\big)<\epsilon  ,$$ Lemma \ref {leafinclusions2} 2) implies that $(X_i)$ has a subsequence that converges to some $X\in \mathcal P^d$. By compactness, $v_i$ has a subsequence that converges to some $v\in \overline {B (0,\epsilon)} $, so  the sequence $(v_i,[X_i])$  is pre-compact in $ \overline {B (0,\epsilon) }\times \mathcal U/ \sim_1$.
    \end {proof}

  We now want to show that the set of all maps $h$ constructed as above is a foliated atlas for $\mathcal  P ^ d $.  The key  is the following lemma:
  
 \begin {lemma}\label {placquepre}
Suppose that $\mathcal U$ and $\epsilon$ are chosen to be small enough so that $\sim_1$ is an equivalence relation on the image $h(B (0,\epsilon)\times \mathcal U/ \sim_1)$. Then if
$$(v,[X]), (w,[Y])\in B (0,\epsilon)\times \mathcal U/ \sim_1,$$  we have that $[X]=[Y] \iff h(v,[X]) \sim_1 h(w,[Y]) .$
   \end {lemma}
\begin {proof}
 The forward direction is immediate, since if $s([X])=(M,e,D)$ then $h(v,[X])=(M,\exp_e(v),D)$ and $h(w,[Y])=(M,\exp_e(w), D)$. Since $\epsilon < \frac 14,$ we have $v,w\in B\left(0,\frac 12\right )$.  So $d_M\left ( \exp_e(v),\exp_e(w)\right) <1 $, as desired.

 Conversely,  suppose $h(v,[X]) \sim_1 h(w,[Y])$. Then
$$X \sim_1 \pi(s([X])) \sim_1 h(v,[X]) \sim_1 h(w,[Y]) \sim_1 \pi(s([Y])) \sim_1 Y.\qedhere$$
\end {proof}

\noindent Lemma \ref{placquepre} implies that transition maps between the charts $h$ have the form
$$t : B \times \mathcal U / \sim_1  \longrightarrow B' \times \mathcal U' / \sim_1, \ \ t(v,[X])=(t_1(v,[X]),t_2([X])),$$
where $B,B'$ are neighborhoods of the origin  in $\BR ^ d $ and $\mathcal  U, \mathcal U'$ are open in $\mathcal P^d$.    Furthermore, for each fixed $[X]$, the map $v \mapsto t_1(v,[X])$ is a transition map between two exponential maps for the same Riemannian manifold $M$, but taken with respect to different base frames. So, $t_1$ is smooth in $v$.  As $X$ varies in $\mathcal P^d$, the Riemannian manifolds $M$ vary smoothly, and the base frames vary continuously, so the $v$-partial derivatives of $t_1$ also vary continuously.

Therefore, $\mathcal P^d$ is a foliated space. The charts $h :  B (0,\epsilon) \times \mathcal U/ \sim_1 \, \longrightarrow\mathcal P ^ d $ above are chosen exactly so that when $[X_i]\to [X]$ in $\mathcal U / \sim_1$,  the pullback metrics $g_{X_i}$ on $B (0,\epsilon)$  converge in the $C^\infty$ topology to $g_X$. The point  is just that if a sequence of framed manifolds converges in the framed smooth topology, then the  almost isometric maps of Equation \eqref{framedalmost} in \S \ref{vectored}  can be chosen to almost commute with the exponential maps  in the small neighborhoods of the base frames.
 Finally, by Lemma \ref{placquepre},  the leaf equivalence relation of $\mathcal P^d$ is generated by $\sim_1$, so leaves are obtained  as promised by fixing $(M,D)$ and varying the base point $p\in M$.

\subsection{Application: an ergodic decomposition theorem}

A  unimodular probability measure $\mu$ on $\mathcal M^d$ is \emph{ergodic} if whenever $B$ is saturated and Borel, either $\mu(B)=0$ or $\mu(B)=1$.  Here, we use  the desingularization theorem (Theorem \ref{desingularizing}) to show that any unimodular probability measure on $\mathcal M^d$  can be expressed as an integral of ergodic such measures.

Ergodic decomposition theorems are usually proved in one of two ways.   Phrased in our context, one of the usual approaches is to disintegrate $\mu$ with respect to the $\sigma$-algebra of saturated subsets, see \cite{Fremlinmeasure4}, and then prove that the  conditional measures are ergodic.  The other approach  uses that ergodic  probability measures are the extreme points of the  convex set of all unimodular probability measures, and then appeals to Choquet's theorem \cite{Phelpslectures}.

Neither of these approaches quite applies to unimodular measures on $\mathcal M^d$ in full generality.  For the first approach, the usual way to prove that the conditional measures are ergodic is to appeal to a pointwise ergodic theorem. While Garnett~\cite{Garnettfoliations} has proved a foliated ergodic theorem with respect to Brownian motion --- this could be applied in our setting  using Theorem \ref{desingularizing} --- her theory requires that the leaves of the foliation have uniformly bound geometry.   The problem with the  second approach is that  to our knowledge, there is currently no version of Choquet's theorem that applies in this generality.  Namely, the original requires that the convex set in question be compact, see \cite{Phelpslectures}, and more general versions such as Edgar's \cite{Edgarnoncompact} require separability assumptions and that the underlying Banach space has the Radon-Nikodym property.

Our approach is to use Theorem \ref{desingularizing} to reduce the problem to an ergodic decomposition theorem for Polish foliated spaces, and then to reduce that to a decomposition theorem for measures on a complete transversal.   Essentially, if one traces through all the arguments in the referenced papers,  in particular in \cite{Greschonigergodic}, the argument does boil down to Choquet's theorem, but considering measures on the transversal allows one to use Varadarajan's compact model theorem \cite[Theorem 3.2]{Varadarajangroups} (or rather, the easier countable case thereof) to reduce  everything to the  case of an ergodic decomposition for a countable group acting by continuous maps on a compact metric space, which is a  setup that Choquet's theorem \cite{Phelpslectures} can handle.

\begin{proposition}[Ergodic decomposition]\label {ergodicdecomposition}
Let $\mu$ be a unimodular Borel probability measure on $\mathcal M^d$. Then there is a standard probability space $(X,\nu)$ and a  family  $\{\eta_x \ | \ x\in X\}$ of ergodic unimodular Borel probability measures on $\mathcal M^d$ such that for every Borel $B \subset \mathcal M^d$,  the map $x \mapsto \eta_x(B)$ is Borel and 
$$\mu(B) = \int \eta_x(B) \, d\nu.$$
\end{proposition}
\begin {proof}
Let $\mathcal P^d $ be the  Polish foliated space defined in Theorem \ref{desingularizing}. We say that a completely invariant probability measure $\mu$ on  $\mathcal P^d $   is \emph{ergodic}  if every Borel,  leaf-saturated subset has measure $0$ or $1$.   The leaf map of Proposition~\ref{leafBorel} pushes forward (ergodic) completely invariant measures on $\mathcal P^d $ to (ergodic) unimodular measures on $ \mathcal M^d$. In light of Theorem \ref{desingularizing}, it then suffices to show that for every  completely invariant Borel probability measure $\mu $  on $\mathcal P^d$, there is a standard probability space $(X,\nu)$ and a  family  $\{\eta_x \ | \ x\in X\}$ of ergodic completely invariant Borel probability measures on $\mathcal P^d$ such that for every Borel $B \subset \mathcal P^d$,  the map $x \mapsto \eta_x(B)$ is Borel and 
$$\mu(B) = \int \eta_x(B) \, d\nu.$$

 Choose a complete Borel transversal $T $ for the foliated space $\mathcal P^d $, i.e.\ a Borel subset that intersects each leaf in a nonempty countable set, and assume $T$  is a Polish space. (The existence of such a $T$ follows from the fact that $\mathcal P^d $ is Polish, and from the foliated structure.)  The leaf equivalence relation restricts to an equivalence relation $\sim$ on $T $ with countable equivalence classes, and $T$ with its Borel $\sigma$-algebra  is a standard Borel space.  So by Feldman-Moore \cite{Feldmanergodic}, $\sim$  is the orbit equivalence relation  of some Borel action $G \circlearrowright T$  of a countable group $G$.  

 The measure $\mu$  is the result of integrating the Riemannian measures on the leaves of $\mathcal P^d $  against a holonomy invariant transverse measure $\lambda$ on $T$.  Since the  action of the holonomy groupoid generates $\sim$ and preserves $\lambda$, the action $G \circlearrowright T$ above must also preserve  $\lambda $, by Corollary 1 of \cite{Feldmanergodic}.  

We now apply the ergodic decomposition of Greschonig-Schmidt~\cite{Greschonigergodic} to the transverse measure $\lambda $. They  show that there is a standard probability space $(X,\nu)$ and a  family  $\{\xi_x \ | \ x\in X\}$ of $G$-ergodic Borel probability measures on $T$ such that for every Borel $B \subset  T$,  the map $x \mapsto \xi_x(B)$ is Borel and 
$$\lambda(B) = \int \xi_x(B) \, d\nu.$$
 Integrating each $\xi_x$ against the Riemannian measures on the leaves of $\mathcal P^d $  gives a completely invariant measure $\xi_x$ on $\mathcal P^d $, and it follows that for every Borel $B \subset \mathcal P^d$,  the map $x \mapsto \xi_x(B)$ is Borel and 
$$\mu(B) = \int \xi_x(B) \, d\nu.\qedhere$$
\end {proof}

\section{Compactness theorems}

\label{hypcompactsec}

Cheeger-Gromov's $C^{1,1}$-compactness theorem \cite{Gromovmetric}  states that  for every $c,\epsilon>0$, the set of pointed Riemannian $d$-manifolds $(M,p)$ such that
\begin{enumerate}
	\item $|K_M(\tau)| \leq c$ for every $2$-plane $\tau$,
	\item $\inj_M(p) \geq \epsilon >0$
\end{enumerate}
 is precompact with respect to Lipschitz convergence, and that the limits are $C^{1,1}$-manifolds, Riemannian manifolds where the metric tensor is only Lipschitz. Here, $K_M$ is the sectional curvature tensor and $\inj_M(p)$ is the injectivity radius of $M$ at $p$.  Variants of this theorem, see e.g. \cite[Chapter 10]{Petersenriemannian}, ensure greater regularity of the limits when the  derivatives of $K_M$ are bounded.  For instance, if we fix some sequence $(c_j)$ in $\BR$  and replace 1)\ by
 \begin {enumerate}
 \item[1')] $|\nabla^j K_M |\leq c_j$, for all $j\in \BN \cup \{0\}$,
 \end {enumerate}
 then  the space of pointed Riemannian $d$-manifolds $(M,p)$  satisfying 1')\ and 2)\ is compact  in the smooth  topology, i.e.\  as a subset of $\mathcal M^d$, see Lessa \cite[Theorem 4.11]{Lessabrownian}.   We also discuss a  similar compactness theorem in \S \ref{smoothsec}.

By the Riesz representation theorem and Alaoglu's theorem, the set of Borel probability measures on a compact metric space is weak* compact, so in particular we have weak* compactness for probability measures supported on the space of pointed manifold satisfying 1')\ and 2)\ above. As unimodularity is weak* closed,  this also gives compactness for unimodular probability measures.

For the Cheeger-Gromov compactness theorems, a lower bound on injectivity radius at the base point is necessary, as no sequence $(M_i,p_i)$  with $\inj_{M_i}(p_i)\to 0$ can converge in the smooth (or even lipschitz) topology.   In this section, however, we prove that for manifolds with pinched negative curvature, a condition on the injectivity radius is not necessary if we are only interested in the weak* compactness of the space of unimodular  probability measures.

\vspace{2mm}

 More precisely, fixing $a,b>0$ and a sequence $c_j \in \BR$, let $\mathcal M^d_{PNC} \subset \mathcal M ^ d$ be the set of all pointed $d$-manifolds $(M,p)$ satisfying 1')\ and 
 \begin {enumerate}
 \item[3)] $-a^2\leq K_M(\tau) \leq -b^2 < 0$ for every $2$-plane $\tau$,  
 \end {enumerate}
and let $\mathcal M^d_{PNC,inj\geq \epsilon}$ be the subset of all $(M,p)$ that satisfy 1'), 2)\ and 3). The former space $\mathcal M^d_{PNC}$ is not compact, but even so we have:

\begin {reptheorem}{hypcompact}[Compactness in pinched negative  curvature]
The space of unimodular probability measures on $\mathcal M^d_{PNC}$ is compact in the weak$^*$  topology.
 \end {reptheorem}
 
The point is that the $\epsilon $-thin part of a manifold $M$ with pinched negative curvature  only takes up a uniformly small proportion of its volume.  More precisely,

\begin{proposition}[Thick at the basepoint]\label{thickbase}
If $\mu$ is a unimodular probability measure on $\mathcal M^d_{PNC}$ and $\epsilon>0$,  there is some $C=C(\epsilon,d,a,b)$ such that
\begin{itemize}
	\item $\lim_{\epsilon \to 0} C=0$,
\item  $\mu(\mathcal M^d_{PNC,inj\geq \epsilon})
\geq 1-C.$
\end{itemize} 
\end{proposition}

Since each $\mathcal M^d_{PNC,inj\geq \epsilon}$ is compact, Prokhorov's theorem \cite[IX.65]{Bourbakiintegration}  implies that  the set of unimodular probability measures on $\mathcal M^d_{PNC}$ is weak* precompact in the space of all probability measures, and therefore compact since unimodularity is a closed condition. So, Theorem~\ref{hypcompact}  follows from Proposition \ref{thickbase}.
In fact, since any sequence $(M_i,p_i)\in \mathcal M^d_{PNC}$ such that $\inj_{M_i}(p_i) \to 0$ must diverge, Theorem~\ref{hypcompact} and  Proposition \ref{thickbase} are formally equivalent.

\vspace{2mm}

We defer the proof of Proposition \ref{thickbase} to the next section, and finish here with some remarks about Theorem \ref{hypcompact}.  First, as in the compactness theorems for pointed manifolds, control on the derivatives of sectional curvature is not necessary if one is willing to consider limits that are measures supported on $C^{1,1}$-Riemannian manifolds, and where the convergence is Lipschitz. The derivative bounds in 1')\  do not factor into  the proof of Proposition \ref{thickbase},  and are used only when appealing to the compactness of $\mathcal M^d_{PNC,inj\geq \epsilon}$. (See \cite[Ch 10]{Petersenriemannian}.)

 Second, suppose that $M_i$ is a sequence of  finite volume Riemannian manifolds and $\epsilon_i \to 0$  is a sequence such that
\begin {equation}\frac {\mathrm{vol}( (M_i)_{<\epsilon_i} )}{\mathrm{vol} \, M_i} \nrightarrow 0,\label {quotienteq}\end {equation}
 where $(M_i)_{<\epsilon_i}$  is the $\epsilon_i $-thin part of $M_i$, i.e.\  the set of points with injectivity radius less than $\epsilon_i$.  Then the  corresponding unimodular probability measures $(\mu_i/\mathrm{vol}(M_i))$, see \S \ref{urms}, form a sequence with no convergent subsequence.  So for example, a \emph {uniform} lower bound on curvature is required in Theorem \ref{hypcompact}, since  when the metric on a closed hyperbolic surface  is scaled by $1/i$,  the whole manifold will be $\epsilon_i$-thin  for some $\epsilon_i\to 0$.  A similar argument shows that there is no analogue of Theorem \ref{hypcompact} for flat manifolds.
  
\begin {example}\label {uppernegbound}The uniform negative upper bound curvature is also necessary.  To see this, construct Riemannian surfaces $S_i$ by cutting a hyperbolic surface along a closed geodesic with length $\epsilon_i\to 0$, and inserting a  flat annulus $A_i$ with  boundary length $\epsilon_i $ and width $1/\epsilon_i$ in between.  The surfaces $S_i $ have bounded volume, and the $\epsilon_i$-thin part of $S_i$ has volume at least some constant, so \eqref{quotienteq} holds. The metrics on the $S_i$ can then be perturbed so  that the metric is smooth everywhere, and slightly negative on the annuli $A_i$, without affecting \eqref{quotienteq}.
\end {example}

 Although there is no  general analogue of Theorem \ref{hypcompact}  in nonpositive curvature, there is  a similar compactness result for nonpositively curved locally symmetric spaces. In \S \ref{localsymcomp},  we will give an algebraic proof of this, and will also indicate how our geometric arguments  might be adapted to  this setting.  We will  also discuss the possibility of a universal theorem that  implies  both Theorem~\ref{hypcompact}  and its locally symmetric analogue.

\subsection{The proof of Proposition \ref{thickbase}}

\label{proofpropsec}
The idea is simple. One needs to show that in a manifold with pinched negative curvatures, the $\epsilon $-thin part takes up a small proportion of the volume when $\epsilon$ is small. Then one transfers this estimate to $\mu$ using unimodularity.  Of course, our manifolds and their thin parts may have infinite volume, so one needs a local version of `small proportion' that is robust enough to work in this setting. More precisely, we will show how to push volume from the $\epsilon $-thin part into a region near the boundary of the $\epsilon_0$-thin part, where $\epsilon_0$ is the Margulis constant, without incurring a large Radon-Nikodym derivative. 

\vspace{2mm}

Before starting the proof in earnest, we record the following facts, which should be well known to those familiar with the literature.

\begin {lemma}\label{comparisons}
Let $M$ be a simply connected Riemannian $d$-manifold with	curvature $-a^2 \leq K_M \leq  0 $. Below,  all geodesics have unit speed.
\begin{enumerate}\item If $\alpha,\beta$ are geodesics that intersect at $\alpha(0)=\beta(0)$ with angle $\theta$, then $$ d(\alpha(t),\beta(t))\leq {\theta} \, \frac{ \sinh(a t)}a. $$
\item If $\alpha,\beta$ are geodesics that both intersect a geodesic $\gamma$ orthogonally at $\alpha(0)$ and $\beta(0)$, and $\frac{d\alpha}{dt}(0),\frac{d\beta}{dt}(0)$ are parallel vectors along $\gamma$, then
$$ d(\alpha(t),\beta(t))\leq d(\alpha(0),\beta(0))\, \cosh(at).$$

	\item If $\xi \in \partial_\infty M$ and $\alpha,\beta$ are geodesics such that $\alpha(0)$  and $\beta(0)$ lie on a common horosphere around $\xi$ and $\lim_{t\to -\infty} \alpha(t) = \lim_{t\to -\infty} \beta(t) = \xi$, \begin{equation*}
\label{spread}  d\left(\alpha(t),\beta(t) \right ) \leq   d(\alpha(0),\beta(0))  \, {e^{a t}},
\end{equation*}

\item Given $\epsilon, T>0$,  there is some $\delta=\delta(\epsilon, T, a) >0$  such that if $\alpha,\beta$ are geodesics and $d(\alpha (t),\beta (t)) \leq \delta$  for all $t\in [0,1]$, then
 $$d(\alpha (t),\beta (t)) \leq \epsilon, \ \ \forall t\in [0,T].$$
\end{enumerate}
\end {lemma}
\begin {proof}
 1) is an immediate consequence of Toponogov's theorem. 2) follows from Berger's extension of Rauch's comparison theorem \cite[Theorem 1.34]{Cheegercomparison}, by interpolating between $\alpha $ and $\beta $ by a one parameter family of geodesics $\alpha_s(t) $  such that $\frac {d\alpha_s}{dt} |_{t=0}$ is a parallel vector field along $\gamma.$
 
  For 3), let $\overline{\Delta}$ be a triangle in the model space $\BH^2_{-a^2}$ with curvature $-a^2$, chosen with one ideal vertex $\overline{\xi}$ and so that the other two vertices lie on a common horocycle centered at $\overline{\xi}$, at a distance of $d(\alpha(t), \beta(t))$ from each other.
 Parametrize the infinite sides of $\overline{\Delta}$ by arc length using $\overline{\alpha},\overline{\beta} : (-\infty, t] \longrightarrow \BH^2_{-a^2}$.
By Toponogov's theorem\footnote{This ideal version of Toponogov's theorem follows the usual one since $\overline{\Delta}$  is a limit of comparison triangles associated to $\alpha(t), \beta(t), \alpha(T) $, as $T \to -\infty.$ For $T<<0$, these triangles will be almost isosceles (since $\alpha(t),\beta(t)$ lie on a common horosphere) which implies that the limit triangle in $\BH^2_{-a^2}$ will have two vertices on a horocycle centered at the other vertex.},  $d(\overline{\alpha}(0),\overline{\beta}(0)) \leq d(\alpha(0),\beta(0))$. If $\gamma$ is the geodesic in $\BH^2_{-a^2}$ from $\overline{\alpha}(0)$ to $\overline{\beta}(0) $, we can flow every point on $\gamma$ away from $\overline{\xi}$ a distance of $t$ to produce a path $\gamma_t$ from $\overline{\alpha}(t)$ to $\overline{\beta}(t)$. In the upper half plane model for $\BH^2_{-a^2}$ with $\overline{\xi}=\infty$, the path $\gamma_t$ is created by dividing all the $y$-coordinates of the points on $\gamma$ by $e^{at}$. Hence, $\mathrm{length} (\gamma_t) = e^{at} \, \mathrm{length} (\gamma)$. This gives the estimate $$d(\alpha(t),\beta(t)) = d(\overline{\alpha}(t),\overline{\beta}(t)) \leq  d(\overline{\alpha}(0),\overline{\beta}(0)) {e^{a t}} \leq   d(\alpha(0),\beta(0)) {e^{a t}}.$$

For 4), see Figure \ref{topfig}.
\end{proof}

 \begin {figure}
 \centering
 \includegraphics {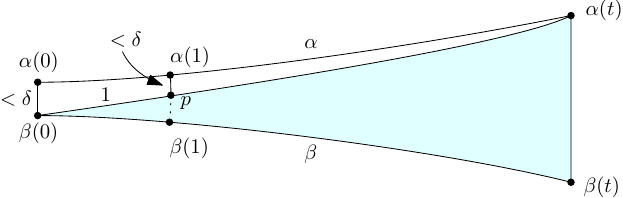}	
 \caption{Above, $d(\alpha(1),p)<\delta$ by convexity of the distance function. Then $d(\beta(1),p)<2\delta$, so applying Toponogov's theorem to the blue triangle gives an explicit upper bound for $d(\alpha(t),\beta(t))$ that depends only on $\delta,t$ and $a$. This upper bound is that which one gets in $\BH^2_{-a^2}$, so goes to zero as $\delta \to 0$. So for small $\delta$, we have $d(\alpha(t),\beta(t))\leq\epsilon$ for $t\leq T$.}
 \label {topfig}
 \end {figure}

\begin{proof}[Proof of Proposition \ref{thickbase}]
Let $\epsilon_0$  be the Margulis constant for manifolds $M$ with curvature bounds $-a^2\leq K_M(\tau)\leq -b^2 <0$, and let $\tilde M$ be the universal cover of $M$.  If $0< \epsilon \leq \epsilon_0$, and $M$ is such a manifold, consider the set $M_{<\epsilon}$ of all points $p\in M$ with $\inj_M(p)<\epsilon$. By the Margulis lemma, see \cite[\S 10]{Ballmannmanifolds}, the components $N$ of $M_{<\epsilon}$ come in two types. 
\begin {enumerate}
\item \emph{Margulis tubes}.  $N$ is the quotient of a tubular neighborhood $\tilde N \subset \tilde M$ of a geodesic $\tilde \gamma$ in $\tilde M$ by an infinite cyclic group $\Gamma$ of hyperbolic-type  isometries that stabilize $\gamma$. Hence, $N$ is a tubular neighborhood of a closed geodesic in $M$, which we call its \emph{core geodesic}.
\item \emph{Cusp neighborhoods}. $N$ is the quotient of an open set $\tilde N \subset \tilde M$ by a virtually nilpotent group $\Gamma$ of parabolic isometries that all have a common fixed point $\xi\in \partial_\infty \tilde M$.
\end {enumerate}

In both cases, the closure of $N$ is a codimension-0 submanifold of $M $ that has piecewise-smooth boundary. To see this, given any compact set $K \subset \tilde M$, proper discontinuity implies that $\tilde N\cap K$ is a finite union of sets of the form $$U_{g}\cap K, \ \text{where} \ U_g=\{x\in \tilde M \ | \ d(x,g(x))< \epsilon\},$$ and the $g$ are deck transformations. So, it suffices to show that the frontier $\partial U_g$ of each $U_g$ is a smooth submanifold.  But $\partial U_g$  is cut out by the equation $d(x,g(x))=\epsilon$, and we claim that $\epsilon$ is a regular value for $x \mapsto d(x,g(x))$. 

As $\tilde M$ is a simply connected manifold with negative curvature, it is uniquely geodesic, so the distance function of $\tilde M$ is smooth off the diagonal.  But $g $ has no fixed points, so this means that  the map $x\mapsto d(x,g(x))$ is smooth. Suppose that $x\in \tilde M$ is a critical point for this map, and let $\gamma: [0,D]\longrightarrow \tilde M$ be the geodesic from $x$ to $g(x)$.  As we start to move along $\gamma$ from $x$, the distance $d(x,g(x))$ is constant to first order. So, by the first variational formula, $dg(\gamma'(0))=\gamma'(D)$. Hence, the biinfinite geodesic extending $\gamma$ is invariant under $g$. This would mean that $\gamma$ is the axis of minimal displacement for $g$, which is impossible since $g$ translates points on $\gamma$ by $\epsilon$, but has translation length on $\tilde M$ less than $\epsilon$.

Let $M_{<\epsilon}^\circ$  be the subset obtained from $M_{<\epsilon}$  by removing the core geodesics of all  Margulis tubes, let $\partial M_{<\epsilon}$  be the boundary of the closure of $M_{<\epsilon}$, and define $M_{<\epsilon_0}^\circ$ and $\partial M_{<\epsilon_0}$  similarly. Note that there is a natural foliation of $M_{<\epsilon_0}^\circ$ by (the interiors of) geodesics with one endpoint on $\partial M_{<\epsilon_0} $, and where the other end either terminates at an orthogonal intersection with the core of a Margulis tube, or is a geodesic ray exiting a cusp\footnote{In the universal cover $\tilde M$, the normal exponential map $\exp : N\gamma \longrightarrow \tilde M $ is a diffeomorphism for any geodesic $\gamma$. For a given $g$, the displacement function $x \mapsto d(x,g(x))$ is convex along (the orthogonal) geodesics, so each $U_g$ from before is a star-shaped neighborhood of $\gamma$. The picture is similar for cusp neighborhoods, except that now we consider the foliation of $\tilde M$ by geodesics with one endpoint at some $p\in \partial_\infty \tilde M$.}. We will call these \emph{geodesic leaves}.

\begin{claim}\label {disthin}
There is some $C=C(\epsilon,d,a)$, with $C\to \infty$  as $\epsilon \to 0$, such that the distance along any geodesic leaf from $\partial M_{<\epsilon}$ to $\partial M_{<\epsilon_0}$  is at least $C$.
\end{claim}
\begin {proof}
Fixing some $R>0$, we want to show that if $\epsilon$ is sufficiently small, then the distance along any geodesic leaf from $\partial M_{<\epsilon}$ to $\partial M_{<\epsilon_0}$ is at least $R$.   

 This is  easiest to show for components $N \subset \partial M_{<\epsilon}$  that are neighborhoods of cusps. If $\tilde N \subset \tilde M$  is a component of the preimage of $N$, the geodesic leaves of $N$ lift to geodesic rays in $\tilde N$ that limit to a point $\xi\in \partial_\infty \tilde M $. If $x\in \tilde N$, there is some deck transformation $g$ that is parabolic with $g(\xi)=\xi$ and $d(x,g(x))< \epsilon.$

Let $\alpha$ be the unit speed geodesic in $\tilde M$ with $\alpha(0)=x$ and $\lim_{t\to -\infty} \alpha(t)=\xi.$ The image $\beta=g\circ \alpha$ is  also a geodesic limiting to $\xi$, and for all $t$, the points $\alpha(t)$ and $\beta (t)$ lie on a common horosphere centered at $\xi$.  By Lemma \ref{comparisons} 3), 
\begin{equation}
\label{spread} d\left(\alpha(t),\beta(t) \right ) \leq \epsilon \, \frac{e^{a t}}a.
\end{equation} 
So as long as $t < C:=\frac 1{a} \log( \frac{a \epsilon_0}{\epsilon} )$, the element $g$ will move $\alpha(t)$  at most $\epsilon_0$. So, the distance along $\alpha$ from $\partial M_{<\epsilon}$ to $\partial M_{<\epsilon_0}$  is at least $C$.

Now consider a component of $M_{<\epsilon_0}$ that is a Margulis tube with core  geodesic $c$, which we lift to a neighborhood $\tilde N$ of a geodesic $\tilde c$ in $\tilde M$. Let $\Gamma$  be the group of deck transformations stabilizing $\tilde N$. Here, all the non-identity elements $\Gamma $ act as nontrivial translations along $\tilde  c$, coupled with the action of some element of $O(d-1)$ in the orthogonal direction. We claim
\begin {quote}
 	($\star$) \it  Given $C>0$, there is some $\epsilon>0$ such that when the  length of the core  geodesic $c$ is less than $\epsilon$, we have $d(c,\partial M_{<\epsilon_0}) \geq C.$  \label {quoteclaim}
\end {quote}

To prove this,  first use Lemma~\ref {comparisons} 1)  to choose  some $\alpha >0 $ small enough ($\alpha={a \epsilon_0}/({2\sinh(a)}) $ works) so that any two geodesic rays $\gamma_1(t),\gamma_2(t)$ in $\tilde M$  that  intersect at $t=0$ with an angle at most $ \alpha$  satisfy
\begin{equation} \label {stayclose}d(\gamma_1(t),\gamma_2(t))\leq \epsilon_0/2, \ \text { for all } t\leq C .\end{equation}
Since $O(d-1)$ is compact, there is some $n\in\BN$ such that any $A\in O(d-1)$ has a power $A^k$ with $k=k(A)\leq n$ that is close enough to the identity so that $$\angle (A^k(v),v)<{\alpha}, \ \ \text{for all } v\in \BR^{d-1}.$$  Finally, using Lemma \ref{comparisons} 2), let $\beta>0$ be small enough so that  any two geodesic rays $\gamma_1(t),\gamma_2(t)$ in $\BH^d$ that start out perpendicular to a third geodesic, parallel to each other and at distance at most $\beta $  from each other,   must also satisfy \eqref{stayclose}.  (For instance, take $\beta = \frac { \epsilon_0}{2\cosh(a)}.$) Now take $\epsilon={\beta \over n}$.

 If $g\in \Gamma$ is any isometry with translational part less than $\epsilon$, then for some $k\leq n$, the isometry $g^{k}$ has translational part at most $\beta$ and rotates vectors orthogonal to the core geodesic $c$ by at most an angle of $\alpha $.  From this, we see that $g^{k}$ has displacement at most $\epsilon_0$ everywhere on the $C$-neighborhood of $\tilde  c$. So in any component of $M_{<\epsilon_0}$, the distance from the core curve $c$ to the boundary $\partial M_{<\epsilon_0}$ is at least $C$ whenever the core has length less than $\epsilon$,  proving $(\star)$.

 To complete the proof of the claim, given $C>0$ we choose $\epsilon$ as in $(\star)$,  but using $C+1$  instead of $C $.    Using Lemma \ref {comparisons} 4), we may also assume that $\epsilon $ is small enough so that any two geodesics $\gamma_1(t),\gamma_2(t)$ in $\tilde M$ satisfy \begin {equation}\label {lastclaimequation} d(\gamma_1(t),\gamma_2(t))\leq \epsilon, \, \forall t \in [0,1] \ \ \Longrightarrow \ \ d(\gamma_1(t),\gamma_2(t))\leq \epsilon_0, \, \forall t \in [0,C+1] .\end {equation}
 
 We want to show that  the distance from $\partial M_{<\epsilon}$ to $\partial M_{<\epsilon_0}$ is at least $C$. If the distance from the core $c$ to  $\partial M_{<\epsilon}$  is less than $1$, then  we are done since we know  by  $(\star) $ that $d(c,\partial M_{<\epsilon_0})\geq C+1$.  So, we may assume that the distance from the core to $\partial M_{<\epsilon}$  is  at least $1$. 
 
If $\ell$ is a geodesic leaf in $M_{<\epsilon_0}$,  parameterized  by arc length so that $\ell(1)$ is on the  boundary $\partial M_{<\epsilon}$, there are two lifts $\gamma_1,\gamma_2$ of $\ell$ in $\tilde N \subset \tilde M$ with $d(\gamma_1(t),\gamma_2(t))\leq \epsilon$ throughout the preimage of $M_{<\epsilon}$,  so for all $t\in [0,1]$.   Therefore, \eqref{lastclaimequation}  implies that the injectivity radius along $\ell$ stays  less than $\epsilon_0$ for all $t \in [0,C+1]$, i.e.\  for at least a length of $C $ after exiting $M_{<\epsilon}$. \end {proof}

 \begin {figure}
 \centering
 \includegraphics {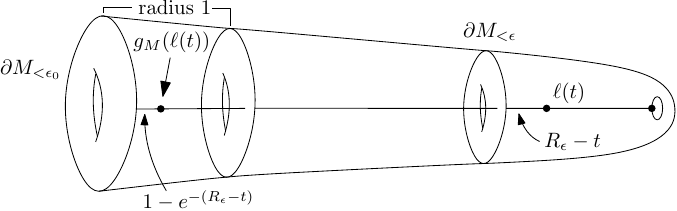}	
 \caption{The definition of the map $g_M$, pictured in a Margulis tube in dimension $3$. Note that the image of $g_M$ lies in a radius $1$ neighborhood of $\partial M_{<\epsilon_0}.$}
 \label {RPC}
 \end {figure}

We now  define a map
$$g_M : M_{<\epsilon}^\circ \longrightarrow M_{< \epsilon_0} $$
as follows. If $\ell$  is a geodesic leaf, parameterized by arc length so that $\ell(0)$ is on the core,  suppose that $\ell(R_\epsilon) \in \partial M_{< \epsilon}$ and $\ell(R_{\epsilon_0}) \in \partial M_{< \epsilon_0}$. Then  define 

$$g_M(\ell(t)) = \ell( R_ {\epsilon_0} - 1 + e^{b(t-R_\epsilon )}).$$ 

 In other words, $g_M$  is constructed  to map $M_{<\epsilon} ^\circ$  to a $1 $-neighborhood  of $\partial M_{<\epsilon_0} $, as shown in Figure \ref{RPC}.
 This $g_M$ is a  piecewise smooth homeomorphism onto its image, since  $R_\epsilon $ and $R_{\epsilon_0} $ are piecewise smooth functions of the point $\ell(t)\in M_{<\epsilon}^\circ$.
 
 \begin{claim}\label {voldistortion}
 	At every point $p$ where $g_M$ is smooth, we have 
\begin{equation*} | \det dg_M(p) |  \geq  D(\epsilon),\end{equation*}
 for some $D(\epsilon)$  that tends to infinity as $\epsilon\to 0$. 
 \end{claim}

Here, the determinant is calculated with respect to orthonormal bases  in the two  relevant  tangent spaces,  and so measures the volume distortion of $g_M$.

\begin{proof}

First, note that at $p=\ell(t_0)$, the map $g_M$ stretches lengths  in the direction of $\ell$ by a factor of $\frac d{dt}(R_ {\epsilon_0} - 1 + e^{b(t-R_\epsilon)})|_{t=t_0} = be^{b(t_0-R_\epsilon)} $.

Assume now that  we are working within a component of $\partial M_{<\epsilon_0}  $ that is a Margulis tube with core geodesic $c$. Given a point $p$ at which $g_M$  is smooth, let $\ell_s(t)$ be a one  parameter family of geodesic leaves satisfying:
\begin {enumerate}
\item $\ell_0(t_0) = p$ and $| \frac d{ds} \ell_s(t_0) | =1$.
\item $|\frac d{dt} \ell_s(t)| = 1$ for all $s,t$, and $\ell_s(0)\in c$  for all $s$,
\end {enumerate}
The path $s \mapsto \ell_s(t_0)$  passes through $p$ at $s=0$ and  moves along the boundary $\partial B(c,t)$ of the  radius $t_0$-neighborhood of $c$ in $M$, orthogonally to the leaves.  Its image under $g_M$  is just the path $s \mapsto \ell_s(R_ {\epsilon_0} - 1 + e^{a(t_0-R_\epsilon )})$.

The equation $J(t) = \frac d{ds} \ell_s(t) |_{s=0}$ defines a Jacobi field  along $\ell_0$, and 
$$dg_M(J(t)) = J(R_ {\epsilon_0} - 1 + e^{b(t-R_\epsilon )}).$$
As the path $s \mapsto \ell_s(t_0)$ lies in $\partial B(c,t_0)$, and $B(c,t_0)\subset M$ is convex, Warner's extension of the Rauch comparison theorem \cite[Theorem 4.3 (b)]{Warnerextensions}  implies that for $t \geq t_0$  we have\footnote{Here, $\frac{\cosh(b(t-t_0))}b$ is the length of a Jacobi field in $\BH^2_{-b^2}$  obtained by differentiating a unit speed geodesic variation $\gamma_s(t)$, where $s \mapsto \gamma_s(t_0)$ is also a unit speed geodesic, and is  perpendicular to all the geodesics $t \mapsto \gamma_s(t)$. Warner's theorem requires the sectional curvatures in $M$  to be less than those in the comparison space, the two Jacobi fields  to have the same length at $t=t_0$, and the path $s \mapsto \ell_s(t_0)$ to lie in a codimension one submanifold $S$ orthogonal to $\frac d{dt} \ell_s(t)$ all of whose principal curvatures are larger than those of  a corresponding submanifold $S'$ containing  $s \mapsto \gamma_s(t_0)$.  In our case,  both Jacobi fields  have  length one at $t=t_0$.  Convexity implies that the  principal  curvatures of $S=\partial B(c,t_0)$  are nonnegative, when calculated with respect to the outward normal $\frac d{dt} \ell_s(t)|_{t=t_0}$, so they are larger than those of the geodesic $s \mapsto \gamma_s(t_0)$,  which are zero. 
} $|J(t)| \geq {\cosh(b(t-t_0))}.$  So in particular,
\begin {align*} {|J(R_ {\epsilon_0} - 1 + e^{b(t_0-R_\epsilon )})|} &\geq   \cosh(b( R_ {\epsilon_0} - 1 + e^{b(t_0-R_\epsilon )} - t_0)) \\
&\geq  \frac 12 e^{b( R_ {\epsilon_0} - 1 - t_0) }\\
&\geq \frac 12 e^{b((R_ {\epsilon_0} - R_{\epsilon}) - 1)} e^{b(R_\epsilon - t_0)}.
\end {align*}
As $|J(t_0)|=1$, this means that $dg_M$  scales the length of $J(t_0)$ by at least a factor of $\frac 12 e^{b((R_ {\epsilon_0} - R_{\epsilon}) - 1)} e^{b(R_\epsilon - t_0)}$.   But above, $J(t_0)$ can be taken to be any vector in $TM_p$ orthogonal to $\ell$,  by choosing the variation $\ell_s$  appropriately. So,
\begin {align*}
	| \det dg_M(p) |  &\geq b e^{b(t_0-R_\epsilon)} \cdot   \Big ( \frac 12 e^{b((R_ {\epsilon_0} - R_{\epsilon}) - 1)} e^{b(R_\epsilon - t_0)} \Big )^{d-1} \\
	& \geq  \frac b{2^{d-1}} e^{b((R_ {\epsilon_0} - R_{\epsilon}) - 1)} := D(\epsilon),
\end {align*}
 a constant that tends to infinity as $\epsilon \to 0$, by Claim \ref{disthin}.  The argument for a component of $\partial M_{<\epsilon_0}$  that is a neighborhood of a cusp  is almost exactly the same.  Instead of parameterizing the geodesic variations $\ell_s$ so that  for constant $t_0$, the path $s \mapsto \ell_s(t_0)$ lies along a metric  sphere around the core of the Margulis tube, we parameterize so that $s \mapsto \ell_s(t_0)$ is contained in a horosphere. Horospheres are $C^2$ \cite{Heintzegeometry} and convex \cite{Eberleingeodesic},  so one can still apply Warner's comparison theorem.\end{proof}

 We now use unimodularity to finish the proof of Proposition \ref{thickbase}.   Let $\mathcal M^d_{PNC,2}$  be the space of  all isometry classes of doubly pointed $d$-manifolds $(M,p,q)$ with pinched negative curvature $-a^2\leq K_M \leq -b^2<0$ and geometry bounds as in 1') at the beginning of the section. Define a Borel function $F : \mathcal M^d_{PNC,2} \longrightarrow \BR_+$ via
$$F(M,p,q)=\begin {cases}  | \det dg_M(p) | & \text { if } p \in M_{<\epsilon} \text { and }   d(g_M(p),q)\leq 1 \\ 0 & \text{otherwise}\end {cases}$$
Using Claim \ref{disthin}, fix some $\delta<\epsilon_0$ with $C(\delta)\geq 1$.  By definition, the image of $g_M$ lies in a radius $1$ neighborhood of $\partial M_{\epsilon_0}$, so the injectivity radius  at every $g_M(p)$ is at least $\delta$.  Setting $V(\delta,\kappa)$  to be the volume of a radius $\delta$  ball in the $d$-dimensional model space of constant curvature $\kappa$, we have that the volume of  a ball of radius $1$ around each $g_M(p)$ satisfies $$V(\delta,-b^2) \leq \volume( B_M(g_M(p),1) \leq V(1,-a^2).$$
So, using this and \eqref{voldistortion}, we compute:
\begin{align}
	\mu(\mathcal H^d_{< \epsilon}) 
&\leq \int_{(M,p)\in \mathcal M^d_{PNC,inj<\epsilon} } \frac{| \det dg_M(p) |}{D(\epsilon)} \, \frac{\volume( B_M(g_M(p),1) )}{V(\delta,-b^2)} \, d\mu \nonumber  \\
& = \frac {1}{D(\epsilon)V(\delta,-b^2)} \int_{(M,p)\in \mathcal M^d_{PNC,inj<\epsilon} } \int_{q\in M} F(M,p,q) \, d\volume \, d\mu  \nonumber \\ 
& =\frac {1}{D(\epsilon)V(\delta,-b^2)}  \int_{(M,p)\in \mathcal M^d_{PNC,inj<\epsilon} } \int_{q\in M} F(M,q,p) \, d\volume \, d\mu \label {eee}\\ 
& =\frac {1}{D(\epsilon)V(\delta,-b^2)}  \int_{(M,p)\in \mathcal M^d_{PNC,inj<\epsilon} } \int_{q\in M_{< \epsilon} \cap g_M^{-1}(B_M(p,1))} | \det dg_M(q) | \, d\volume \, d\mu \nonumber  \\ 
& \leq\frac {1}{D(\epsilon)V(\delta,-b^2)}  \int_{(M,p)\in \mathcal M^d_{PNC,inj<\epsilon} }  \volume (B_M(p,1))\, d\mu \label {fff} \\  
& \leq \frac {V(1,-a^2)}{D(\epsilon)V(\delta,-b^2)},\nonumber
\end{align}
where here $\mathcal M^d_{PNC,inj<\epsilon}$ is the set of all $(M,p) \in \mathcal M^d_{PNC}$ where $inj_M(p)<\epsilon$.
Equation \eqref{eee} is unimodularity, while \eqref{fff} is just the change of variables formula. The last line goes to zero as $\epsilon\to 0$, which proves Proposition \ref{thickbase}.
\end{proof}

\subsection{Locally symmetric spaces}
\label{localsymcomp}
Let $X$ be a  symmetric space of nonpositive curvature with no Euclidean factors,  and let $G  = \mathrm{Isom}(X)$. An \emph{$X$-manifold} is a quotient $X / \Gamma$, where $\Gamma<G $  is discrete and torsion free. Let $\mathcal M^X \subset \mathcal M^d$ be the subset of pointed $X$-manifolds.

\begin{reptheorem}{Xcompact}[Compactness for locally symmetric spaces]
The space of unimodular probability measures on $\mathcal M^X$ is weak*-compact. 
\end{reptheorem}

By Proposition \ref{urmirsprop},  there is a dictionary between unimodular measures on $\mathcal M^X $ and discrete, torsion free invariant random subgroups of $G$.   The space of all invariant random subgroups of $G$ is compact, since the Chabauty topology on  the set of closed subgroups of any  locally compact group is compact. As $G$  is semi-simple, c.f. \cite{Helgasondifferential}, it suffices to prove the  following proposition:

\begin {proposition} \label {discreteclosed}
 For IRSs  in a semi-simple Lie group $G $, `discrete' is a closed condition.  For discrete subgroups $H < G$, `torsion-free'  is a closed condition.
\end {proposition}

Note that  unlike the second sentence, the first is true only on the level of IRSs, since  there are always  discrete, cyclic subgroups of $G$ that  limit in the Chabauty  topology to  subgroups isomorphic to $\BR$.

 \begin {lemma}\label{connectedIRS}
If $G $ is a semi-simple Lie group,  every connected IRS $H\leq G$ is a normal subgroup of the identity component $G^{\circ} \leq G$.
 \end {lemma}

\begin{proof}
If $H<G$ is any ergodic connected IRS, the Lie algebra $\mathfrak h $ is a random $k$-dimensional subspace of $\mathfrak g$,  for some $k$, and the distribution of $\mathfrak h $ is invariant under the adjoint action of $G $ on the Grassmannian $\mathrm{Gr}(k,\mathfrak g)$.  Since $G^{\circ} $ has no nontrivial compact quotients,  applying the arguments of \cite[Lemmas 2,3]{Furstenbergnote} to $\mathrm{Gr}(k,\mathfrak g)$  instead of  the projective space $P(\mathfrak g)$, we see that the distribution of $\mathfrak h $ is concentrated on $\mathrm{Ad} \, G^{\circ} $-invariant subspaces of $\mathfrak g$. \end{proof}

 \begin {proof}[Proof of Proposition \ref {discreteclosed}]
 Any subgroup $H\leq G$ that is a Chabauty  limit of discrete groups  has a nilpotent identity component, c.f.\ \cite[Theorem 4.1.7]{Thurstongeometry}.   So, any IRS  that is a limit of discrete IRSs also  as this property. However, a semi-simple Lie group does not have any nontrivial nilpotent normal subgroups, so Lemma \ref {connectedIRS}  implies that a limit of discrete IRSs is discrete.
 
Next,  we show that `torsion-free'  is a closed condition within the space of discrete $H <G$. For suppose $H_n \to H$ are discrete and  the only torsion is in the limit, and take $1 \neq \gamma \in H$ with $\gamma^k=1$.  Picking $\gamma_n \in H_n$  that converge to $\gamma$,  the sequence $(\gamma_n^k)$  consists of nontrivial elements converging to $1$.  
  Aiming for a contradiction,  choose open balls $B_1, B_2 , \cdots \ni 1 $ in $G$ with $B_i \supset \overline{B_{i+1}}$. Exponentiating the $\gamma_n^k$ by  appropriate powers determined using exponential coordinates,  it follows that for any fixed $i$,  and large enough $n$,  there is an element of $H_n$  inside $B_i \smallsetminus \overline{B_{i+1}}$. This contradicts discreteness of $H$.
\end{proof}


 Using the same arguments as in the paragraph after the statement of Proposition \ref{thickbase}, Theorem \ref {Xcompact} is equivalent to the following:
\begin{proposition}\label {thickbaseX}If $\mu$ is a unimodular probability measure on $\mathcal M^X$ and $\epsilon>0$,  there is some $C=C(\epsilon,X)$ such that
\begin{itemize}
	\item $\lim_{\epsilon \to 0} C=0$,
\item  $\mu(\mathcal M^X_{inj \geq \epsilon })
\geq 1-C,$
\end{itemize} 
where $\mathcal M^X_{inj \geq \epsilon } $ be the set of pointed $X$-manifolds $(M,p)$ such that $ \inj_M(p)\geq\epsilon $. \end {proposition} 

While the algebraic proof of Theorem \ref {Xcompact} is quite direct, it is natural to ask whether there is a geometric proof of Proposition \ref{thickbaseX}. In particular, why is it true for locally symmetric spaces when it fails for general spaces of nonpositive, or even unpinched negative, curvature? (See Example \ref{uppernegbound}.)

The  biggest difference  is that when $X $ is a non-positively curved symmetric space without Euclidean factors, its Ricci curvature 
 is bounded away from zero,  and hence the same is true for any $X$-manifold. (In contrast, the surfaces of Example \ref{uppernegbound}  have points where the Ricci curvature is arbitrarily close to zero.)  Specifically, the Ricci curvature tensor of $X=G/K$  is a  constant negative multiple of the Killing form on $\mathfrak g$ \cite[Theorem B.24]{Ballmannlectures}, so is negative definite. There is  then some  constant $b=b(X)>0$  such that \begin {equation}
 	\mathrm{Ric}(v,v)\leq -b^2/(d-1), \ \ \forall v \in T^1X, \label {riceq} \end {equation}
since $X$ is homogenous. Here, $d$ is dimension,  and we prefer to write the upper bound as above since then  it follows that every $v$  is contained in a $2$-plane  with sectional curvature at most $-b^2$. 
We ask:

\begin{question}
 Consider the space $\mathcal M^d(a,b,c_j)$ of all $(M,p) \in \mathcal M^d$ with
\begin{enumerate}
	\item[\normalfont 1')] $|\nabla^j K_M|\leq c_j$ for all $ j\in \BN \cup \{0\},$
\item[\normalfont  3')] $-a^2 \leq K_M(\tau) \leq 0$, and $\mathrm{Ric}(v,v) \leq -b^2$, for  every $2$-plane $\tau$  and unit vector $v\in TM_p$,
\end{enumerate}
where $\mathrm{Ric}$ is the Ricci curvature. Is the set of unimodular probability measures on $\mathcal M^d(a,b,c_j)$ weak* compact?  More concretely, does  condition 3')	  imply a uniform bound $C=C(\epsilon,a,b,d)$ on the proportion  of volume that the $\epsilon$-thin part occupies in (say) a finite volume $(M,p)\in \mathcal M(a,b,c_j)$, with $C\to 0$  as $\epsilon\to 0$?
\end{question}

 A proof of such a compactness result would be possible, for instance, if one could prove (say, for Hadamard manifolds) that an upper bound on Ricci curvature gives volume comparison results   analogous to those given by Bishop-Gromov, see \cite[Lemma  7.1.4]{Petersenriemannian}, for Ricci lower bounds. 

We do expect that our proof in  pinched negative curvature can be adapted to  the locally symmetric setting---that is, to give a geometric proof of Proposition~\ref {thickbaseX}---without proving the volume comparison results mentioned above.   Much of the argument goes through unchanged, although it becomes more subtle to develop an analogue of the foliation by geodesic leaves when the curvature is only nonpositive.

\section {Unimodular random hyperbolic  manifolds}
\label{fingensec}

Recall that the \emph{limit set}  of a subgroup $\Gamma $ of $\isometries(\BH^d)$  is the subset of $\partial \BH^d \cong S^{d-1}$ consisting of all accumulation points  at infinity of any orbit of $\Gamma \actson \BH^d$.  In \cite{abert2012growth}, it is shown that if $\mu$ is an invariant random subgroup of $\isometries(\BH^d)$ without an atom at the identity, then the limit set of $H \subset \isometries (\BH ^ d) $ must have full limit set $\mu$-almost surely.   In two dimensions, this has the following corollary:

\begin {corollary}\label{fingen2dim}
Any unimodular random hyperbolic surface with finitely generated fundamental group is either $\BH^2$ or has finite volume.
\end {corollary}

 As usual, the translation between IRSs and URHSs is through  Proposition~\ref{urmirsprop}.  Recall that a unimodular  random manifold is a random element of $\mathcal M^d$ whose law is a unimodular  probability measure, and is  hyperbolic if its law is supported on hyperbolic manifolds. For those allergic to probabilistic language, the statement of the corollary is that if $\mu$ is a unimodular measure on $\mathcal H ^ 2 $  then $\mu$-a.e.\ pointed hyperbolic surface $(S,p)$ with  finitely generated $\pi_1$ is either isometric to $\BH^2$  or has finite volume.

  There is  an alternative way to prove Corollary \ref{fingen2dim}, using the No-Core Principle (Theorem \ref{nocore}).  Recall that the \emph{convex core} of a hyperbolic surface $S$ is the  smallest convex subsurface  whose inclusion is a homotopy equivalence. When $S$ has finitely generated fundamental group, its convex core is compact.  Alternatively, the ends of $S$ are geometrically either infinite volume {flares}  or finite volume  {cusps}.  Here, a \emph{flare} is cut off by a closed geodesic, and is isometric to half of the quotient of $\BH^2$ by the group generated by a single isometry of hyperbolic type. The convex core of $S$  is obtained by chopping off all flares at the bounding closed geodesics. See \cite{Benedettilectures} for details.

Define a function $f:\mathcal H^2 \longrightarrow \{0,1\}$ by setting $f(S,p)=1$  whenever $p$ is in the open radius $1$ neighborhood of the convex core of $S $.  We claim $f^{-1}(1)$ is open,  so that $f$ is Borel. If $f(S,p)=1$, then $p$  lies at distance  less than $1$  from some closed geodesic  $\gamma$ in $S$, by density of closed geodesics in the convex core.  Using the almost isometric maps defining the smooth topology, see \S \ref{smoothsec},  $\gamma$ can be transferred to a closed $(1+\epsilon)$-quasigeodesic in any nearby $(S',p')$ and then tightened to a closed geodesic  $\gamma'$ using the Morse Lemma \cite[Theorem 1.7, pg 401]{Bridsonmetric}.  This geodesic $\gamma'$  is contained in the convex core of $S'$, and will lie at a distance less than $1$ from $p
$  as long as $(S',p') \approx (S,p)$.

So, by Theorem \ref{nocore},  we have for $\mu$-a.e.\ $(S,p)$ that 
	 $$0<\volume_S \{q\in M \ | \ f(S,q)=1\}<\infty \implies \volume(S)<\infty.$$ 
As a hyperbolic surface with  finitely generated $\pi_1$ either is $\BH^2$ or has a finite volume core, this  proves Corollary \ref{fingen2dim}.

\vspace{2mm}

 As explained in the introduction, there are examples of unimodular random hyperbolic $3$-manifolds  with finitely generated fundamental group other than $\BH^3$ and  finite volume manifolds, e.g.\  cyclic covers  $\hat M$ of  closed hyperbolic $3$-manifolds $M$ fibering over the circle.  We now show  that in $3$-dimensions, every URHM with finitely generated $\pi_1$  looks coarsely like such $\hat M$.  To do this, though, we need to recall some  background about the geometry of ends.

 Suppose that $M $ is a hyperbolic $3$-manifold with  finitely generated fundamental group. Then $M$ is homeomorphic to the interior of a compact  $3$-manifold, by the Tameness Theorem  of Agol \cite {Agoltameness} and Calegari-Gabai \cite {Calegarishrinkwrapping}. Each end  $\mathcal E$ of $M$  then has a neighborhood that is a topological product $S \times (0,\infty)$, for some closed surface $S$, and can be classified geometrically according to its relationship with the convex core of $M$. Here, the \emph {convex core} of $M$  is the smallest convex submanifold of $M$  whose inclusion is a homotopy equivalence. When $M$  has no cusps\footnote{In general, cusps may split the topological ends of $M$ into `geometric ends' with smaller genus, which have a similar classification. See \cite{Matsuzakihyperbolic}.}, work of Bonahon \cite{Bonahonbouts} and Canary \cite{Canaryends} implies that either $\mathcal E$ has a neighborhood disjoint from the convex core,  in which case we call it a \emph{convex cocompact} end, or it has a neighborhood  completely contained in the convex core, in which case $\mathcal E $  is \emph {degenerate}.  See \cite{Kapovichhyperbolic}  for more details.

A  cyclic cover $\hat M$ of a mapping torus is  homeomorphic to $S \times \BR$, and both of its two ends are degenerate---the convex core of $\hat M$ is the entire manifold. In ABBGNRS \cite[\S 12.5]{abert2012growth}, we constructed more general examples of IRSs that give doubly degenerate unimodular random hyperbolic structures on $S \times \BR$; for instance,  our examples do not cover any  finite volume $3$-manifold.  

\begin {reptheorem}{fingen3dim}
 Every 	unimodular random hyperbolic $3$-manifold with finitely generated fundamental group  either is isometric to  $\BH^3$,  has finite volume,  or is a doubly degenerate hyperbolic structure on $S\times \BR $ for some finite type surface $S$.
\end {reptheorem}

The proof is another application of the No-Core Principle, and the idea is that any infinite volume hyperbolic $3$-manifold  that is not $\BH^3$ and is not a doubly degenerate hyperbolic structure on $S\times \BR $ has a geometrically defined `core'. Intuitively, the core should just be obtained by cutting off the ends of the manifold, but  the difficult part is  doing this in a canonical enough way that for as $(M,p)$  varies through $ \mathcal H^3$, the condition that $p $ lies in the core is Borel. For simplicity, we'll first prove Theorem \ref{fingen3dim} only for manifolds  with no cusps, and then at the end we will make some brief comments about the modifications needed to extend to the general case.

Fix  once and for all some $\epsilon>0$  less than the Margulis constant and let $M $ be a hyperbolic  $3$-manifold with no cusps. The \emph {$\epsilon $-electric distance} between two points $p,q$ in $M$  is the infimum  over all smooth paths $\gamma$ joining $p,q$ of  the length of the intersection of $\gamma$ with $M_{\geq \epsilon} $,  the $\epsilon $-thick part of $M$.
  Given $R>0$, an \it $R$-core \rm for $M$ is a compact $3 $-dimensional submanifold $N \subset M$  such that 
\begin {enumerate}
\item  the diameter of $N $ is less than $ R$,
\item $N$  is contained in an open radius $1$  neighborhood of the convex core of $M$, 
\item  the component $E \subset M \smallsetminus N$  facing each component $S\subset \partial N$  is a neighborhood of an end that is homeomorphic to $S\times \BR$,
\item if $E$, as in 3), is a neighborhood of a convex cocompact end of $M$, then $E$  lies completely outside the convex core of $M$,
\item if $E$, as in 3), is a neighborhood of a degenerate end of $M$ and $p\in E$ has  $\epsilon $-electric distance  more than $R$  from $N$, there is a level surface $\Sigma$ in $E \cong S \times\BR$  that passes through $p$ and has $\epsilon$-electric diameter  less than $R$.
\end {enumerate}

 Here, a \emph {level surface} of $S \times \BR$ is any embedded surface isotopic to a fiber $S\times \{t\}$.    The point of the definition above is that $N$ is a small-diameter `core' for $M$,  obtained topologically by chopping off the ends of $M$. Conditions  2), 4) and 5)  require that convex cocompact ends  are chopped off near the convex core boundary, and  the removed  neighborhoods of degenerate ends have cross-sections with small electric diameter.
 
Every hyperbolic $3 $-manifold with finitely generated $\pi_1$  and no cusps, that is not isometric to $\BH^3$, has an $R$-core for some $R>0$.  Namely,  the Tameness Theorem gives  such an $N$  satisfying 1) and 3)  for some $R$.   The  complement of the convex core of $M$  consists of product neighborhoods of the convex cocompact ends \cite[II.1.3]{Canarynotes}, so we may pick $N$ to satisfy 2) and 4).  Finally, Canary's Filling Theorem \cite{Canarycovering} states that there is a neighborhood $E$ of each degenerate end of $M$   such that  $E$  is homeomorphic to $ S \times \BR $  and is exhausted by the images of \emph{simplicial hyperbolic surfaces}\footnote{A  simplicial hyperbolic surface is a  map from a triangulated surface that is totally geodesic on each triangle, and where  the total angle around each vertex is at least $2\pi$, \cite{Canarycovering}.} in the homotopy class of the fiber, on which no essential simple closed curves of length less than $\epsilon$  are null-homotopic in $M$.  By the Bounded Diameter Lemma \cite[Lemma 4.5]{Canarycovering}, such surfaces have $\epsilon $-electric diameter  bounded above by some  constant depending only on $\epsilon $ and $\chi(S)$. So increasing $R$ and enlarging $N$,  we have 5) using Freedman-Hass-Scott \cite{Freedmanleast} to replace the simplicial hyperbolic surfaces  by embedded level surfaces, compare with \cite[Corollary 3.5]{Biringerranks}.

\begin{proposition}[Borel-parametrized cores]\label {hypcores}
Suppose $M\not \cong \BH^3$ is a hyperbolic $3 $-manifold with finitely generated fundamental group and no cusps, let $R>0$ and $C_R(M)$ be the union of all $R$-cores of $M$. Then:
\begin {enumerate}
\item  Unless $M$ is  a doubly degenerate  hyperbolic metric on $S \times \BR$, for some closed surface $S $, the subset $C_R(M) \subset M$  has finite, nonzero volume for sufficiently large $R $.
\item If we set $C_R(M)=\emptyset$  for all other $(M,p)\in \mathcal H^3$,  the subset $\mathcal C_R \subset \mathcal H^3$  consisting of all $(M,p)$ with $p\in C_R(M)$ is Borel.
\end {enumerate}
\end{proposition} 

Both conditions 4) and 5)  in the definition of an $R$-core are necessary for the `finite volume'  part of this proposition.   Condition 4) is  needed in order to prevent a sequence of $R$-cores  from exiting  the degenerate end of a hyperbolic $3$-manifold homeomorphic to $S \times \BR$  that has one convex cocompact end and one degenerate end.  Less obviously, condition 5)  is needed to prevent  a sequence of $R$-cores  from exiting a  hyperbolic $3$-manifold   homeomorphic to the interior of a handlebody    whose single end is degenerate. Here, the  point is to express the  interior of the handlebody as $S \times \BR$, where $S$  is a  closed, orientable surface with a single puncture, and then take appropriate `cores' of the form $\hat S \times [t-1,t+1]$, where $\hat S$  is  obtained from $S$ by truncating its cusp.

 Deferring the proof, let's understand how Proposition~\ref {hypcores} implies Theorem~\ref{fingen3dim} in the no-cusp case.  Suppose $\mu$ is a unimodular measure on $\mathcal H^3$, and apply the No-Core Principle (Theorem \ref{nocore}) to the characteristic function of each $\mathcal C_R$.  We obtain that for each $R>0$, the following holds for $\mu$-a.e.\ $(M,p)$: $$0< \volume \, C_R(M)<\infty \implies \volume \, M<\infty.$$  Taking a countable union of measure zero sets, we have for $\mu$-a.e.\ $(M,p)$ that $$0< \volume \, C_R(M)<\infty \text{ for some } R\in \BN \ \  \implies \ \ \volume \, M<\infty.$$
 So by Proposition~\ref {hypcores},  we have that $\mu$-a.e.\  $(M,p)$ with finitely generated fundamental group and no cusps is either finite volume (i.e.\ closed), $\BH^3$ or a  doubly degenerate hyperbolic $3$-manifold  homeomorphic to $S \times \BR$.

\subsubsection{Proof  of Proposition \ref {hypcores}}
 
 We first prove 1), so  assume that $M$  has finitely generated $ \pi_1$, no cusps  and is not isometric to $\BH^3$.  As mentioned above, $M$ admits an $R$-core for some $R $. So, for  large $R $ the subset $C_R(M)$  has nonzero volume. Our  goal is to show that $C_R(M)$ is  always bounded in $M$, so always has finite volume. 
 
 Suppose on the  contrary that there is a sequence $C_i$  of $R$-cores of $M$ that exits an end $\mathcal E$ of $M$. By condition 2) in the definition, $R$-cores lie near the convex core, so $\mathcal E$  is a degenerate end of $M$. Increasing $R $ if necessary, pick a neighborhood $E$ of $\mathcal E$  that is homeomorphic to a product $S\times\BR $ and is exhausted by level surfaces with $\epsilon $-electric diameter at most $R$. 
 
 We claim that $M$ is  a doubly degenerate hyperbolic structure on $S\times\BR$.  For large $i $, the $\epsilon $-electric distance from $C_i \subset E$ to the frontier of $E $ is at least $6R$.  This electric distance is realized by a path $\gamma $ in some component $D \subset M \smallsetminus C_i$.  Then $D \cong S'\times\BR$  for some closed surface $S'$, and since $\gamma$  is contained in the convex core of $M$, this $D$  cannot be a neighborhood of a convex cocompact end by condition 4)  in the definition of an $R$-core. So, by condition 5), there  are surfaces $\Sigma'_1,\Sigma'_2 \subset M$  that are level surfaces of $D \cong S'\times\BR$,  have $\epsilon $-electric diameter less than $R$, and which  pass through points on $\gamma$  at $\epsilon $-electric distance $R$ and $5R$,  respectively,  from $C_i $. These surfaces must then be disjoint, so they bound a  submanifold $N$ homeomorphic to $S' \times [0,1]$. However, we can also pick a level surface $\Sigma$ in $E \cong S \times\BR$ with $\epsilon $-electric diameter less than $R$ that  passes through a point on $\gamma$  at $\epsilon $-electric distance $3R$  from $C_i $. This $\Sigma \subset N \cong S' \times [0,1]$, and is  incompressible since it is  incompressible in $E$,  which contains $N$.  Therefore, by Waldhausen's Cobordism  Theorem \cite{Waldhausenirreducible},  $\Sigma$  is a level surface of $N$, which in turn means that  $\Sigma'_1 $ and $\Sigma'_2$ both bound product neighborhoods of degenerate ends of $M$ \emph {on both sides}, one contained in $D$  and the other in $E$. Hence, $M$  is  a doubly degenerate  hyperbolic metric on $S\times\BR $.

For the second part of the proposition,  we need to show that a Borel subset  of $\mathcal H^3$  is defined if we require $(M,p)$  to satisfy the following three conditions:
 \begin {enumerate}
 \item[(A)] $M$  has no cusps,
 \item[(B)]  $\pi_1 M$ is finitely generated,
 \item[(C)]  $p$ lies in an $R$-core of $M$.
 \end {enumerate}
We'll  first show that (A) and (B) each define Borel  subsets, and deal with (C)  afterwards. Note that (C) doesn't make sense on its own, since $R$-cores are only defined for manifolds with   finitely generated $\pi_1$ and no cusps.

To see that (A)  defines a Borel set, we check that for each $R>0 $ and small $\epsilon>0$, the set of all $(M,p)$ where  there is a cuspidal $\epsilon$-thin part at distance at most $R$ from $p$ is closed, and then take a union over $R\in\BN$. For if $(M_i,p_i) \to (M_\infty,p_\infty)$, we can write $M_i = \BH^3/\Gamma_i$  in such a way that $\Gamma_i \to \Gamma_\infty$ in the Chabauty topology on subgroups of $\PSL_2\BC$, and the points $p_i$  are all projections of a fixed $\tilde p\in \BH^3$, see \cite[Chapter 7]{Matsuzakihyperbolic}. Cuspidal $\epsilon$-thin parts at distance at most $R$ from $p_i\in M_i$ then give elements $1 \neq \gamma_i\in \Gamma_i$ such that
\begin{enumerate}
	\item $ |tr \, \gamma_i|=2$ for all $i$,
\item  there are points $\tilde x_i \in \BH^3$ with $d(\tilde x_i,\tilde p ) \leq R$ and $d(\tilde x_i,\gamma_i(\tilde x_i))\leq 2 \epsilon$.
\end{enumerate}  After passing to a subsequence, $\tilde x_i \to \tilde x_\infty\in \BH^3$ and $\gamma_i \to \gamma_\infty \in \Gamma_\infty$, where $d(\tilde x_\infty,\tilde p ) \leq R$  and $ |tr \, \gamma_\infty|=2$.  Passing to the quotient, we have a cuspidal $\epsilon$-thin part at  distance at most $R$ from $p_\infty\in M_\infty$.

To prove that (B)  is a Borel condition, we use:

\begin{lemma}\label{Borelhyplemma} Fix a compact $3$-manifold $N_0$, a Riemannian  metric on $N_0$, and a constant $\lambda>1$.  Let $\mathcal S$ be the set of all $(M,p)\in \mathcal H^3$ that  admit a smooth embedding $f: N_0 \hookrightarrow M$ such that
\begin{enumerate}
\item the iterated derivative maps $D^kf : T^k N_0 \longrightarrow T^kM$ are locally $\lambda$-bilipschitz on the $1$-neighborhood of the zero section in $T^k N_0$, for $k=0,1,2$,
	\item the image $N=f(N_0)$ contains $p$,
\item  each component of $M \smallsetminus N$  is homeomorphic to $S \times \BR$,  for some closed surface $S$.
 \end{enumerate} 
Then $\mathcal S$ is a Borel subset of $\mathcal H^3$.
\end{lemma}

See \S \ref{smoothtopology} for  details about iterated derivatives.  The point of 1) is the following, though.
As shown in at the end of the proof of Theorem \ref{manifoldmetrizable}, bounds on iterated total derivatives up to order 2 give bounds  in fixed local coordinates for the $C^2$-norm of $f$. So by Arzela-Ascoli, if $f_i: N_0 \hookrightarrow M$  is a sequence of embeddings satisfying 1) and 2), then  after passing to a subsequence  we may assume that $f_i$ converges in the $C^1$-topology to some $C^1$-embedding $f : N_0 \longrightarrow M$. Working in normal-bundle coordinates in a regular neighborhood of  $\partial f(N_0)$, we can then move any $f_i$ (with $i$ large) by a small $C^1$-isotopy so that the image  agrees with $f(N_0)$.  In particular, \emph{the images of all such $f_i$  differ by small isotopies.}

\begin{proof}[Proof of Lemma \ref{Borelhyplemma}]
We'll write $\mathcal S $ for the set of all $(M,p)$  admitting an embedding as above. Fix $R>0$ and consider the set $\mathcal S_R$ of all $(M,p)$ such that there is an embedding $N_0 \hookrightarrow M$  satisfying 1) and 2), and also a compact submanifold $N' \subset M$ that contains the radius $R$-ball around $p$, and where  every component of $N' \smallsetminus int(N)$  is homeomorphic to $S \times [0, 1]$  for some closed surface $S$.
This $\mathcal S_R$ is open in the  smooth topology, since the approximate isometries defining the smooth topology (see \S \ref{smoothsec})  allow us to transfer compact submanifolds of $(M,p) $ to nearby $(N,q)$, with small metric distortion.

We claim that $\mathcal S = \cap_R \mathcal S_R$, which will imply $\mathcal S$ is Borel.  The forward inclusion is obvious, so assume $(M,p)$ is in the intersection. Then there is a sequence $R_i \to \infty$  such that if $f_i : N_0 \hookrightarrow M$, $N_i$ and $N_i'$ are the corresponding embeddings and submanifolds, we have $N_1' \subset N_2' \subset  \cdots $. Passing to a subsequence, we may assume by  1) and 2) that the  embeddings $f_i : N_0 \hookrightarrow M$ all differ by small isotopies. Hence,  the complement of $int(N_1)$, say, in every $N_i'$ is a union of topological products $S \times [0,1]$.  Consequently, each $N_{i+1}' \smallsetminus int(N_i')$ is   also a  union of products $S \times [0,1]$, so taking a union over $i$, the components of $M \smallsetminus N_1$ are homeomorphic to  $S \times \BR$. \end{proof}

To prove (B) defines a Borel set, apply the lemma and take a union over the countably  many homeomorphism types of compact $3$-manifolds $N_0$, over a  countable dense subset of the space of Riemannian metrics on a given $N_0$, and over $\lambda \in \BN$.  This proves that a Borel subset is defined  by the condition that there is  a submanifold $N\subset M$  containing $p$ whose complementary components are homeomorphic to products $S \times \BR$. By  the Tameness Theorem,  this is equivalent to $\pi_1 M$  being finitely generated.

 The proof that (C) defines a Borel set uses a more complicated version of Lemma~\ref{Borelhyplemma},   but the argument is  very similar. 
Fixing  a compact $3$-manifold $N_0$, a  Riemannian metric on $N_0$, and  a constant $\lambda >1$, let $\mathcal S$ be the set of all $(M,p) \in \mathcal H^3$ with no cusps and  finitely generated $\pi_1$, that admit a smooth embedding $N_0 \hookrightarrow M$ satisfying 1) and 2) of Lemma~\ref{Borelhyplemma}, and whose image is an $R$-core of $M$.

We claim that $\mathcal S$  is Borel.  To that end, fix $T>0$ and consider the set $\mathcal S_T$ of all $(M,p)$ with  finitely generated $\pi_1$ and no cusps, satisfying the following conditions. We require that there is an embedding $N_0 \hookrightarrow M$   with bilipschitz constant less than $\lambda$, whose image $N$ contains $p$ and satisfies
\begin {enumerate}
\item  the diameter of $N $ is less than $ R$,
\item $N$  is contained in an open radius $1$  neighborhood of the convex core of $M$.
\end {enumerate}
We also require that $N$ is contained in a compact submanifold $N' \subset M$  whose interior contains the closed radius $T$ ball around $N$, such that
\begin {enumerate}
\item[3)]  each component $E$ of  $N' \setminus int(N)$ is homeomorphic to $S\times [0,1]$,  for some closed surface $S$,
\item[4)] if a component $E$ of  $N' \setminus int(N)$ intersects the convex core of $M $, then through each point $p\in E$ that lies more than an  $\epsilon $-electric distance of $R$ from $M \setminus E$, there is a level surface $\Sigma$ in $E \cong S \times[0,1]$  that passes through $p$ that has $\epsilon$-electric diameter less than $R$.
\end {enumerate}

 We now claim that $\mathcal S_T$ is Borel.  As we have  shown above, the conditions that $M$ has no cusps, and that $\pi_1 M$  is  finitely generated are both Borel.  Furthermore, the approximate isometries defining the smooth topology (see \S \ref{smoothsec})  allow us to transfer compact submanifolds of $(M,p) $ to nearby $(N,q)$, with small metric distortion. So, the existence of an embedding $N_0 \hookrightarrow M$ as above that satisfies 1) and 3)  is an open condition.

To  deal with the other conditions, let $N>0$ and let $\mathcal H^3_N$  be the set of all $(M,p)\in \mathcal H^3$  where the injectivity radius of $(M,p)$  is bounded above by $N$  throughout the convex core of $M $.
   Each $\mathcal H^3_N$ is a  closed subset of $\mathcal H^3$, and the convex core of $M$ varies continuously as $(M,p)$ varies within  $\mathcal H^3_N$, by \cite[Proposition 2.4]{McMullenrenormalization}.   Using this and the  previous paragraph, one can see that 2) and 4) are also open conditions within $\mathcal H^3_N$, so $\mathcal S_T \cap \mathcal H^3_N$ is a Borel subset of $\mathcal H^3_N$ for all $N $. As any  hyperbolic $3$-manifold with finitely generated $\pi_1$ has  an upper bound for the injectivity radius over its convex core, \cite[Corollary A]{Canarycovering}, this expresses $\mathcal S_T$ as a union of Borel sets, so $\mathcal S_T$ is Borel.

We claim that $\mathcal S = \cap_T \mathcal S_T$, which will imply $\mathcal S$ is Borel.  The forward inclusion is obvious, so assume $(M,p)$ is in the intersection.  Arguing as in  the proof of Lemma \ref{Borelhyplemma}, we may assume that we  have a fixed embedding $f : N_0 \hookrightarrow M$  satisfying 1) and 2), and a  sequence $T_i \to \infty$ such that there are $N_i'$ whose interiors contain the closed $T_i$-ball around $N=f(N_0)$ and satisfy 3) and 4). Again as in the proof of Lemma \ref{Borelhyplemma},  the complementary components of $N$ are all  homeomorphic to products $S \times \BR$. But then conditions 3) and 4)  in the definition of an $R$-core  follow for $N$ from condition 4) above. For if a component of $M\smallsetminus N$  intersects the convex core, it must have the necessary level surfaces with bounded electric diameter (and is a neighborhood of a degenerate end). Otherwise, we have condition 4)  in the definition of an $R$-core.

 This proves that $\mathcal S$ is Borel, which proves (C) is a Borel condition, and Proposition \ref{hypcores} follows.

\subsubsection{The case with cusps}

 In the presence of cusps, the proof is the same, but more complicated. When $M$ has  finitely generated $\pi_1$, the topological ends of $M$ may be split by cusps into `geometric ends'.   More precisely, fixing $\epsilon>0$ we let $M_{np}$ be  the manifold with boundary that is the complement of the cuspidal  $\epsilon$-thin part of $M$. Then the  topological ends of $M_{np}$  are called \emph{geometrically finite} or \emph{degenerate}  depending on whether they have neighborhoods  disjoint from, or contained in,  the convex core of $M$.  The definition of an $R$-core is  basically the same, except that $M$  should be replaced by $M_{np}$, complementary components should  now be homeomorphic to $S \times \BR$, where $S$ is a surface with boundary, and the $\epsilon$  defining electric distance should be chosen smaller than that defining $M_{np}$.

Proposition \ref{hypcores} should now allow cusps, and exclude in 1) products $S \times \BR$, where $S$ is a finite type surface. Its  proof is the same,  as long as one keeps track of the relationship between  the new definition of $R$-cores and $M_{np}$. However, the increase in the already formidable amount of notation will be unpalatable.

\appendix

\section {Spaces of Riemannian manifolds}

In this appendix, we discuss the smooth topology  on the space of pointed Riemannian manifolds,  and related topologies on similar spaces. 

\subsection {Smooth convergence and compactness theorems}
As introduced in the introduction, let $\mathcal M ^ d $ be the set of isometry classes of  connected, complete, pointed Riemannian manifolds $(M,p) $.  
\label {smoothsec}
\begin {remark}
The class of all pointed manifolds is not a set.  However, every connected manifold has at most the cardinality of the continuum, so one can discuss $\mathcal {M}^ d $ rigorously by only considering manifolds whose underlying sets are identified with subsets of $\BR $.  The same trick works in all the related spaces below, so we will make no more mention of set theoretic technicalities.
\end {remark}

A sequence of pointed Riemannian manifolds $(M_i, p_i) $ \emph {$C^k$-converges} to $(M, p)$  if there  is a sequence of $C^k$-embeddings
\begin {equation} f_i:B(p,R_i)\longrightarrow M_i \label {fi}\end {equation} with $R_i \rightarrow\infty$ and $f_i (p) = p_i $,  such that $f_i ^*g_i\rightarrow g $ in the $C ^k $-topology, where $g_i,g $ are the Riemannian metrics on $M_i,M $.  Here,  $C ^k $-convergence of tensors is defined locally: in each pre-compact coordinate patch, the coordinates of the tensors and all their derivatives up to order $k$ should converge uniformly.   Note that each  metric $f_i ^*g_i$  is only partially defined on $M$, but their  domains of definition exhaust $M$, so  it still makes sense to say that  $f_i ^*g_i\rightarrow g $  on all of $M $. 
We say that $(M_i, p_i) \rightarrow (M, p)$ \emph {smoothly} if the convergence is $C ^ k $ for all $k \in \mathbb N $.

Whether the convergence is $C^k$ or smooth,  we will call such an $(f_i)$ a \emph {sequence of almost isometric maps} witnessing the convergence $M_i \rightarrow M$.  When the particular radii $R_i$  do not matter, we will denote our  partially defined maps by $$f_i : M \dashmap M_i,$$  where the notation indicates that each $f_i$  is only partially defined on $M$, but  any  point $p\in M$ is in the domain of $f_i$  for all large $i$.

Of course,  another way to define  $C ^ k$-convergence would be to require  that for every fixed radius $R>0$, there is a  sequence of maps $f_i:B(p,R)\longrightarrow M_i $  satisfying the  properties above.  To translate between the two definitions, we can restrict the $f_i$ in \eqref{fi} from $R_i$-balls to $R$-balls,  or in the other direction, we can take a diagonal sequence where $R$  increases with $i$. Most of the time, we will use the `fixed $R$'  perspective in this appendix.

In some references, e.g.\ Petersen \cite [\S 10.3.2]{Petersenriemannian}, $R$ is fixed and the maps $f_i $ are defined on open sets containing $B(p,R)$ and their images are required \emph{to contain $B(p_i,R)\subset M_i$}. Such restrictions on the image of the $f_i$ do not change the resulting $C^k$-topology, though, since for large $i$ the maps $f_i$ are locally $2$-bilipschitz embeddings, and we can appeal to the following Lemma.

\begin {lemma}\label {invert}
Suppose $M, N $ are complete Riemannian $d $-manifolds, $p\in M $.  If $f:B(p,R)\longrightarrow N $ is a smooth locally $\lambda $-bilipschitz embedding,  $$f(B(p,R))\supset B (f  (p),R/\lambda) . $$
\end {lemma}
\begin {proof}
Fix a point $q\in B (f  (p),R/\lambda)$ and let $\gamma : [0, 1]\longrightarrow N $ be a length-minimizing geodesic with $\gamma (0)= f (p)$ and $\gamma (1)=q$.  Let
$$\mathcal T =\{t\ | \ \exists \, \gamma_t: [0, t]\longrightarrow B(p,R) \text { with } \gamma_t (0) = p, \ f \circ \gamma_t =\gamma |_{[0, t]}\} \subset [0,1]. $$ 
The set $\mathcal T $ is a subinterval of $[0, 1] $ that contains $0 $.  It is open in $[0, 1]$, since $f $ is a local diffeomorphism.  We claim that it is also closed.  First, as $f $ is a local diffeomorphism the lifts $\gamma_t $ are unique if they exist, so if $t\in [0,1]$ is a limit point of $\mathcal T$ then the limiting lifts patch together to give $$\alpha : [0,t)\longrightarrow \BR, \ \  \alpha (0) = p, \ f \circ \alpha =\gamma |_{[0, t)}. $$ 
The path $\alpha$ is itself locally $\lambda $-bilipschitz, so its image is contained within the compact subset $\overline {B (p,  \lambda d(f(p),q))}\subset B(p,R)$.  From this and the fact that it is locally bilipschitz, $\alpha $ can be continuously extended to a map $[0,t]\longrightarrow \BR $.  This map must lift $\gamma $, so $t\in \mathcal T$.
Therefore, $\mathcal T = [0,1]$, implying in particular that $q =\gamma (1)\in f (B(p,R)) $. 
\end {proof}

One reason that these convergence notions are useful is that sequences of manifolds with `uniformly bounded geometry' have convergent subsequences. 

\begin {definition} \label {boundedgeometry} Suppose that $M$ is a complete $C ^ k $-Riemannian manifold.  A subset $A \subset M $ has \emph {$C ^ k $-bounded geometry} if for some fixed $r,\epsilon,L >0 $ there is a system of coordinate charts $$\phi_s : B_{\BR ^ d} (0, r)\longrightarrow U_s\subset M$$ with the following properties:
\begin {enumerate}
\item[1)] for every $p \in A$, the ball $B(p,\epsilon) \subset U_s$ for some $s$,
\item[2)] $\phi_s $ is locally $L$-bilipschitz,
\item[3)] all coordinates of the metric tensor $(\phi_s ^* g)_{ij}$ have $C ^ k$-norm at most $L$,
\item[4)] the transition maps $\phi_s ^ {-1}\circ\phi_t$ have $C ^ {k + 1} $-norm at most $L$.
\end {enumerate}
A sequence of subsets $A_i \subset M_i$ of complete Riemannian manifolds has \emph {uniformly $C ^ k $-bounded geometry} if  the constants $r,\epsilon, L$ can be chosen independently of $i $.
\end {definition}

The conditions above are a simplification of those given by Petersen \cite [pp. 289, 297]{Petersenriemannian} that are sufficient for the following theorem.

\begin {theorem}\label {firstcompactness}
Suppose that $(M_i, p_i) $ is a sequence of complete pointed $C ^ k $-Riemannian manifolds and that for some fixed $R >0$ the balls $B(p_i,R) \subset M_i$ have uniformly $C ^ k $-bounded geometry.    Then there is a complete pointed $C ^ {k - 1} $-Riemannian manifold $(M,p)$ and, for sufficiently large $i $, embeddings
 $$f_i:B(p,R-1)\longrightarrow M_i $$ with $f_i (p) = p_i $ such that $f_i ^*g_i\rightarrow g $  in the $C ^{k-1} $-topology on $B(p,R-1)$, where $g_i,g $ are the Riemannian metrics on $M_i,M $.
\end {theorem}

In other words, if the balls $B(p_i,R)\in M_i $ have uniformly $C ^ k $-bounded geometry, there is a subsequence on which we see $C^ {k - 1} $-convergence of the $(M_i,p_i)$, at least within a distance of $R-1$ from the base points.  
This is a version of Cheeger's compactness theorem \cite [Ch.\ 10, Thm.\ 3.3]{Petersenriemannian} for $R$-balls -- Cheeger's theorem usually requires uniform geometry bounds over the entire manifolds and then gives a fully convergent subsequence.  

\begin {proof}[Proof of Theorem \ref {firstcompactness}]
The theorem is not implied by the statement of Cheeger's compactness theorem given in \cite [Ch.\ 10, Thm.\ 3.3]{Petersenriemannian}, since the latter requires uniform $C ^ k $ bounds over the \emph {entire} manifolds $(M_i,p_i)$, but the proof is the same.  

The ideas to take a subsequential Gromov-Hausdorff limit $(X,x)$ of the balls $B_ {M_i}(p_i, R)$, which exists by the uniform geometry bounds.  An atlas of $C^{k-1}$-charts for $X$ is obtained as a limit of the charts for $B_ {M_i}(p_i, R)$ with uniformly $C^k$-bounded geometry, and then one shows that the convergence is $C^{k-1}$ in addition to Gromov-Hausdorff.  Then choose a complete pointed Riemannian manifold $(M,p)$ such that $B(p,R-1) \subset M $ and $B (x, R - 1)\subset X $ are isometric.  

The details are entirely the same as those of \cite [Ch.\ 10, Thm.\ 3.3]{Petersenriemannian}.
\end {proof}

Theorem \ref{firstcompactness} also gives a strong version of Cheeger's theorem in which the uniform geometry bounds may depend on the distance to the base point:

\begin {corollary}\label {2ndcompactness}
Let $(M_i, p_i) $ be a sequence of complete pointed Riemannian manifolds such that for every $R >0$ and $k\in \mathbb N$, the balls $B(p_i,R) \subset M_i$ have uniformly $C ^ k $-bounded geometry, where the bounds may depend on $R,k$ but not on $i$. Then $(M_i, p_i) $ has a smoothly convergent subsequence.
\end {corollary}
\begin {proof}
Applying a diagonal argument and passing to a subsequence, we may assume that for every $k \in \mathbb N$, there is a complete pointed $C ^ {k - 1} $-Riemannian manifold $(L_k,q_k)$ and, for sufficiently large $i $, embeddings
 $$f_{i,k}:B(p,k-1)\longrightarrow M_i $$ with $f_i (q_k) = p_i $ such that $f_i ^*g_i\rightarrow g_k $ in the $C ^{k-1} $-topology on $B(p,R-1)$, where $g_i,g_k $ are the Riemannian metrics on $M_i,M $.

By Arzela-Ascoli's theorem, for each $k $ there is a pointed isometry  $$B_{L_k} (q_k, k-1) \longrightarrow B_{L_{k+ 1}} (q_{k+ 1}, k-1) .$$  So, the direct limit of the system
$$B_{L_{2}} (q_{2}, 1) \longrightarrow B_{L_{3}} (q_{3}, 2) \longrightarrow B_{L_{4}} (q_{4}, 3) \longrightarrow \cdots $$
is a complete  pointed $C ^\infty $-Riemannian manifold to which $(M_i, p_i) $ smoothly converges.
\end {proof}

\subsection {Metrizability of $\mathcal M ^ d $ in the smooth topology}
\label {secmetrizability}
The goal here is to show that smooth convergence on $\mathcal M ^ d $ is induced by a Polish topology.   As mentioned in the introduction, this result was established independently and concurrently by Candel, \'Alvarez L\'opez and Barral Lij{\'o} \cite{Alvarezuniversal}.  Their paper became available earlier than ours, so the theorem is certainly theirs.  The two approaches use the same key idea, encoding partial derivatives of a metric in \emph{iterated Sasaki metrics}, but ours produces an explicit metric that will be used elsewhere in the paper, namely in \S \ref{smoothChabauty}.  Our approach is also a bit simpler, since we use metric neighborhoods of the zero section in  iterated tangent bundles instead of the iteratively defined neighborhoods in \cite{Alvarezuniversal}.

Suppose $M $ is a manifold with a Riemannian metric $g$. Sasaki \cite {Sasakidifferential} introduced a Riemannian metric $g ^ 1 $ on the tangent bundle $TM$.  If $(x_1,\ldots, x_d) $ is a coordinate system for some $U \subset M $, let $$(x,v)=(x_1,\ldots, x_d,v_1,\ldots, v_d) $$ be the \emph {induced coordinates} on $TU$, where $v_i =Dx_i$.    At a point $(x, v) \in TU$, the \emph {Sasaki metric} $g ^ 1 $ is given as follows, see \cite [p. 342]{Sasakidifferential}: for $1\leq i, j\leq d$,
\begin {align}
&( g ^ 1)_{i j}= g_{i j }+\sum_{1\leq\alpha,\beta,\gamma,\delta\leq d} g_{\gamma\delta}\Gamma ^\gamma_{\alpha i}\Gamma ^\delta_{\beta j} v_\alpha v_\beta \nonumber\\
&( g ^ 1)_{i \, (d + j)}  =\sum_{1\leq \alpha \leq d}\Gamma_{j\alpha i} v_\alpha \label {metricincluded}\\
&(g ^ 1)_{(d + i)\,(d + j)} = g_{i j}\nonumber .
\end {align}
Here, $\Gamma_{ka j} $ and $\Gamma ^\beta_{a j}$ are the Christoffel symbols of $g $ of the first and second kind, and all the metric data on the right sides of the equations are taken at $x\in U$.

The $k $-fold iterated tangent bundle of a manifold $M $ is the manifold
$$T^k M = T(\cdots T(T(M))\cdots ).$$
Any smooth map $f$ of manifolds induces a smooth map of iterated tangent bundles, the \emph {iterated total derivative} $D^k f$, and if $M $ has a Riemannian metric $g$ then $T^ k M $ inherits the \emph {k-fold iterated Sasaki metric} $g ^ k $.  Note that 
$$T^kM\longrightarrow T^{k-1}M \longrightarrow \cdots \longrightarrow TM \longrightarrow M$$ is a sequence of vector bundles, so for each $i$ we have  a zero section $$\zeta : T^{i-1}M \longrightarrow T^iM.$$  Define the \emph{zero section} of the fiber bundle $T^{k}M \longrightarrow M$ to be the iteration 
$$\zeta^k : M \longrightarrow T^kM.$$

We now show that the iterated Sasaki metric $g ^ k$  encodes all order $\leq k$ derivatives of $g $. We will state this in a strong way that will be useful in Corollary  \ref{bilipschitzderivatives}, but the point is that partial derivatives of the $g_{ij}$ can be written in terms of the $g^k_{\alpha \beta}$.

\begin {lemma}\label {relationship}
Suppose $g $ is a Riemannian metric on an open subset $U\subset \BR^d$ and let $g^1$ be the Sasaki metric on $T U$.  Fix coordinates $(x_1,\ldots, x_d) $ on $U $ and let $(x_1,\ldots, x_d,v_1,\ldots, v_d) $ be the induced coordinates on $TU$, where $v_i =Dx_i$. 

For every $x=(x_1,\ldots, x_d)\in U$ and indices $\alpha, i, j\in\{1,\ldots,  d\}$, we have
\begin {align}
g_{ij} (x) &= g^ 1_{(i +d) (j +d)}(x,v), \ \ \forall v\in T_x U\label{k11} \\
\partial_{\alpha}g_{ij} (x)  &= t\cdot   (g^1_{i(d+j)} + g^1_{j(d+i)})\underset{\ \ \ \alpha ^ {th} \text { place}}{ (x,0,\ldots,\frac 1t,\ldots,0) } , \ \ \ \forall t>0.\label{k12}
\end {align}
Now fix $k$, and iterate the above construction to give a system of coordinates on $T^kU$. For any compact $C\subset U$ and any open neighborhood $O\supset\zeta^k(C)$ in $ T^kU$, any partial derivative $\partial_{\alpha_1,\ldots,\alpha_l} g_{ij}(x)$ with $x\in C$, $0\leq l\leq k$ can be represented as a linear combination
\begin {equation}\partial_{\alpha_1,\ldots,\alpha_l} g_{ij}(x) = \sum_{n} t_n \cdot g^k_{\beta_n\gamma_n} (v_n)\label {derivatives}\end {equation}
where $v_ n\in O$ and $1\leq \beta_n,\gamma_n \leq kd$. Here, $v_n,\beta_n,\gamma_n, t_n$ are determined just by the indices $\alpha_1,\ldots,\alpha_l, i, j$ and the choices of $C,U$, and \emph{not} by the metric $g$.
\end {lemma}
\begin {proof} 
The $k = 1 $ case follows from \eqref {metricincluded}, using the identity $\Gamma_{j\alpha i}+\Gamma_{i \alpha j} =\partial_\alpha g_{ij} $.

For \eqref {derivatives}, first choose some $t>1$ such that for all $x\in C$, the point $$\underset{\ \ \ \alpha ^ {th} \text { place}}{ (x,0,\ldots,\frac 1t,\ldots,0) } \in O.$$ Then \eqref {derivatives} is proved inductively: every time a successive partial derivative of is taken, one can instead consider the appropriate linear combination of entries of the next Sasaki metric, given in \eqref{k11}. One takes \emph{at most} $k$ of these derivatives, and then uses \eqref{k12} to encode the resulting data in the $k^{th}$ iterated Sasaki metric, instead of a previous one.\end {proof}

It will follow that $C^k$-convergence of metrics $g_i$ is  equivalent to $C^0$-convergence of the Sasaki metrics $g_i^k$. Here, recall that $C^k$-convergence means uniform convergence on compact sets of all derivatives up to order $k$. In fact, slightly more is true. 

\begin {corollary}\label {bilipschitzderivatives}
Fix $k\in\BN$. Suppose that $g_i,g$ are Riemannian metrics on some manifold $M$, and let $g_i^k, g^k$ be the induced Sasaki metrics on $T^kM$. Then the following are equivalent:
\begin {enumerate}
\item 	$g_i \to g$ in the $C^k$-topology on $M$,
\item $g_i^k \to g^k$ in the $C^0$-topology on $T^kM$,
\item for any compact $C \subset M$, there is some  open set $O \subset T^kM$ containing $ \zeta^k(C)$ such that $g_i^k \to g^k$ uniformly on $K$.
\end {enumerate}
\end {corollary}
\begin{proof}
Covering $M$ with a locally finite set of coordinate charts, it suffices to prove the Corollary when $M$ is an open subset of $\BR^n$. So, we will feel free to use the coordinate expressions in \eqref{metricincluded} and Lemma \ref{relationship} below without comment.

	For (1) implies (2), fix a compact subset $K \subset T^k M$. The $C^k$-convergence $g_i\to g$ implies that all derivatives up to order $k$ of $g_i$ converge to those of $g$, uniformly on the projection of $K$ into $M$.  In particular, when computing each successive Sasaki metric, the Christoffel symbols in \eqref{metricincluded} converge. Since $K$ is compact, all coordinates $v_\alpha,v_\beta$ from \eqref{metricincluded}  that are relevant in the construction of $g^k$ are bounded. It follows that $g_i^k \to g^k$ uniformly on $K$.

Since $C^0$-convergence is by definition  uniform convergence on compact sets, the implication $(2) \implies (3)$ is obvious, taking $O$ to be any such open subset that has compact closure in $T^kM$.

It remains to prove $(3) \implies (1).$  Given a compact subset $C \subset M$, let $O$ be as in (3). For each $x\in M$, we can use \eqref{derivatives} to write the partial derivatives of each $g_i$ as linear combinations of entries of the $g_i^k$, evaluated at points of $O$, where then particular linear combination is  independent of $i$, and also works for $g$. These linear combinations converge uniformly, by our assumption, so $g_i\to g$ in the $C^k$-topology as desired. \end{proof}

The advantage of using convergence of Sasaki metrics instead of $C^k$ convergence is that it is easier to metrize, since two Sasaki metrics can then be compared by analyzing the minimal distortion of a bilipschitz map between them. Now, it is not the case that when two metrics on $M$ are $C^k$-close, the associated metrics on $T^kM$ are bilipschitz. (We thank a referee for first bringing this to our attention.)  For instance, looking at \eqref{metricincluded}, if two metrics are $C^1$-close then their Christoffel symbols are close, but the expression given for the Sasaki metric also includes the coordinates $v_\alpha,v_\beta$, which can be arbitrarily large\footnote{One can also see this issue when calculating $D^2 f$, where $f : \BR\longrightarrow\BR$ is a map between open subsets of $\BR$ that is $C^2$-close to the identity. On $T^2\BR$, the map $D^2f$ has the form
$$(x,v,w)\longrightarrow \left (x, f'(x) v, \begin{pmatrix}
	f'(x) & 0 \\ f''(x) v & f'(x) \end{pmatrix} w \right ), $$
where here $v\in T\BR_x\cong \BR$ and $w \in T(T\BR)_{(x,v)} \cong \BR^2$. The term $f''(x)v$ can be arbitrarily large, so $D^2f$ is only bilipschitz if $f''(x)\equiv 0$.}. To fix this, we only look at the bilipschitz  distortion on subsets of $T^kM$ in which the coordinates $v_\alpha,v_\beta$ are bounded. Namely,  for each $r>0$ and $k\in \BN$, let $$Z^k_r(M) \subset T^k M$$ be the radius-$r$ neighborhood of the zero section $\zeta^k(M)$, with respect to the $k$-fold iterated Sasaki metric. Then we have:

\begin {corollary}\label {iterated}
Suppose that $(M_n, p_n)\in \mathcal M ^ d $, where $n = 1, 2,\ldots $.  Then $(M_n, p_n) $ converges $C ^ k $ to $ ( M, p) \in \mathcal M^d$ if and only if for every $R >0 $, we have that for sufficiently large $n $, there are smooth embeddings $$f_n : B_M (p, R) \longrightarrow M_n $$
with $f_n (p) = p_n $ such that for some $r>0$, the iterated total derivatives $$D^kf_n : T^k B_M (p, R) \longrightarrow T^kM_n$$ are locally $\lambda_n$-bilipschitz embeddings on $Z_r:=Z^k_r(B_M(p,R))$, with $\lambda_n\rightarrow 1$.
\end {corollary}

\begin {proof}

Let $g$ and $g_n $ be the metrics on $M $ and $M_n $. 
The point is to show that the condition that $D^k f_n$ is locally $\lambda_n $-bilipschitz on $Z_r$, with $\lambda_n\rightarrow 1$, is equivalent to the condition that $f ^*_n g_n \rightarrow g$ in the $C ^ k $-topology.
First, note that $$(D^kf_n) ^*g_n ^ k=\left (f_n ^*g_n \right)^ k,$$ where $g_n ^ k$ is the Sasaki metric on $T ^ k M_n $ and $\left (f_n ^*g_n \right)^ k$ is the Sasaki metric on $T^kB_M(p,R)$ corresponding to $f_n ^*g_n$.  So, $D^kf_n $ is locally $\lambda _n$-bilipschitz on $Z_r$ exactly when 
\begin{equation}
	id : (B_M (p, R), g ^ k)\longrightarrow(B_M (p, R), \left (f_n ^*g_n \right)^ k )\label{idmap}
\end{equation}
is locally $\lambda_n $-bilipschitz on $Z_r$.  

The `only if' direction then follows from our previous work, since if  the map \eqref{idmap} is locally $\lambda_n $-bilipschitz on $Z^k_r(M)$ with $\lambda_n\to 1$, then $$(f_n ^*g_n)^ k\to g^k$$ uniformly on $Z_r$, 
implying that $f_n^*g_n$ converges $C^k$ to $g$ by Corollary \ref {bilipschitzderivatives}.

The `if' direction is also easy. Suppose that $f_n^*g_n\to g$ in the $C^k $-topology. Then Corollary \ref {bilipschitzderivatives} implies that $(f_n^*g_n)^k\to g^k$ in the $C^0$ topology. The ball $B_M(p,R)$ has compact closure in $M$, so the set $Z_r$ has compact closure in $T^kM$. Hence, $(f_n^*g_n)^k\to g^k$ uniformly on $Z_r$, implying that the identity map in \eqref{idmap} is locally $\lambda_n$-bilipschitz with $\lambda_n\to 1$.\end {proof}

Finally, we record for later use the following fact.

\begin{fact}\label{alsobilipschitz}
Fix $\lambda>1$, $k\in \BN$ and $r>0$. Suppose that $f : M \longrightarrow N$ is a smooth map of Riemannian manifolds and that $D^k f$ is locally $\lambda$-bilipschitz when restricted to some neighborhood of the zero section $\zeta^k(M) \subset T^kM$, for instance the neighborhoods $Z^k_r(M)$ above. Then  the map $f$ is itself locally $\lambda$-bilipschitz.
\end{fact}

\begin {proof}
Let $x\in M$, set $v_0:=x$ and set $v^i=\zeta^i(x)$, for $i=1,\ldots,k$. For each $i$, we have isometric embeddings
$$\iota: T^iM_{v_{i-1}} \overset{\cong}{\longrightarrow} T((T^iM)_{v_{i-1}})_{v_{i}} \hookrightarrow T^{i+1}M_{v_{i}},$$
which satisfy the commutative diagram
$$\begin{tikzcd}
    T^iM_{v_{i-1}} \ar{r}{\iota} \ar{d}{D^if} & T^{i+1}M_{v_{i}} \ar{d}{D^{i+1}f} \\
     T^iN_{D^{i-1}f(v_{i-1})} \ar{r}{\iota} & T^{i+1}N_{D^{i}f(v_{i})} 
\end{tikzcd},$$
defining $\iota$ similarly on $N$. The commutativity  is just the statement that a linear map is its own derivative, and the fact that $\iota$ is an isometric embedding is immediate from the definition of the Sasaki metric. Composing, we have a diagram $$\begin{tikzcd}
    TM_{x} \ar{r}{\iota^k} \ar{d}{Df} & T^{k+1}M_{v_{k}} \ar{d}{D^{k+1}f} \\
     TN_{f(x)} \ar{r}{\iota^k} & T^{k+1}N_{D^{k}f(v_{k})} 
\end{tikzcd},$$
so if the linear map  $D^{k+1}f$ is $\lambda$-bilipschitz, so must be $Df$.
\end {proof} 

Corollary \ref {iterated} suggests a description of a basis of neighborhoods around a point $(M, p)\in\mathcal M ^ d $.  We define the \emph {$k^{th} $-order $(R,r,\lambda)  $-neighborhood} of $(M, p) $, written $$\mathcal N ^ k_{R,r,\lambda} (M, p), $$ 
to be the set of all $(N,q) $ such that there is a smooth embedding $$f: B_M (p, R)\longrightarrow N $$ with $f (p) = q $ such that $D  ^ k f: T ^ k U\longrightarrow T ^ k N  $ is locally $\lambda $-bilipschitz with respect to the iterated Sasaki metrics on $Z^k_r(B_M(p,R))$.  Note that $\mathcal N ^ k_{R,r,\lambda} (M, p)$ is a \emph {closed} neighborhood of $(M, p) $.  By Arzela-Ascoli, its interior is the open neighborhood $$\mathring{\mathcal N} ^ k_{R,r,\lambda} (M, p)$$ that is defined similarly, except that we require $D ^ k f $ to be locally $\lambda' $-bilipschitz for some $\lambda' <\lambda $. 

\begin {theorem}\label {manifoldmetrizable}
$\mathcal M ^ d $ has the structure of a Polish space (a complete, separable metric space), in which convergence is smooth convergence.
\end {theorem}

\begin {proof}
For each $R >0 $ and $k\in\mathrm {N} $,  define a function $d_{R, k} : \mathcal M ^ d\times\mathcal M ^ d\longrightarrow\BR $ by
\begin {align*}
 d_{R, k}\left((M, p ),(N, q )\right) = \inf\{\log \lambda \ | \ (N, q ) \in\mathcal N ^ k_{R/\lambda,1/\lambda,\lambda} (M, p)\}.
\end {align*}
Each $d_{R, k}$ satisfies an (asymmetric) triangle inequality.  For suppose we have $(M,p),(N,q),(Z,z)\in\mathcal M ^ d $ and basepoint respecting embeddings $$f: B_M (p , R/\lambda)\longrightarrow N, \ \ g: B_N(q , R/\mu)\longrightarrow Z$$ such that $D ^k f $ and $D ^ k g $ are locally $\lambda $-bilipschitz and locally $\mu $-bilipschitz embeddings, respectively, when restricted to the sets $$Z^k_{1/\lambda}(B_M(p,R/\lambda)) \subset T^kM, \ \ Z^k_{1/\mu}(B_N(q,R/\mu)) \subset T^kN.$$  By Fact \ref {alsobilipschitz}, the map $f $ is also a locally $\lambda$-bilipschitz embedding, so
$$f\left(B_M \left(p , \frac R{\lambda \mu}\right)\right) \subset B_N \left(q, \frac R \mu\right),$$
and since $D^kf$ is $\lambda$-bilipschitz, 
$$D^kf\left(Z^k_{1/(\lambda\mu)}(B_M(p,R/(\lambda\mu)))\right ) \subset Z^k_{1/\mu}(B_N(q,R/\mu)) .$$
Therefore, the composition $g\circ f : B_M (p , \frac R{\lambda \mu}) \longrightarrow N$ is defined and the map $D ^ k (g\circ f)  $ is locally $(\lambda\mu )$-bilipschitz on $Z^k_{1/(\lambda\mu)}(B_M(p,R/(\lambda\mu)).$ So,
\begin {align*}d_{R, k}\left((M, p ),(Z,z )\right) &\leq \inf_{\lambda,\mu }\, \log \lambda\mu \\\ &= 
\inf_{\lambda }\, \log \lambda + \inf_{\mu }\, \log \mu \\ &= d_{R, k}\left((M, p ),(N, q )\right)+ d_{R, k}\left((N, q ),(Z, z )\right). 
\end {align*}

The subsets of $\mathcal M ^ d $ defined for each $(M, p) \in\mathcal M ^ d$, $R>0 $, $k\in\mathrm{N} $, $\epsilon >0 $ by $$d_{R, k}\left((M, p ),\, \cdot \,\right)< \epsilon $$ form a basis for a \emph {smooth  topology} on $\mathcal M ^ d $ that induces smooth convergence, by Lemma \ref {iterated}.  Although the $d_{R, k}$ are not symmetric, the reversed inequalities $$d_{R, k}\left(\, \cdot \,, (M, p )\right)< \epsilon $$ define a basis for the same topology, as Lemma \ref {invert} allows the relevant locally bilipschitz maps to be inverted at the expense of decreasing $R $.  So, the smooth topology is generated by the family of pseudo-metrics 
$$\hat d_{R, k} : \mathcal M ^ d\times\mathcal M ^ d\longrightarrow\BR, \ \ \hat d_{R, k}(x, y) = d_{R, k} (x, y)+ d_{R, k} (y, x).$$

 As the topology on $\mathcal M ^ d  $ induced by a particular pseudo-metric $\hat d_{R, k} $ becomes finer if $R, k $ are increased, it suffices to consider only $\hat d_{k, k} $ for $k\in\mathbb N $.  Therefore, the following is a metric on $\mathcal M ^ d $ that induces the smooth topology:
$$D:\mathcal M ^ d\times\mathcal M ^ d\longrightarrow\BR, \ \ D\big ( x,y\big) = \sum_{k = 1} ^\infty 2^ {- k} \min \{\hat d_{k, k}(x,y), 1\}. $$

We now show that $\mathcal M ^ d $ is separable.  An element $(M, p)\in\mathcal M ^ d $ is a limit of closed Riemannian manifolds: for instance, we can exhaust $M$ by a sequence of compact submanifolds with boundary, double each of these and extend the Riemannian metric on one side arbitrarily to the other.  So, it suffices to construct a countable subset of $\mathcal M ^ d $ that accumulates onto every closed manifold in $\mathcal M ^ d $.

There are only countably many diffeomorphism types of closed manifolds: this is a consequence of Cheeger's finiteness theorem \cite {Cheegerfiniteness}, for instance.  So, it suffices to show that the space $\mathcal M (M) $ of isometry classes of pointed closed Riemannian manifolds \emph {in the diffeomorphism class of some fixed $M$} is separable in the smooth topology.  This space $\mathcal M (M) $ is a continuous image of the product of $M$ with the space of Riemannian metrics on $M$, with the smooth topology on tensors.  The manifold $M$ is separable, and so is the space of Riemannian metrics on $M$, by the Weierstrass approximation theorem. So, their product is separable, implying $\mathcal M (M) $ is too, finishing the proof.

Finally, we want to show that  $(\mathcal M ^ d, D)$ is complete, so let $(M_i, p_i) $ be a Cauchy sequence.  We claim that for every $R>0 $ and $k\in \mathbb N $, the balls $$B (p_i, R)\subset M_i $$ have uniformly $C ^ k $-bounded geometry, in the sense of Definition \ref {boundedgeometry}.   Corollary \ref {2ndcompactness} will then imply that $(M_i, p_i) $ has a smoothly convergent subsequence, which will finish the proof.

It suffices to show that there are arbitrarily large $R,k$ for which the balls $$B (p_i, R)\subset M_i $$ have uniformly $C ^ k $-bounded geometry.  So, fix some $k\in \mathbb N$.  Since $(M_i, p_i) $ is $D $-Cauchy and $D= \sum_{j = 1} ^\infty 2^ {- j} \min \{\hat d_{j, j}, 1\}, $
 the $d_{k, k} $-diameter of the tail of $(M_i, p_i) $ can be made arbitrarily small.  In other words, there is some $(M,p)\in \mathcal M ^ d $ such that for sufficiently large $i $,
$$d_{k, k} \big ((M, p), (M_i, p_i)\big) < \log 2.$$ 
This means that for each large $i $ there is a pointed, smooth embedding
$$f_i: B_M (p, k)\longrightarrow M_i $$ such that $D^k f_i $ is locally $2 $-bilipschitz on $Z:=Z^k_{\frac 12}(B_M(p,k/2))$.  By Lemma \ref {invert}, $$f (B_M (p, k/2)) \supset B_{M_i} (p_i,k/4) .$$  

By precompactness, the ball $B_M (p,k/2-1) \subset M$ has $C ^ k $-bounded geometry, in the sense of Definition \ref {boundedgeometry}, for some constants $r,\epsilon, L $. We will use the maps $f_i$ to translate this to \emph {uniform} $C^k$-geometry bounds for the balls $$B_{M_i}(p_i,R) \subset M_i, \text { where } R:=(k - 1)/4 .$$   Let $\phi_s: B \longrightarrow U_s \subset M $ be finitely many coordinate charts as in  Definition \ref {boundedgeometry}, where $B$ is a ball around the origin in $\BR^d$, and by shrinking $B $ assume that $U_s\subset B _M (p, k)$ for all $s $.  Define
$$\phi_{i, s} : B \longrightarrow U_{i, s} \subset M_i, \ \ \phi_{i,s} = f_i\circ\phi_s  \ \ \forall i,s.$$  It is now straightforward to verify that 1) -- 4) of Definition \ref {boundedgeometry} are satisfied for the subsets $B_{M_i} (p_i, R)\subset M_i $ by the charts $\phi_{i,s}$, with modified constants.  As the maps $f_i: B_M (p, R) \longrightarrow M_i $ are locally $2 $-bilipschitz, the $\phi_{i,s}$ are locally $2 L$-bilipschitz, so by Lemma \ref {invert} the $\epsilon/2 $-ball around every $q\in B_M(p,R') $ is contained in some $U_{i,s}$, as $R' = (R -1 )/2 $. The transition maps of 4) are unchanged by the composition, so it remains to prove condition 3). The argument for here is like an effective version of $(3)\implies (1)$ in Corollary \ref{bilipschitzderivatives}. Fix some $s$, and let $O = D^k \phi_s^{-1}(Z) \subset T^kB$, where  $Z \subset T^kM$ is the precompact subset defined above on which $f_i$ is locally $2$-bilipschitz. By Lemma \ref{relationship}, each partial derivative of a Riemannian metric $g$ on $B$ can be expressed as some fixed linear combination \begin {equation}\partial_{\alpha_1,\ldots,\alpha_l} g_{ij}(x) = \sum_{n} t_n \cdot g^k_{\beta_n\gamma_n} (v_n)\label {derivatives2}\end {equation}
of coefficients of the metric $g^k$ on $T^k B$, where $v_n  \in O$ and the linear combination depends only on the partial derivative taken and the neighborhood $O$, not the metric $g$. So, let $g_M$ and $g_i$ be the metrics on $M$ and $M_i$, respectively. By assumption, the coefficients of the metric $\phi_s^* g_M$ have bounded $C^k$ norm, so since $O \subset T^kB$ is precompact, the construction of the Sasaki metric implies that all the coefficients of $(\phi_s^* g_M)^k=(D^k\phi_s^*) g_M^k$  are bounded on $O$. But since $D^kf_i$ is $2$-bilipschitz on $Z$, the metrics $g_M^k$ and $(D^k f_i^*)g_i^k$ differ by a factor of at most $2$ on $Z$, and hence the coefficients of the metric $$\left ((f_i \circ \phi_s)^* g_i \right )^k = (D^k\phi_s ^i)^*((D^k f_i^*)g_i^k)$$ are also bounded on $O$. Equation \eqref{derivatives2} then implies that all partial derivatives of the metric $(f_i \circ \phi_s)^* g_i $ are bounded, which proves condition (3).

This shows that for all $k\in \mathbb  N $, the balls $B(p_i,(k-1)/2)\subset M_i $ have uniformly $C ^ k  $-bounded geometry.  (We showed this explicitly for large $i $, but the initial finitely many terms only contribute a bounded increase to the constants.)  So, Corollary \ref {2ndcompactness} implies that $(M_i, p_i) $ has a smoothly convergent subsequence.  
\end {proof}

\label {smoothtopology}

\subsection {Vectored and framed manifolds} \label {vectored} It is often convenient to supplement the basepoint of a pointed manifold $(M,p)$ with additional local data: for instance, a unit vector or an orthonormal basis for $T_pM$.  An orthonormal basis for $T_pM$ is called a \emph {frame} for $M$ at $p$, and we let $\CF M \longrightarrow M$ be the bundle of all frames for $M$.

A \emph {vectored Riemannian manifold} is a pair $(M, v) $, where $v\in TM $ is a unit vector, and a \emph {framed Riemannian manifold} is a pair $(M,f)$, where $f\in \CF M$ is some (orthonormal) frame.  We define
\begin {align*}T^1\mathcal M^ d&=\{\substack{\text {vectored, connected, complete}\\ \text { Riemannian d-manifolds } (M, v)} \}/\sim\\
\CF \mathcal M^ d&=\{\substack{\text {framed, connected, complete}\\ \text { Riemannian d-manifolds } (M, f)} \}/\sim
\end {align*}
where in both cases we consider vectored (framed) manifolds up to vectored (framed) isometry.
Smooth convergence of vectored Riemannian manifolds is defined as follows: we say that $(M_i, v_i)\rightarrow (M, v) $ if for every $R >0 $ there is an open set $U \supset B (p,R)  $ and, for sufficiently large $i $, embeddings
\begin {equation}f_i:U\longrightarrow M_i \label{framedalmost}\end {equation} with $Df_i (v) = v_i $ such that and $f_i ^*g_i\rightarrow g $ on $U$ in the $C ^\infty $-topology, where $g_i,g $ are the Riemannian metrics on $M_i,M $; an analogous definition gives a notion of smooth convergence on $\CF \mathcal M^ d$.   We then have:

\begin {theorem}
$T^1\mathcal M^ d$ and $\CF \mathcal M ^ d $ both admit complete, separable metrics that topologize smooth convergence, and such that the natural maps$$\CF \mathcal M ^ d \longrightarrow T^1\mathcal M^ d \longrightarrow \mathcal M ^ d$$
defined by taking a frame to its first element, and a vector to its basepoint, are quotient maps.
\end {theorem}

The proof is identical to the work done earlier in this section.  In particular, one can still reinterpret smooth convergence through locally bilipschitz maps of iterated tangent bundles, as long as these maps respect the obvious lifts of the base vectors or frames of the original manifolds.  Also, all compactness arguments still apply since we have chosen \emph {unit} vectors and \emph {orthonormal} frames.

\subsection {The Chabauty topology}\label {Chabautysection}
Suppose $M $ is a proper metric space and let $\mathcal C (M) $ be the space of closed subsets of $M $.  The \emph {Chabauty topology} on $\mathcal C(M)$ is that generated by subsets of the form
\begin {equation}\label {neighborhoods}\{C \in \mathcal C (M) \ | \ C \cap K=\emptyset\}, \ \   \  \ \{C \in \mathcal C (M) \ | \ C \cap U \neq\emptyset\},
\end {equation}
where $K\subset M$ is compact and $U\subset M $ is open.  It is also called the \emph {Fell topology} by analysts.  Convergence can be characterized as follows:

\begin {proposition}[Prop E.12, \cite{Benedettilectures}]\label {Chabautyconvergence}
A sequence $(C_i)$ in $\mathcal C (M) $ converges to $C \in \mathcal C (M)$ in the Chabauty topology if and only if
\begin {enumerate}
\item if $x_{i_j} \in C_{i_j}$ and $x_{i_j} \to x \in M$, where $i_j \to \infty$, then $x\in C$.
\item if $x\in C$, then there exist $x_i \in C_i$ such that $x_i \to x$.
\end {enumerate}
\end {proposition}

The Chabauty topology is compact, separable and metrizable \cite[Lemma E.1.1]{Benedettilectures}.  When $M$ is compact, it is induced by the \emph {Hausdorff metric} on $C(M)$, where the distance between closed subsets $C_1,C_2 \subset M$ as defined as
$$d_{\text {Haus}}(C_1,C_2)=\inf \{ \epsilon \ | \ C_1 \subset \mathcal N_\epsilon (C_2) \text { and }  C_2 \subset \mathcal N_\epsilon (C_1) \}. $$
For noncompact $M$, the Chabauty topology is almost, but not quite, induced by taking the Hausdorff topology on all compact subsets of $M $.  Namely, fix a base point $p\in M$.  If $A\subset M$ is closed and $R>0$, set
$$
A_R=A\cap \overline {B(p,R)},$$
and then define a pseudo-metric $d_R$ on $\mathcal C (M) $ by setting
$$ d_R(A,B)=\min  \Big \{ 1, d_{\text {Haus}}(A_R, B_R) \Big\},$$
where $d_{\text {Haus}}$ is the Hausdorff metric of the compact subset $\overline {B(p,R)} \subset M $. 

The family of pseudo-metrics $\{d_R \ | \ R>0\}$ does not determine the Chabauty topology, since if $x_i\rightarrow x$ is a convergent sequence of points with $d(x,p)=R$ and $d(x_i,p)>R$ for all $i $, then $\{x_i\} \rightarrow \{x\}$ in the Chabauty topology, but $d_R(\{x_i\},\{x\})=1$ for all $i$. 

We now describe how to taper down $d_R$ near  the boundary of $B(p,R)$  so that even when points converge into $\partial B(p,R)$  from outside, $d_R$  does not jump in the limit.
The idea is  to view the Hausdorff distance on  closed  subsets of $ \overline {B(p,R)}$  as a special case of a  distance $d_{usc}$ on upper semicontinuous  (u.s.c.)  functions $$\phi,g:  \overline {B(p,R)} \longrightarrow [0,1] .$$
Closed subsets $A,B$  have u.s.c.\ characteristic functions $1_A, 1_B$.  The advantage of functions is that $1_A (x)$ and $1_B (x)$  can be scaled to converge to zero as $x \to \partial B(p,R)$,  so that  near $\partial B(p,R)$, the contribution to distance is negligible.

\vspace{2mm}

 To define a metric on u.s.c.\ functions, we use an idea of Beer \cite{Beerupper}.   Given a  compact metric space $K $ and a function $\phi:K\longrightarrow [0,1]$,  let $$\mathcal H(\phi) = \{ (x,s) \ | \ s \leq \phi(x) \} \subset K\times [0,1]$$  be the \emph {hypograph} of $\phi$. The  distance between functions $\phi,\psi: K\longrightarrow [0,1]$  is
$$d_{\text {usc}}(\phi,\psi) := d_{\text {Haus}}\left (\mathcal H(\phi),\mathcal H(\psi)\right) ,$$
 where $d_{\text {Haus}}$ is the Hausdorff  metric on $K \times [0,1] $,  considered with the product metric $d((x,s),(y,t))=d(x,y)+d(s,t)$.  Note that if  $A,B \subset K$  are closed, $$d_{\text {usc}}(1_A,1_B) = \min \big \{ 1,d_{\text {Haus}}\left (A,B\right) \big \}.$$

Fix a function $\phi : K \longrightarrow [0,1]$,  and define a new metric
$d_{\text {Haus}}^\phi$ on $\mathcal C(K)$ via
$$d_{\text {Haus}}^\phi(A,B) = d_{\text {usc}}\left (\mathcal H\left( \phi  \cdot 1_{A} \right), \mathcal H\left(\phi \cdot 1_{B}\right) \right).$$
 So in words, we are just taking the Hausdorff distance between $A,B$, but are  scaling down the importance of different parts of the sets as dictated by $\phi$.

\begin {lemma}
 Suppose $\phi,\psi : K	\longrightarrow [0,1]$ are u.s.c.\ functions,  that $\phi$  is $\lambda $-lipschitz, and  that $\phi(x) \leq C \psi(x)$.  Then 
$d_{\text {Haus}}^\phi \leq   \max\{ C, \lambda +1\} d_{\text {Haus}}^\psi.$ \label {lemchab}
\end {lemma}

 Note that the lipschitz condition is necessary for any sort of inequality. For instance, if $\phi$  approximates $1_{\{x\}}$ and $\psi=1$  is constant, we  can make $$1 \approx d_{\text {Haus}}^\phi(\{x\},\{y\}) >> d_{\text {Haus}}^\psi(\{x\},\{y\})\approx 0$$  by  taking $y\approx x$ and the approximation  $\phi \approx 1_{\{x\}}$ sufficiently close.   

\begin {proof}
 Suppose that $\mathcal H (\psi \cdot 1_A)$ and $\mathcal H (\psi \cdot 1_B)$ are $\epsilon $-close  in the Hausdorff metric. We want to show $\mathcal H (\phi \cdot 1_A)$ and $\mathcal H (\phi \cdot 1_B)$ are $(\max\{ C , \lambda +1\}  \cdot \epsilon)$-close, i.e.\  that  the two sets are each contained in $(\max\{ C , \lambda +1\}  \cdot \epsilon )$-neighborhoods of each other.

Let $x\in A$. It suffices to show that  there is some $(y,t) \in \mathcal H (\phi \cdot 1_B)$ with
\begin {equation}
	d((x,\phi(x)),(y,t)) \leq \epsilon \,  \max\{ C , \lambda +1\}. \label {wts}
\end {equation}

For the  same estimate will also work with $\phi(x)$  replaced  by any $s < \phi(x)$, so $\mathcal H (\phi \cdot 1_A)$  is contained in a $ {\lambda C\epsilon}  $-neighborhood of  $\mathcal H (\phi \cdot 1_B)$.  The proof of the other  inclusion is the same, switching the roles of $A,B$.

 We know that there is some $(y,t) \in \mathcal H (\psi \cdot 1_B)$ with $d((x,\psi(x)),(y,t)) \leq   \epsilon .$ If $t=0$, then $\psi(x) \leq \epsilon $, implying $\phi(x) \leq C\epsilon$, and $d((x,\phi(x)),(x,0)) \leq    { C\epsilon}, $  which proves \eqref{wts}.  So, assume $t\neq 0$. In this case, $y \in B$ and $d(x,y) \leq \epsilon$. Since $\phi$  is $\lambda $-lipschitz, $|\phi(x)-\phi(y)| \leq \lambda\epsilon$. Hence, 
$$d((x,\phi(x)),(y,\phi(y))) \leq  \epsilon(1 + \lambda). \qedhere$$
\end {proof}

 We now return to the problem of constructing a metric for the Chabauty topology  on $\mathcal C(M)$, for a proper metric space $M $.   Fix a point $p\in M$ and for each $R>0$, define a pseudometric $\hat d_R $ on $\mathcal C (M)$ by
\begin {equation}\hat d_R = d_{\text {Haus}}^{\phi_R}, \text{ where } \phi_R(x)=\begin {cases}
\frac {R-d(p,x )}R & d(p,x)\leq R \\
0 & d(p,x) \geq R. \end {cases}  \label {drhat} \end {equation}
Note that $\hat d_R$ induces the  Hausdorff  topology  on  the set of compact subsets of  the open ball $B(p,R)$, but  cannot tell apart subsets of $M \smallsetminus B(p,R)$. Also, earlier we only defined  our modified Hausdorff metrics $ d_{\text {Haus}}^{\phi}$ for $K$  compact, while $M$ is not compact.  However, since $\phi_R(x)=0$  when $d(p,x)\geq R$,  one can consider  the above construction as taking place within $K=\overline {B(p,R)}$.

\begin {proposition}\label {2ndChabauty}
The Chabauty topology on $\mathcal C(M)$  is induced by the  family of pseudo-metrics $\hat d_R$, for $R \in (0,\infty)$. 
\end {proposition}

By Lemma \ref{lemchab}, we have $\hat d_R \leq 2 \hat d_{R'}$  whenever $1 \leq R \leq R'$, since $\phi_R$ is $1$-lipschitz and $\phi_{R'} \geq \phi_R$.  This implies the Chabauty topology  is induced by any family  $\hat d_{R_i} $  with $R_i \to \infty$,  although this is also clear from the proof below.

\begin {proof}Assume that  $A_i \to A$ in the Chabauty topology.  Fixing $R $, we want to show that the hypographs  $ \mathcal H (\phi_R \cdot 1_{A_i})$  Hausdorff converge to $ \mathcal H (\phi_R \cdot 1_{A})$.   

First, suppose that  $(x,t) \in \mathcal H (\phi_R \cdot 1_{A})$. If $t=0$, we have $(x,t) \in \mathcal H (\phi_R \cdot 1_{A_i})$  for all $i $.  If $t\neq 0$, then $p \in A \cap B(p,R)$, so by Chabauty  convergence, $x=\lim_i x_i$  for some sequence $x_i \in A_i$. So, $(x,t)$  is a limit of points $(x_i,t_i) \in \mathcal H (\phi_R \cdot 1_{A_i})$. 

Next, suppose $(x,t)$ is  the limit of  some  sequence $(x_{i_j},t_{i_j}) \in \mathcal H (\phi_R \cdot 1_{A_{i_j}})$.   Again, if $t=0$  then $(x,t) \in  \mathcal H (\phi_R \cdot 1_{A})$ automatically, so we are done.  Otherwise,  we can assume after passing to a  further subsequence that  $t_{i_j}\neq 0$  for all $i_j$.  In this case,  each $x_{i_j}\in A_i$, so $x=\lim x_{i_j} \in A$.  Hence $(x,t) \in \mathcal H (\phi_R \cdot 1_{A})$.

Finally, we must show that  if $ (A_i)$ does not converge to $A$  in the Chabauty topology,  then there is some $R $ with $\hat d_R (A_i,A) \not \to  0$.  There are two  cases.  Assume first that $x\in A$  is not the limit of any  sequence $x_i\in A_i$.  Taking $R>d(p,x)$, we see that $(x,\phi_R \cdot 1_{A}(x)) \in \mathcal H (\phi_R \cdot 1_{A})$ is not the limit of any sequence of points in the hypographs  $ \mathcal H (\phi_R \cdot 1_{A_i})$, so  we have $\hat d_R (A_i,A) \not \to  0$.  Similarly,  if there is some sequence $x_{i_j} \in A_{i_j}$  that converges to a point outside of $A $, the points $(x_{i_j},\phi_R \cdot 1_{A_{i_j}}(x_{i_j}) )$  will converge to a point outside of $\mathcal H (\phi_R \cdot 1_{A})$.
\end {proof}

Finally, for use in the next section,  we prove:

\begin {corollary}\label {distortioncor}
 Suppose that $f : B_{M_1}(p_1,R_1) \longrightarrow M_2$  is a locally $\lambda $-bilipschitz  embedding  with $R_1 \geq 1$, and $f(p_1)=p_2$.   Then for any $R_2 \geq \lambda R_1,$  we have $$\hat d_{R_1} (f^{-1}(C),f^{-1}(D)) \leq \lambda \, \hat d_{R_2} (C,D), \ \ \forall C,D \in \mathcal C (M_2).$$
\end {corollary}

 Note that $R_2 \geq \lambda R_ 1 $ implies  $f\big(B_{M_1}(p_1,R_1)\big) \subset B_{M_2}(p_1,R_2)$.

\begin {proof}   The two  sides of the inequality are $d_{\text {Haus}}^{\phi_i}(C,D)$, $i=1,2$,  where
	\begin {align*}
	\phi_1  : M_2 \longrightarrow [0,1], \ \ &\phi_1(x_2)= \begin {cases}
\frac { R_1-d(p_1,x_1)}{ R_1} & 	x_2=f(x_1), \ x_1 \in B_{M_1}(p_1,R_1), \\
0 & x_2 \not \in f\big(B_{M_1}(p_1,R_1)\big),
 \end {cases} \\
  \phi_2 : M_2 \longrightarrow [0,1], \ \ & \phi_2(x_2)= \begin {cases}
\frac {R_2-d(p_2,x_2)}{R_2} & x_2	\in B_{M_2}(p_1,R_2) \\
0 & x_2 \not \in B_{M_2}(p_1,R_2).
 \end {cases}
\end {align*}
Since $f$  is $\lambda $-lipschitz, we have $d(p_2,x_2) \leq \lambda d(p_1,x_1)$  if $f (x_1) = x_2 $.  Conversely, suppose $\gamma $ is a path in $M_2$ joining $x_2$ to $p_2$.  The preimage $f^{-1}(\gamma)$  is either a path from $x_1$ to $p_1$, or is a union of paths,  the last of which is a path from  $\partial B_{M_1}(p_1,R_1)$ to $p_1$.  In either case,  the length of $f^{-1}(\gamma)$  is at least $d(p,x_1)$, so the length of $\gamma $ is at least  $\frac 1\lambda d(p,x_1)$, as $f$ is locally  $\lambda $-bilipschitz. This shows  
\begin {equation}\frac 1 \lambda d(p_1,x_1) \leq d(p_2,x_2)\leq \lambda d(p_1,x_1).\label {globalbilipschitz}	
\end {equation}

 Note that it may not be true that $f$ is globally  $\lambda $-bilipschitz, e.g.\ if $f$ is the inclusion of an  interval of length $.999$  into a circle of length $1$, but \eqref{globalbilipschitz}  holds  in this case because one of the two points  is the center of the interval.

 It follows from \eqref{globalbilipschitz}  that $\phi_1$ is $\lambda $-lipschitz.  Moreover, since $R_2 \geq \lambda R_1$,  \begin {align*}\phi_1 (x_2)  \leq  \frac { R_1-d(p_1,x_1)}{ R_1} 
  \leq  \frac  {\lambda R_1-d(p_2,x_2)}	{\lambda R_1}
  \leq \phi_2 (x_2)
\end {align*}
 for all $x_2 \in B_{M_2}(p_1,R_2)$.  So $\phi_1/\phi_2 \geq 1$. Therefore, the  hypotheses of Lemma \ref{lemchab}  are satisfied  with the constant $\lambda $,  which proves the corollary.\end {proof}
\label {Chabautytopology}

\subsection {The smooth-Chabauty topology}
\label {smoothChabauty}
In this section we combine the smooth topology of \S \ref {smoothtopology} with the Chabauty topology of \S \ref {Chabautytopology}.  Consider the set
$$\mathcal {CM}^d = \{(M,p,C) \ | \ \substack{M \text{ a complete, connected Riemannian } \\ d\text {-manifold,}  \ p\in M, \text { and } C \subset  M \text { a closed subset}} \}/\sim,$$
where $(M_1,p_1,C_1)\sim (M_2,p_2,C_2)$ if there is an isometry $M_1\longrightarrow M_2$ with $p_1\mapsto p_2$ and $C_1\mapsto C_2$.
We say that $(M_i, p_i, C_i)\rightarrow (M,p,C)$ in the \emph {smooth-Chabauty topology} if for large $i$ there are embeddings
\begin {equation}f_i:B_M(p, R_i) \longrightarrow M_i \label{CMmaps}\end {equation} with $f_i (p) = p_i$ such that $R_i \rightarrow \infty$ and the following two conditions hold:
\begin {enumerate}
\item $f_i ^*g_i\rightarrow g $ in the $C ^\infty $-topology, where $g_i,g $ are the Riemannian metrics on $M_i,M $, and
\item $f_i ^ {- 1} (\overline{C_i}) \rightarrow C$ in the Chabauty topology on closed subsets of $M $.
\end {enumerate}
Note that the metrics $f_i^* g_i $ are only partially defined, but they are defined on larger and larger subsets of $M $ as $i\to \infty$.  So as $C ^\infty $-convergence of metrics is checked on compact sets, the convergence in 1) still makes sense. 
As in \S \ref{smoothsec}, we  call $(f_i)$ a sequence of \emph {almost isometric maps}  witnessing the convergence $(M_i,p_i,C_i) \to (M,p,C)$. Also, when the $R_i$ do not matter, we  will  again write 
$$f_i : M \dashmap M_i$$
to  indicate that the maps $f_i$ are partially defined, but that their  domains of definition exhaust $M$. (This  notation will  be mostly used in the body of the paper, not  in this appendix.)

 We now show how to construct a quasi-metric that induces the smooth--Chabauty topology.  Here, a
\emph {quasi-metric} on a set $X$ is a nonnegative, symmetric function $d: X \times X \longrightarrow \BR$ that vanishes exactly on the diagonal and for some $K\geq 1$ satisfies the  quasi-triangle inequality
$$d(x_1,x_3) \leq K(d(x_1,x_2)+d(x_2,x_3) ), \  \ \forall x_1,x_2,x_3 \in X.$$
 Examples of quasi-metrics include powers $d=\rho^\beta$ of metrics $\rho$, and a theorem of Frink, c.f.\ \cite{Aimarmacias},  implies that for \emph{every} quasi-metric $d$,  there is an honest  metric $\rho$ on $X$  such that $\frac 1K \leq d/\rho^\beta \leq K$  for some $\beta\geq 1$ and $K >0$.    The added flexibility in the quasi-triangle inequality makes it much easier to construct quasi-metrics than metrics, and yet Frink's theorem shows that essentially, one can do as much with the former as with the latter. 

\vspace{2mm}

Given  points $X_i=(M_i,{p_i},{C_i})\in \mathcal {CM}^d$,  $i=1,2$, define $$d_{R,k}(X_1,X_2)=\min \{ 1, \inf \{ \log \lambda + \epsilon \}\},$$  where the infimum is taken over all $\lambda,\epsilon$  such that
there is a smooth embedding $f: B_{M_1} ({p_1}, R/\lambda)\longrightarrow M_2$ with $f ({p_1}) = p_2 $ such that 
\begin {enumerate}
\item $D  ^ k f: T ^ k B_{M_1} ({p_1}, R/\lambda)\longrightarrow T ^ k M_2  $ is locally $\lambda $-bilipschitz on the subset $Z^k_{1/\lambda}(B_{M_1}(p_1,R/\lambda)\subset T^kM_1$ with respect to the iterated Sasaki metric\footnote{Here, recall that $Z^k_r(M)$ is the $r$-neighborhood of the zero section in $T^k(M)$.},
\item $\hat d_{R/\lambda}(C_1,f^ {- 1} (C_2)) \leq \epsilon$, where $\hat d_{R/\lambda}$ is    as in Proposition \ref{2ndChabauty}.
\end {enumerate}
Note that if $\lambda$ and $\lambda'$ are   at least $e=2.718\ldots$, the $1$  realizes the minimum defining $d_{R,k}$. So, everywhere below, we will always assume $\lambda,\lambda'< e.$

 The functions $d_{R,k}$ are not symmetric, so later on we will  symmetrize them.  However, they are already `quasi-symmetric':

\begin {lemma}\label {quasisym}
If $R>0$, $k\in \BN$ and $X_i=(M_i,p_i,C_i) \in \mathcal {CM}^d$ for $i =1,2$, 
$$d_{R, k}(X_2,X_1) \leq  e \,d_{R, k}(X_1,X_2).$$

\end {lemma}
\begin {proof}
	Suppose $d_{R, k}(X_1,X_2) < \log \lambda + \epsilon$,  where  the sum  is that in the definition of $d_{R,k}$, and choose
$f: B_{M_1} ({p_1}, R/\lambda)\longrightarrow M_2$
 as above realizing  this inequality.  Applying Lemma \ref{invert}, the inverse map is defined on the domain $$f^{-1} : B_{M_2}(p_2,R/\lambda ^ 2) \longrightarrow M_1,$$
and of course is locally $\lambda $-bilipschitz.  Note that the iterated derivative of $f^{-1}$ is $(D^kf)^{-1}$, which is locally $\lambda$-bilipschitz on $D^kf(Z)$, where $$Z= Z^k_{1/\lambda}(B_{M_1}(p_1,R/\lambda)\subset T^kM_1.$$ And since $D^kf$ is locally $\lambda$-bilipschitz, we have 
$$D^kf(Z) \supset Z^k_{1/\lambda^2}(B_{M_2}(p_2,R/\lambda^2),$$
so $D^kf^{-1}$ is locally $\lambda$-bilipschitz
We have
\begin {align*}
	\hat d_{R/\lambda^2}(C_2, \left(f^{-1}\right)^ {- 1} (C_1)) &=\hat d_{R/\lambda^2}(C_2, f(C_1)) \\
 & \leq \lambda \, \hat d_{R/\lambda}(f^{-1}(C_2), C_1)\\
& \leq \lambda \epsilon
\end {align*}
 where the first inequality uses Corollary \ref{distortioncor}.   The lemma follows as $\lambda < e.$
\end {proof}

 We  now prove a quasi-triangle inequality:

\begin {lemma}\label {reallyweaktriangle}
If $R\geq e^2, k \in \BN$ and  $X_i=(M_i,p_i,C_i) \in \mathcal {CM}^d$ for $i =1,2,3$,
$$d_{R, k}(X_1,X_3) \leq e(d_{R, k}(X_1,X_2) +  d_{R, k}(X_2,X_3)). $$
\end{lemma}

\begin {proof}
Suppose $d_{R, k}(X_1,X_2) < \log \lambda + \epsilon$ and $d_{R, k}(X_2,X_3) < \log \mu  +\delta$,  where these sums are those in the definition of $d_{R,k}$, and choose
$$f: B_{M_1} ({p_1}, R/\lambda)\longrightarrow M_2, \ \ g: B_{M_2} ({p_2}, R/\mu)\longrightarrow M_3$$
 as above realizing these inequalities.  \emph {} The composition
$$g \circ f : B_{M_1} ({p_1}, R/(\lambda\mu))\longrightarrow M_3$$
 is defined, since $f $ is itself $\lambda $-lipschitz, by  Corollary \ref {bilipschitzderivatives}.   The  iterated total derivative $D  ^ k (g \circ f): T ^ k B_{M_1} ({p_1}, R/ (\lambda\mu))\longrightarrow T ^ k M_3  $  is also locally $\lambda\mu$-bilipschitz on the appropriate domain, so we just need to deal with condition 2).  But 
\begin {align}
& \ \ \ \hat d_{R/(\lambda\mu)}\big(C_1,(g \circ f)^ {- 1} (C_3)\big)\nonumber \\&\leq
 \hat d_{R/(\lambda\mu)}\big(C_1,f^{-1}(C_2)\big) + \hat d_{R/(\lambda\mu)}\big(f^{-1}(C_2),(g \circ f)^ {- 1} (C_3)\big) \nonumber \\
& \leq  2 \, \hat d_{R/\lambda}\big(C_1,f^{-1}(C_2)\big)  +
\lambda \, \hat d_{R/\mu}\big(C_2,g^ {- 1}  (C_3)) \label {52}\\
& \leq  2 \epsilon  +
\lambda \delta \nonumber\\
& \leq e \left (\epsilon + \delta \right). \nonumber
\end {align}
 Here, the first term of \eqref{52}   comes from the comment after  the statement of Proposition \ref{2ndChabauty},  and the second term comes from Corollary \ref {distortioncor}.
 This finishes the proof, since  then $d_{R, k}(X_1,X_3) \leq \log (\lambda \mu) + e\left (\epsilon + \delta \right) \leq e( \log \lambda +  \epsilon + \log \mu + \delta ) \leq e(d_{R, k}(X_1,X_2) +  d_{R, k}(X_2,X_3))$. \end {proof}

  One now proceeds as in the proof of Theoreom \ref{manifoldmetrizable} to construct a quasi-metric $D$ on $\mathcal {CM}^d$  that induces the  smooth-Chabauty  topology:  
$$D(X,Y) = \sum_{k=9}^\infty \frac 1{2^{-k}} \left (d_{k,k}(X,Y) + d_{k,k}(Y,X)\right), $$
where $9>e^2$ is chosen because of Lemma \ref{reallyweaktriangle}.   Note that the quasi-symmetry lemma  (Lemma \ref{quasisym}) implies that  symmetrizing does not change the topology  induced by $d_{R,k}$. Now as mentioned above, Frink's  theorem, c.f.\ \cite{Aimarmacias},  implies that there is a metric  $\rho$   on $\mathcal{CM}^d$  with $$\frac 1K \leq D/\rho^\beta \leq K, \ \ \text {for some } \beta\geq 1,K >0.$$

 This allows us to prove:

\begin {theorem}\label {cmsepmet}
$\mathcal {CM}^ d$ is  a Polish space.
\end {theorem}
\begin {proof}
	The definition of a Cauchy sequence  extends verbatim to quasi-metrics, and  we claim that the quasi-metric $D$  is complete.  If $X_i=(M_i,p_i,C_i)$ is a $D$-Cauchy sequence,  the sequence of pointed manifolds $(M_i, p_i)$ is Cauchy for the  complete metric,  also called $D$,  introduced in the proof of Theorem \ref{manifoldmetrizable}.  Hence,  we can assume $(M_i,p_i) \to (M,p)$  in the smooth topology.  

 Fix  sequences $R_n \to \infty, \lambda_n \to 1$ and $k_n \to \infty$.  Then for each $n $,  we have that for all $i\geq I_n$  there are maps $$f_{n,i} : B_M(p,R_n) \longrightarrow M_i$$  with $f_{n,i} ({p_1}) = p_2 $  such that $D^{k_n}f_{n,i}$ is locally  $\lambda_n $-bilipschitz on  an appropriate neighborhood of the zero section in $T^kM$. Because the Chabauty topology is compact, we can pass to a subsequence in $n$  such that 
\begin {equation}
	\overline {f_{n,I_n}^{-1} C_{I_n}} \to C\in \mathcal C (M),\label {1stconvergence}
\end {equation}
 and passing to a further subsequence\footnote{We are performing a bit of sleight-of-hand here in order to tame the  proliferation of indices.  Namely, we are passing to a subsequence in $n$, but then also replacing the $R_n$  and $k_n$ with sequences that goes to infinity more slowly.  The convergence in \eqref{1stconvergence} alone  would not be enough to conclude  the first part of \eqref{2ndconvergence}, otherwise,  for instance.}
we may assume that
\begin {equation}
	\hat d_{R_n}(f_{n,I_n}^{-1} C_{I_n},C) \leq \frac 1n, \text{ and } d_{k_n,R_n}(X_i,X_j) \leq \frac 1n \text{ when } i,j \geq I_n, \label {2ndconvergence}
\end {equation}
 where $d_{R,k}$ is  as in  Proposition \ref{reallyweaktriangle} and the second part of \eqref{2ndconvergence} uses that $(X_i)$  is $D$-Cauchy.  Now set $X=(M,p,C)$. Then for each $n$ and $i \geq I_n $,  we have
\begin {align*}d_{k_n,R_n}(X,X_i)   &\leq d_{k_n,R_n}(X,X_{I_n}) + d_{k_n,R_n}(X_{I_n},X_i) \\ 
 &\leq ( \log \lambda_n + 1/n) + \frac 1n.
\end{align*}
 This converges to zero with $n$,  so $X_i \to X$ in $\mathcal {CM}^d$.  In other words, $D$  is a complete quasi-metric. But Cauchy sequences for $D$  are the same as Cauchy sequences for $\rho$,  so  this means that $\rho$ is a complete \emph{metric} on $\mathcal{CM}^d$. 

 Separability of $\mathcal {CM}^ d$   follows from separability of $\mathcal {M}^ d$:  we can choose an element from a countable dense subset of  $\mathcal {CM}^ d$  by first choosing a pointed manifold $(M,p)$ from  a countable dense subset of $\mathcal {M}^ d$ and  then choosing a finite subset  $C \subset M$  that lies within a fixed countable dense subset of $M$.
\end {proof}

There are a number of variants of $\mathcal {CM}^d$: one could substitute pointed manifolds with vectored or framed manifolds, or require the distinguished closed subset to lie in either the unit tangent bundle or the frame bundle.  (See the space $\mathcal{P} ^ d_{all}$ introduced in \S \ref {pdsec}, for instance, which is the space of pointed manifolds with distinguished closed subsets of the frame bundle.)  The techniques above apply just as easily to all these situations, so we will feel free to use their metrizability without comment in the text.

\subsection {A stability result for Riemannian foliated spaces}
\label {stabilitysec}
This section is devoted to the following stability result for leaves of Riemannian foliated spaces.  It is a slight strengthening of Theorem 4.3 (see also Theorem 4.1) from Lessa~\cite{Lessabrownian}, which requires that $X $ is compact, and implies a number of classical results on foliations, such as the Reeb local stability theorem, see \cite{Lessabrownian}.

\begin {theorem}\label {foliationcompactness}
Suppose $X $ is a $d$-dimensional Riemannian foliated space in which $x_i \to x$ is a convergent sequence of points. Then $(L_{x_i},x_i)$ is pre-compact in $\mathcal M ^ d $, and every accumulation point is a pointed Riemannian  cover of $(L_x,x)$. 

Moreover, if $(M,p)$ is such an accumulation point, it is covered by the holonomy cover $(\tilde L_x^{hol},\tilde x)$ of $(L_x,x)$:  that is, there is are covering maps $$(\tilde L_x^{hol}, \tilde x) \longrightarrow (M,p) \longrightarrow (L_x,x),$$ where the composition of the two is the holonomy covering.
\end {theorem}

 Recall that the \emph{holonomy cover} of a pointed leaf $(L,x)$  is the pointed cover corresponding to the subgroup of $\pi_1(L,x)$ consisting of all loops on $L$ with trivial holonomy.  The last part of Theorem \ref{foliationcompactness} will never be used in this paper, but we think it is worth including in the theorem statement for future reference.

 Before giving the proof, we note the following:

\begin {lemma}\label {compactsubset}
Let $X $ be a Riemannian foliated space and let $R >0 $.    
Suppose that $x_i \rightarrow x\in X$ and for each $i $, we have a point $y_i \in L_{x_i}$ with
$$\limsup_{i\rightarrow\infty} \, d_{L_{x_i}} (x_i, y_i)\leq R . $$Then there is a subsequence of $(y_i) $ that converges to a point $y\in L_x $ with 
$$d_{L_{x}} (x, y)\leq R . $$
\end {lemma}
\begin {proof}
There is some $\epsilon >0$ with the following property: for every $z $ in the ball $ B(x,R) \subset L_x $, there is a foliated chart $\phi : \BR ^d \times  T \longrightarrow X $ with $\phi (0,s) =z $ and
$$d_t \big (\, [- 1, 1]^d ,\, \BR^d \setminus [-2,2]^d\,  \big)\geq \epsilon \ \ \forall t\in T, $$
where $d_t $ is the Riemannian metric on $\BR ^ d $ induced from that on $\phi(\BR ^ d \times\{t\})$.  This is a simple consequence of the relative compactness of $B (x , R) $.

If $R <\epsilon$, then pick such a foliated chart with $\phi (0,s)=x$.  For large $i $, we have $x_i\in \phi([-1,1] ^ d\times T)$ and $d_{L_{x_i}} (x_i, y_i)<  \epsilon$, so it follows that
$$y_i =\phi(a_i, t_i), \text { where } a_i\in[-2,2]^d,\ t_i\rightarrow s,$$ and extracting a convergent subsequence of $(a_i)$ proves the claim.

Most likely $R\geq\epsilon$, though.  In this case, construct a new sequence by choosing $z_i \in L_{x_i} $ along a geodesic from $x_i $ to $y_i $ so that $$d_{L_{x_i}} (x_i, z_i) = \epsilon /2, \ \ \  \limsup_{i\rightarrow\infty} d_{L_{x_i}} (z_i,y_i) \leq R -\epsilon /2.$$  Passing to a subsequence, $(z_i)$ converges, so we may replace $(x_i)$ with $(z_i)$, thus substituting $R $ with $R -\epsilon/2 $.  The proof concludes by induction.
\end {proof}

We now prove the theorem.

\begin {proof}[Proof of Theorem \ref {foliationcompactness}]
Given $R >0 $ and $k\in \mathbb N$, the $R$-balls around the base points $x_i$ within the leaves $L_{x_i} $ have uniformly $C ^ k $-bounded geometry, in the sense of Definition \ref {boundedgeometry} in the appendix.  For by compactness, this is true of the $R $-ball around $x \in L_x$, and for large $i$ the coordinate charts of Definition \ref {boundedgeometry} can be transferred to $L_{x_i}$ with arbitrarily small distortion via the vertical projections in local flow boxes, see Lemma \ref{compactsubset}.  Compare also with Lemma 4.34 of Lessa \cite{Lessabrownian}, in which similar arguments are used.  It follows from Theorem \ref {2ndcompactness} that $( L_{x_i}, x_i)$ is pre-compact in $\mathcal M ^ d $.  

Suppose now that $x_i \to x$ in $X$ and $( L_{x_i}, x_i) \rightarrow (M,p)$ smoothly.  From the smooth convergence, it follows that for every $R>0$ we have maps $$f_i : B(p,R) \longrightarrow L_{x_i} \subset X, \ \ f(p)=x_i,$$ that are locally bilipschitz with distortion constants converging to $1$.  We claim:

\begin {claim}\label {Arzela}
After passing to a subsequence, the maps $f_i $ converge to a local isometry 
$f : B (p, R) \longrightarrow L_x$ with $f (p)=x$.
\end {claim}

\begin{proof}The first step is to construct a metric on $X$ with respect to which the maps $f_i $ are uniformly lipschitz, and the second is to show that the images of the $f_i $ are contained in some compact subset of $X $, so that Arzela-Ascoli applies.

In \cite [Lemma 4.33]{Lessabrownian}, Lessa shows that any \emph {compact} Riemannian foliated space $X $ admits a metric $d$ that is \emph {adapted} to the leafwise Riemannian structure: i.e.\ when $x,y$ lie on the same leaf, their leafwise distance is at least $d(x,y)$.  The idea is to first construct pseudo-metrics on $X $ that vanish outside of a given foliated chart $\BR ^ d\times T \longrightarrow X$, by combining the leafwise metrics $d_t$ with the distance $d_T$ in $T $.  
Then, one covers $X$ with a finite number of such charts, and sums the resulting pseudo-metrics to give an adapted metric on $X $.  

In our situation, $X $ may not be (even locally) compact, so this method fails to produce an adapted metric.  However, as $B (x,R) \subset L_x$ is relatively compact, the same argument does give a pseudo metric $d$ on $X $ such that
\begin {enumerate}
\item if $x, y  $ lie on the same leaf, their leafwise distance is at least $d (x, y) $,
\item there is a neighborhood $U \subset X$ of the ball $B (x, R)\subset L_x $ such that $d$ restricts to a metric on the closure of $U $.
\end {enumerate}
As the maps $f_i : B(p,R) \longrightarrow L_{x_i} \subset X$ are locally bilipschitz with distortion constants converging to $1$, they are uniformly lipschitz with respect to the adapted pseudo-metric $d$. Lemma \ref {compactsubset} implies that $f_i (B (p, R))\subset U$ for large $i $ and that $$K=\overline {B(x,R)} \ \cup \ \bigcup_i \overline {f_i(B(p,R))} $$
is compact, so the intersection of $K \cap \overline U $ is a compact metric space that contains $f_i (B (p, R)) $ for large $i $.  By Arzela-Ascoli's theorem, $f_i $ converges after passing to a subsequence.  The limit is a local isometry 
$f : B (p, R) \longrightarrow L_x $
with $f (p) = x $, proving Claim \ref {Arzela}, and therefore Theorem \ref {foliationcompactness}.
\end{proof}

 Using Claim \ref {Arzela}, a diagonal argument now gives a local isometry
$$f : M \longrightarrow L_x, \ \ f(p)=x$$
defined on all of $M $.  As $M $ and $L_x $ are both complete, connected Riemannian manifolds, this $f $ is a Riemannian covering map. The fact that $\tilde L_x^{hol} \longrightarrow L_x$  factors through $M$ is exactly the same as in \cite{Lessabrownian}, see the last 3 paragraphs of the proof of Theorem 4.3.
\end {proof}

\textit{\textrm{
\bibliographystyle{amsplain}
\bibliography{bibrefs}
}}

\end{document}